\input ./style/arxiv-general.cfg
\documentclass[seceqn,aop,MSNbibl,dvips]{arximspdf}
\makeatletter
   \@ifpackageloaded{graphicx}{}{\usepackage{graphicx}}
\makeatother

\doi{10.1214/15-AOP1004}
\volume{44}
\issue{2}
\pubyear{2016}
\firstpage{1426}
\lastpage{1487}
\docsubty{FLA}

\makeatletter
\def\sklfrac#1#2{(#1/#2)}
\def\eqref#1{(\ref{#1})}
\newcommand{\rrvert}{\vert}
\newcommand{\rrVert}{\Vert}
\newcommand{\llvert}{\vert}
\newcommand{\llVert}{\Vert}
\newtheorem{theorem}{Theorem}[section]
\newtheorem{lemma}[theorem]{Lemma}
\newtheorem{proposition}[theorem]{Proposition}
\newproclaim{rem}[theorem]{Remark}

\newcommand{\bTm}{\mathbf{T}_{\mathrm{mix}}}

\makeatother

\begin{document}
\begin{frontmatter}

\title{Mixing time and cutoff for the adjacent transposition shuffle
and the simple exclusion}
\runtitle{AT shuffle and exclusion on the segment}

\begin{aug}
\author[A]{\fnms{Hubert}~\snm{Lacoin}\corref{}\ead[label=e1]{lacoin@impa.br}}
\runauthor{H. Lacoin}
\affiliation{IMPA---Instituto Nacional de Matem\'atica Pura e Aplicada}
\address[A]{IMPA---Instituto Nacional\\\quad  de Matem\`atica Pura e Aplicada\\Estrada Dona Castorina 110\\ Rio de Janeiro 22460-320\\ Brasil
\\\printead{e1}}
\end{aug}

%
\received{\smonth{1} \syear{2014}}
%
\revised{\smonth{1} \syear{2015}}

%
\begin{abstract}
In this paper, we investigate the mixing time of the adjacent
transposition shuffle for a deck of $N$ cards.
We prove that around time $N^2\log N/(2\pi^2)$, the total variation
distance to equilibrium of the deck
distribution drops abruptly from $1$ to $0$, and that
the separation distance has a similar behavior but with a transition
occurring at time $(N^2\log N)/\pi^2$.
This solves a conjecture formulated by David Wilson.
We present also similar results for the exclusion process on a segment
of length $N$ with $k$ particles.
\end{abstract}

%
\begin{keyword}[class=AMS]
\kwd{37L60} \kwd{82C20} \kwd{60J10}
\end{keyword}
\begin{keyword}
\kwd{Markov chains} \kwd{mixing time} \kwd{shuffle} \kwd{particle
systems} \kwd{cutoff}
\end{keyword}
\end{frontmatter}

\tableofcontents
\section{Introduction}\label{sec1}

\subsection{A brief history of card shuffling}

Let us consider the following way of shuffling a deck of $N$ cards:
at each step, with probability $1/2$ we interchange the position of a
pair of adjacent cards chosen uniformly at
random (among the $N-1$ possible choices), and with probability $1/2$
we do nothing.
How many steps do we need to perform until the deck has been shuffled?

Even though this shuffling method may be of very little practical use
for card players (indeed the usual
rifle-shuffles allow a much faster mixing of the deck if executed
properly; see \cite{cfBD}),
this question has raised a considerable interest in the domain of
Markov chains for a number of years, since
Aldous \cite{cfAld}, Section~4, proved that $O(N^3 \log N)$ steps were
sufficient to mix the deck and that $\Omega(N^3)$ steps were necessary.
This appears in \cite{cfLPW}, Chapter~23, in a short list of open
problem concerning Markov chains mixing times.

The first reason that can be given for this interest is that it is that
allowing only local moves (i.e., adjacent transpositions) adds a
constraint which makes the problem more challenging than the usual
transposition shuffle; see \cite{cfDS}
for a computation of the mixing time by algebraic methods, \cite
{cfMat} for a simpler probabilistic proof and \cite{cfkcyc} for
a recent paper on the subject with additional results on the evolution
of the cycle structure of the permutation.

The second reason is that shuffling with a geometrical constraint is a
reasonable toy-model to describe the relaxation of a low density gas.
Consider $N$ (labeled) particles in a box with erratic moves and local
interactions. We can now ask ourselves a difficult question:
how much time is needed for the system to forget all the information
about its initial configuration? Of course the adjacent transposition
is an over-simplification of the problem because it is one dimensional,
and the only motion that particles (or cards)
can make is by exchanging their position with
a neighbor, but a solution to the toy problem might give an idea of the
qualitative behavior of the system.
This connection with particle systems becomes more obvious when the
simple exclusion process (which corresponds to the case of unlabeled particles)
is introduced in the next section.

The last substantial progress toward a solution prior to the writing of
this paper was by Wilson \cite{cfWilson}, who proved that
$\frac{1}{\pi^2} N^3 \log N$ steps where necessary and that $\frac
{2}{\pi^2} N^3 \log N$ where sufficient, and conjectured that the
first was the correct answer.\vspace*{2pt}
In this paper we solve this conjecture by showing that the pack is
mixed after $\frac{1}{\pi^2} N^3 \log N(1+o(1))$ steps.

For notational convenience all our results are proved for the
continuous time version
of the Markov chain and the mixing time presented in the theorems
differs by a factor $2N$.
We show how to prove the result in discrete time is the Appendix~\ref{appB}.

\subsection{The exclusion process}

A significant part of the paper is devoted to the study of the mixing
of the exclusion process,
which is a projection of the adjacent transposition shuffle.
The simplest way to describe it is the following: consider a segment
with $N$ sites,
and place $k\in\{1,\dots,N-1\}$
particle on this segment, with \textit{at most} one particle per site.

We consider the following dynamics: each particle jumps independently
with a rate equal to the number of empty sites in its neighborhood, the
site on which it jumps being chosen uniformly at random between these
sites (equivalently it jumps with rate one on each of the empty
neighbors; see Figure~\ref{partisep}
and the next section for a more normal description).
We want to know how long we must wait to come close to the equilibrium
state of the particle system, for which all configurations are equally likely.

This model too has a long history and can be considered in a more
general setup, with an $N\times N$ grid instead of a segment (or a
higher dimensional
cube, or a more general graph), we refer to \cite{cfLiggett}, Section VIII, for a classical introduction.
The problem of computing the mixing time of the exclusion process has
also been well developed in the
case of the complete graph $\mathbb{Z}^d$, grid, torus and of general graphs; see \cite{cfMor,cfLacleb,cfRO} and references therein.

\section{Models and results}

\subsection{The AT shuffle and the total variation cutoff}
Let us now introduce card shuffles in a mathematical framework.
The adjacent transposition shuffle (or AT shuffle) is a continuous time
Markov chain on the symmetric group $S_N$.
We consider that we have a deck of $N$ cards that are labeled from $1$
to $N$.
We number the positions of the cards from top to bottom saying that the
top card has position $1$
and the bottom one $N$. To an array of cards, we associate a
permutation $\sigma$ saying that
$\sigma(x)=y$ if the $x$th position in the pack is occupied by the
card labeled $y$.
Our chain selects a card uniformly at random among those in position
$1$ to $N-1$
and exchanges its position with the one that is immediately below it.

More formally, we let $(\tau_x)_{1\leq x\leq N-1}$ denote the nearest
neighbor transpositions $(x,x+1)$ (note that
the set $\{ \tau_x \vert 1\leq x\leq N-1 \}$ is a generator $S_N$ in
the group-theoretical sense). The generator $\mathcal L$ of the
AT shuffle is defined by its action on the functions of $\mathbb{R}
^\Omega$ as follows:
%
%
\begin{equation}
\label{defgen} (\mathcal L f) (\sigma):=\sum_{x=1}^{N-1}
f( \sigma\circ\tau_x)-f(\sigma).
\end{equation}

Let $(\sigma_t)_{t\geq0}$ denote trajectory of the Markov chain with
initial condition
$\sigma_0=\mathbf{1}$ (the identity) and $P_t$ denote the law of
distribution of the time marginal $\sigma_t$.
Given a probability distribution $\nu$, we define $P_t^\nu$ to be the
marginal distribution of $\sigma^\nu_t$,
the Markov chain starting with initial distribution $\nu$.

This is a simple example of dynamics where geometry plays a role (as
opposed to mean field models):
a given card can only interact with its neighbors.

We write $\mu$ for the uniform measure on $S_N$
(we do not underline the dependence in $N$ in the notation when there
is no risk of confusion).
As the transpositions $(\tau_x)_{x=1}^{N-1}$ generate the group $S_N$,
this Markov chain is irreducible,
and $\mu$ is the unique invariant probability measure.
Hence, for $N$ fixed, when $t$ tends to infinity $P^{\nu}_t$ converges
to $\mu$ for any initial probability distribution,
and for this reason we refer to $\mu$ as the equilibrium measure.

We want to study properties of the relaxation to equilibrium of the
Markov chain or in other words the way in which $P_t$ converges to $\mu
$ when
$t\to\infty$, for large values of $N$. We investigate the asymptotic
behavior the \textit{total variation distance} to equilibrium which is
perhaps the most natural metric for probability measures.

If $\alpha$ and $\beta$ are two probability measures on a common
space $\Omega$, it is defined by
%
%
\begin{equation}
\llVert\alpha-\beta\rrVert_{\mathrm{TV}}:=\frac{1}{2}\sum
_{\omega\in\Omega} \bigl\llvert\alpha(\omega)-\beta(\omega)
\bigr\rrvert
= \sum_{\omega\in\Omega} \bigl(\alpha(\omega)-\beta(\omega)
\bigr)_+,
\end{equation}
where $x_+=\max(x,0)$ is the positive part of $x$.
An equivalent definition is
%
%
\begin{equation}
\llVert\alpha-\beta\rrVert_{\mathrm{TV}}=\max_{A\subset\Omega}
\alpha(A)-\beta(A).
\end{equation}
We will also sometimes use the following alternative characterization
of the distance:
we say that $\pi$ is a coupling of $\alpha$ and $\beta$ if $\pi$ is a
probability law on $\Omega\times\Omega$ for which the projected laws on
the first and second marginal
are respectively $\alpha$ and~$\beta$.

%
\begin{lemma}[{(\cite{cfLPW}, Proposition~4.7)}] \label{caravar}
We have
%
%
\begin{equation}
\label{caracoupl}\llVert\alpha-\beta\rrVert_{\mathrm{TV}}:=\min
\bigl\{ \pi(
\omega_1\ne\omega_2) \vert\pi\mbox{ is a coupling of $
\alpha$ and $\beta$} \bigr\}.
\end{equation}
\end{lemma}

We define the distance to equilibrium of the Markov chain
%
%
\begin{equation}
\label{dat} d^N(t):=\llVert P_t-\mu\rrVert
_{\mathrm{TV}}.
\end{equation}

By symmetry of $S_N$, the distance to equilibrium does not depend on
the initial condition. The reader can further check that
\[
d^N(t)=\max_{\{ \nu\mathrm{\ probability\ on\ } S_N\} } \bigl
\llVert
P^{\nu}_t-\mu\bigr\rrVert_{\mathrm{TV}}.
\]

For a given $\varepsilon\in(0,1)$, we define the $\varepsilon
$-mixing-time to be
the time needed for the system to be at distance $\varepsilon$ from equilibrium
%
%
\begin{equation}
T_{\mathrm{mix}}^N(\varepsilon):=\inf\bigl\{t\geq0 \vert
d^N(t)\leq\varepsilon\bigr\}.
\end{equation}

Our first result states that for the first order asymptotics of
$T_{\mathrm{mix}}
^N(\varepsilon)$ for $N$ large does not depend on $\varepsilon$,
meaning that on a
certain time scale,
the distance to equilibrium drops abruptly from $1$ to $0$ in a very
short time.
This phenomenon has been conjectured or proved for a few types of
dynamics and has been called cutoff; this expression was coined in the
seminal paper \cite{cfDS}; see also \cite{cfLPW}, Chapter~18, for
more on this notion.
We further identify the exact location of the cutoff.

%
\begin{theorem}\label{mainres}
For the adjacent transposition shuffle we have for every \mbox
{$\varepsilon
\in(0,1)$},
%
%
\begin{equation}
\label{cutoff} \lim_{N\to\infty} \frac{2\pi^2T_{\mathrm
{mix}}^N(\varepsilon
)}{N^2\log N}=1.
\end{equation}
\end{theorem}

The mixing time for the AT shuffle has been the object of investigation
since Aldous \cite{cfAld}, Section~4, proved that one had to wait a
time at least of order
$N^2$
(more precisely of order $N^3$ steps in the discrete setup he
considered; see the \hyperref[sec1]{Introduction}) to reach equilibrium.
The last significant progress was made by Wilson in \cite{cfWilson},
where path coupling techniques developed in \cite{cfBubdy} were used to
prove that the mixing time was of order $N^2\log N$.

He proved that for any given $\varepsilon$,
\[
\frac{1}{2\pi^2}{N^2\log N} \bigl(1+o(1) \bigr)\leq
T_{\mathrm
{mix}}^N(\varepsilon)\leq\frac
{1}{\pi^2}{N^2
\log N} \bigl(1+o(1) \bigr),
\]
and predicted that the lower bound was sharp.
Our result brings this prediction to a rigorous ground and answers the
original questions of Aldous \cite{cfAld}.

\subsection{The separation cutoff}

Total variation is not the only kind of distance in which one might be
interested. Another commonly used distance
in the study of convergence to equilibrium is \textit{the separation
distance} (which is not a metric), defined by
\[
d_S(\alpha,\beta):= \max_{x\in\Omega} \biggl( 1-
\frac{\alpha
(x)}{\beta
(x)} \biggr).
\]

Another notion of distance to equilibrium can be derived from this
distance. We define
\[
d^N_S(t):=d_S(P_t,\mu)=\max
_{\{ \nu\mathrm{\ probability\ on\ } S_N\}
}d_S \bigl(P^\nu_t,
\mu\bigr).
\]

For $\varepsilon$ we define the separation mixing time as
%
%
\begin{equation}
T_{\mathrm{sep}}^N(\varepsilon):=\inf\bigl\{t\geq0 \vert
d^N_S(t)\leq\varepsilon\bigr\}.
\end{equation}

We prove that cutoff also occurs for the separation distance, but at a
time twice as large.

%
\begin{theorem}\label{sepdist}
For the adjacent transposition shuffle we have for every \mbox
{$\varepsilon
\in(0,1)$},
%
%
\begin{equation}
\label{cutoff2} \lim_{N\to\infty} \frac{\pi^2T_{\mathrm
{sep}}^N(\varepsilon
)}{N^2\log N}=1.
\end{equation}
\end{theorem}

This result solves another conjecture by Wilson (see \cite{cfWilson},
Table~1) and improves both the best previous lower bound and
upper bound by a factor $2$.

\subsection{The simple exclusion process}

The exclusion process is the simplest lattice model for particles with
hardcore interaction.
Consider the segment $[0,N]$ as being divided in $N$ intervals of unit size.
We identify the interval $[x-1,x]$, with $x\in\{1,\dots,N\}$, and
call each interval a site.
Each of these sites has two possible states: either it is empty or it
contains a particle.

When considering the exclusion process with $k$ particles,
the state space is defined by
%
%
\begin{equation}
\Omega_{N,k}= \Biggl\{ \gamma\in\{0,1\}^{N} \Big\vert\sum
_{x=1}^N \gamma(x)= k \Biggr\}.
\end{equation}

The simple exclusion process on the segment $[0,N]$ is a the
continuous-time Markov chain on $\Omega_{N,k}$ where each of the $k$ particles
jump to the left and to the right neighboring site with rate one
whenever these sites are empty.
An equivalent (but maybe less physical) description of the process is
to say that the content of each pair of neighboring sites gets
exchanged with rate one.
To be more formal, note that $S_N$ naturally acts on $\Omega_{N,k}$. For
$\sigma\in S_N$, $\gamma\in\Omega_{N,k}$, one can define
%
%
\begin{equation}
\label{actionee} \sigma\cdot\gamma(x):=\gamma\bigl(\sigma(x)
\bigr).
\end{equation}
The generator of the simple exclusion on the segment can be written as follows:
%
%
\begin{equation}
\label{crading} (\mathcal L f) (\gamma):=\sum_{x=1}^{N-1}
f(\tau_x\cdot\gamma)-f(\gamma),
\end{equation}
where $\tau_x$ denotes the adjacent transposition $(x,x+1)$.
The equilibrium measure of this chain process is the uniform measure on
$\Omega_{N,k}$ that we call $\mu_k$ or $\mu$ when there is no possible
confusion.
We write $(\gamma^\xi_t)_{t\geq0}$ for the Markov chain starting from
$\xi\in\Omega_{N,k}$.
We set also $P_t^\xi$ to be the law of the time marginal $\gamma^\xi_t$.
We define the distance to equilibrium at time $t$, for total variation
distance and separation respectively to be equal to
%
%
\begin{eqnarray}
\label{dissep} %
d^{N,k}(t)&:=& \max_{\xi\in\Omega_{N,k}}
\bigl\llVert P^\xi_t-\mu\bigr\rrVert_{\mathrm{TV}}=
\max_{\{\nu\mathrm{\ probability\ on\ } \Omega_{N,k}\}} \bigl
\llVert P^\nu_t-\mu
\bigr\rrVert_{\mathrm{TV}},
\nonumber
\\[-8pt]
\\[-8pt]
d^{N,k}_S(t)&:=& \max_{\xi\in\Omega_{N,k}}
d_S \bigl(P^\xi_t,\mu\bigr)=\max
_{\{
\nu\mathrm{\ probability\ on\ } \Omega_{N,k}\}} d_S \bigl(P^\nu_t,
\mu\bigr).
\nonumber
\end{eqnarray}
Note that contrary to what happens for the AT shuffle, the distance
$\llVert P^\xi_t-\mu\rrVert _{\mathrm{TV}}$ depends on the initial
condition $\xi$ as there
is no symmetry.
The respective mixing times are defined by
%
%
\begin{eqnarray}
T_{\mathrm{mix}}^{N,k}(\varepsilon)&:=&\inf\bigl\{t\geq0
\vert d^{N,k}(t)\leq\varepsilon\bigr\},
\nonumber
\\[-8pt]
\\[-8pt]
T_{\mathrm{sep}}^{N,k}(\varepsilon)&:=&\inf\bigl\{t\geq0 \vert
d^{N,k}(t)\leq\varepsilon\bigr\}.
\nonumber
\end{eqnarray}

%
\begin{theorem}\label{mainressep}
For any $\varepsilon>0$, given a sequence $k(N)$ which is such that
both $k$
and $N-k$ tend to infinity, we have
the following asymptotics for the mixing time:
%
%
\begin{equation}
\label{mixsep} \lim_{N\to\infty} \frac{2\pi^2T_{\mathrm
{mix}}^{N,k}(\varepsilon
)}{N^2\log\min(k,N-k)}=1.
\end{equation}
If furthermore we have
%
%
\begin{equation}
\lim_{N\to\infty} \frac{\log\min(k, N-k) }{\log\log N}=\infty,
\end{equation}
then
%
%
\begin{equation}
\lim_{N\to\infty} \frac{\pi^2T_{\mathrm{sep}}^{N,k}(\varepsilon
)}{N^2\log\min(k,N-k)}=1.
\end{equation}
\end{theorem}

In this case also the lower bound for $T_{\mathrm
{mix}}^{N,k}(\varepsilon)$
\[
T_{\mathrm{mix}}^{N,k}(\varepsilon)\geq\frac{1}{2\pi^2}N^2
\log\min(k,N-k) \bigl(1+o(1) \bigr),
\]
corresponds to \cite{cfWilson}, Theorem~4.

%
\begin{rem}
The assumption on $k$ for the separation mixing time is purely
technical, and we do not believe it to be necessary.
As exposed in the next section,
the upper bound
\[
\limsup_{N\to\infty} \frac{\pi^2T_{\mathrm
{sep}}^{N,k}(\varepsilon)}{N^2\log\min
(k,N-k)}\leq1
\]
is a consequence of \eqref{mixsep} and thus is valid whenever both $k$
and $N-k$ tend to infinity.
\end{rem}

\subsection{Connection between exclusion and AT shuffle and between
separation and total variation}

There is a natural projection for the set of permutations onto the set
of particle configurations
%
%
\begin{eqnarray}
\label{gammasigma} S_N &\to&\Omega_{N,k},
\nonumber
\\[-8pt]
\\[-8pt]
\sigma&\mapsto&\gamma_\sigma.
\nonumber
\end{eqnarray}

It gives to the card labeled from $1$ to $k$ the role of particles and
to those labeled from $k+1$ to $N$
the role of empty sites (see Figure~\ref{partisep})
with
%
%
\begin{equation}
\gamma_\sigma(x):= %
\cases{\displaystyle1 &\quad if $\displaystyle
\sigma(x)\leq k$,
\cr
0 &\quad  if $\displaystyle\sigma(x)> k$. } %
\end{equation}

%
\begin{figure}

\includegraphics{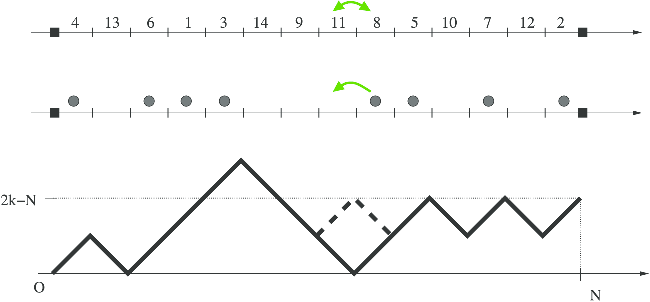}

\caption{On the first line, a permutation with
$N=15$ is represented as with a possible composition by an adjacent
transposition (double arrow).
The second line gives the image of the permutation by the mapping
\protect\eqref{gammasigma} for $k=8$, the adjacent transposition of
the first
line corresponds to a particle jump.
The third line gives the lattice paths version of the particle system:
each particle corresponds to an up step and each empty site to a down
step. When a particle jumps, a local extremum of the path is
``flipped.'' This lattice path correspondence is used in the
construction of $\widetilde\sigma$ [equation \protect\eqref
{tildesigma}] and
$\eta$ [equation \protect\eqref{defeta}].}
\label{partisep}
\end{figure}

With this mapping, the AT shuffle $(\sigma_t)_{t\geq0}$ is mapped on
the exclusion process [this is a simple consequence of \eqref{crading}].
As the total variation distance shrinks with projection, we have
[recall \eqref{dat} and \eqref{dissep}] for all $k\in\{1,\dots,N-1\}$,
%
%
\begin{eqnarray}
d^{N,k}(t)&\leq& d^N(t)\qquad\forall t \geq0,
\nonumber
\\[-8pt]
\\[-8pt]
T_{\mathrm{mix}}^{N,k}(\varepsilon) &\leq& T_{\mathrm
{mix}}^N(
\varepsilon) \qquad\forall\varepsilon\in(0,1).
\nonumber
\end{eqnarray}
Similar inequalities are valid for the separation distance.
For these reasons, the lower bound asymptotics for the mixing time in
Theorems \ref{mainres} and \ref{sepdist} are implied by the lower
bound asymptotics
in Theorem~\ref{mainressep} for $k=N/2$, and the upper bound in
Theorem~\ref{mainressep} for $k=N/2$ is implied by the upper bound in
Theorem~\ref{mainres}.

Furthermore, there exists a general comparison inequality for the total
variation distance and separation distance for reversible Markov chains
(see, for instance, \cite{cfLPW}, Lemma~19.3),
%
%
\begin{equation}
d_S(2t)\leq4d(t).
\end{equation}
This implies
\[
T_{\mathrm{sep}}^N(\varepsilon)\leq2 T_{\mathrm{mix}}^N(
\varepsilon/4) \quad\mbox{and} \quad T_{\mathrm{sep}} ^{N}(
\varepsilon)\leq2 T_{\mathrm{mix}}^N(\varepsilon/4),
\]
the analogous inequality being valid for the exclusion process.
In view of this and of the bounds proved in \cite{cfWilson}, to prove Theorems
\ref{mainres}, \ref{sepdist} and \ref{mainressep} it is sufficient
to prove the following statements:
\begin{itemize}
\item The sharp asymptotic upper bound on the mixing time of the AT shuffle
\[
T_{\mathrm{mix}}^N(\varepsilon)\leq\frac{1}{2\pi^2}N^2
\log N \bigl(1+o(1) \bigr).
\]
\item A sharp asymptotic lower bound on the mixing time for the
separation distance for the
exclusion process
\[
T_{\mathrm{sep}}^{N,k}(\varepsilon)\geq\frac{1}{\pi^2}N^2
\log\min(k,N-k) \bigl(1+o(1) \bigr).
\]
The case $k=N/2$ gives the lower bound for the AT shuffle.
\item A sharp asymptotic upper bound on the mixing time of the
exclusion process
\[
T_{\mathrm{mix}}^{N,k}(\varepsilon)\leq\frac{1}{2\pi^2}N^2
\log\min(k,N-k) \bigl(1+o(1) \bigr).
\]
\end{itemize}

For the sake of completeness, we will also provide a short proof for
the lower bound on the mixing time
of the exclusion process
\[
T_{\mathrm{mix}}^{N,k}(\varepsilon)\geq\frac{1}{2\pi^2}N^2
\log\min(k,N-k) \bigl(1+o(1) \bigr).
\]

\subsection{Open questions}

\subsubsection{The cutoff window}
Our results only identify the main asymptotic term for the mixing time,
and a natural question would be how
to obtain a more complete asymptotic. In particular, one would like to
know on what time scale around $T_{\mathrm{mix}}(1/2)$ the total
variation distance
drops from $1$ to zero [i.e., e.g., the asymptotic behavior of
$T_{\mathrm{mix}}
(3/4)-T_{\mathrm{mix}}(1/4)$].
This time scale is usually referred to as the \textit{cutoff window},
and from heuristics of Wilson \cite{cfWilson}, Section~10, the
natural conjecture would be that
it is of order $N^2$.

With some tedious effort, an upper bound on the cutoff window could be
derived from our proof, but there are some serious reasons why we
cannot push this up to the optimal order $N^2$.

Our proofs rely very much on the graph structure which is considered,
that is, the segment $\{1,\dots,N\}$, and in particular on the fact
that it is totally ordered.
Hence a natural challenge is to try to generalize the method for the
$\sqrt{N}\times\sqrt{N}$ grid (or higher dimensional ones)
for which most of the monotonicity tool cannot be used, or at least,
not in the manner it is used in the present paper. In fact, even the
case of the circle $\mathbb{Z} / N\mathbb{Z} $
is a challenging one.

%
\begin{rem} Since the competition of this work, we have developed an
alternative approach to tackle the problem of the mixing time for the
exclusion process on the circle \cite{cfLac}. While the method is
slightly more robust and, in particular, does not depend on
monotonicity consideration, it does not permit us to treat the case of
the adjacent transposition shuffle. On the positive side, it gives a
sharp result on the cutoff window [which is shown to be indeed $O(N^2)$].
\end{rem}

\subsection{Organization of the paper}

A key ingredient in the proof of all our results is the use of mononicity:
we introduce a natural order on our state space which is preserved by
the dynamics,
and then use order-preservation to get extra information about the
convergence to equilibrium.

Hence an important part of the paper, Section~\ref{monotool}, is
dedicated in introducing the order,
and various properties of order preservation on the symmetric group.
In Section~\ref{nomonotool}, we introduce further important technical
tools: we show how our processes are related to the heat equation
and exhibit a weaker upper bound on the mixing time, which is used in
the proof as an input.
These two preliminary sections are absolutely crucial to understanding
the rest of the paper,
though the proof of the results presented
in them might be skipped on a first reading. Some of the more technical
proofs of these sections are postponed to Appendix~\ref{app}.

In Section~\ref{proofmres} we prove an upper bound for the mixing time
of the AT shuffle (which together with the lower bound of \cite
{cfWilson} implies
Theorem~\ref{mainres}). In Section~\ref{lbexp} we prove the lower
bound result on the separation mixing time and total variation mixing
time for the exclusion process,
from which we deduce Theorem~\ref{sepdist} and half of Theorem~\ref
{mainressep}. In Section~\ref{excluproc}, we prove an upper bound for
the mixing time of the exclusion process for an arbitrary number of
particles to complete the proof of Theorem~\ref{mainressep}.

\subsection{Notation}

Let us introduce some notation that we will repeatedly use in the paper.

We use $:=$ to define new quantities (and in a few cases, $=:$ when the
quantity which is defined is on the right-hand side).

If $\nu$ is a probability distribution on $S_N$ (or $\Omega_{N,k}$) and
$\sigma\in S_N$, we write
$\nu(\sigma)$ for $\nu(\{\sigma\})$.

We write $\nu(f)$ or $\nu(f(\sigma))$
for the expected value of $f(\sigma)$,
\[
\nu(f):=\sum_{\sigma\in S_N} f(\sigma)\nu(\sigma).
\]

Expectations are denoted by $\mathbb{E} $ when the probability is
denoted by
$\mathbb{P} $.

We write $\frac{\nu}{\mu}$ for the probability density
\[
\sigma\mapsto\frac{\nu(\sigma)}{\mu(\sigma)}.
\]

Finally, we say that an event or rather a family of events
$(A_N)_{N\geq
0}$ holds with high probability (and write w.h.p.)
if
\[
\lim_{N\to\infty} \mathbb{P} (A_N)=0.
\]

\section{A tool box to take advantage of monotonicity}\label{monotool}

Putting an order on the set of permutations might seem a strange idea
at first glance because of the complete symmetry of $S_N$.
What we do to break that symmetry is we choose to give a special role
to the identity which we fix to be the maximal element.
Then the idea is to say that $\sigma$ is larger than $\sigma'$ if it
is ``closer to the identity'' in a certain sense.

However, in order to give a simple definition of our order on $S_N$, we
must first introduce a mapping that transforms permutations into
discrete surfaces.

\subsection{Mapping permutations onto discrete surfaces}

The following mapping is inspired by \cite{cfWilson}, Figure~3.
We associate with each $\sigma\in S_N$ a function $\widetilde\sigma
\dvtx
\{
0,\dots,N\}^2\to\mathbb{R} $,
defined as follows:
%
%
\begin{equation}
\label{tildesigma} \widetilde\sigma(x,y):=\sum_{z=1}^x
\mathbf{1}_{\{\sigma(z)\leq y\}
}-\frac{xy}{N}.
\end{equation}
The term $xy/N$ is subtracted so that $\widetilde\sigma(x,y)$ has zero
mean under the equilibrium measure.
The map is injective. Indeed,
\[
\widetilde\sigma(x,y)- \widetilde\sigma(x,y-1)-\widetilde\sigma(x-1,y)+
\widetilde\sigma(x-1,y-1)+\frac{1}{N}=\mathbf{1}_{\{\sigma(x)= y\}}.
\]
We identify the image set $\{\widetilde\sigma\vert \sigma\in S_N\}$
with $S_N$ as it brings no confusion. This mapping induces a natural
(partial) order relation on $S_N$ defined by
\[
\sigma\leq\sigma'\quad\Leftrightarrow\quad\forall x,y,\qquad
\widetilde\sigma(x,y)\geq\widetilde\sigma'(x,y).
\]

The identity (which we denote by $\mathbf{1}$) is the maximal element of
$(S_N,\geq)$, and the permutation $\sigma_{\min}$
defined by
%
%
\begin{equation}
\forall x \in\{1, \dots, N\} ,\qquad\sigma_{\min}(x)=N+1-x
\end{equation}
is the minimal one.

\subsection{The graphical construction}\label{graphix}

We present now a construction of the dynamics which allows us to
construct all the trajectories $\sigma^\xi_t$ starting from
all initial conditions $\xi\in S_N$ simultaneously (a \textit{grand
coupling} and has the property of conserving the order).

We associate with each $x\in\{1,\dots,N-1\}$ an independent Poisson
processes $(\mathcal{T} ^x)= (\mathcal{T} ^x_n)_{n\geq0}$ which has
intensity two.
In other words $\mathcal{T} ^x_0=0$ for every $x$ and
\[
\bigl(\mathcal{T} ^x_n-\mathcal{T} ^x_{n-1}
\bigr)_{x\in\{1,\dots
,N-1\}, n\geq1}
\]
is a field of i.i.d. exponential variables with mean $1/2$.
We refer to $\mathcal{T} =(\mathcal{T} ^x)_{1\leq x\leq N-1}$ as {\sl
the clock process}.
Note that the set of values taken by the clock processes is almost
surely a discrete subset of $\mathbb{R} $.

Let $(U^x_n)_{x\in\{1,\dots,N-1\}, n\geq1}$, be a field of i.i.d.
Bernoulli random variables
($U^x_n\in\{ 0,1\}$) with parameter one half, which is independent of
$\mathcal{T} $.

Now given $\mathcal{T} $ and $U$, we construct, in a deterministic fashion
$(\sigma^\xi_t)_{t\geq0}$, the trajectory of the Markov
chain starting from $\xi\in S_N$.
The trajectory $(\sigma^\xi_t)_{t\geq0}$ is c\`adl\`ag and is
constant on the intervals where the clock process is silent.

When a clock rings, that is, at time $t=\mathcal{T} ^x_n$ $(n\geq
1)$, $\sigma
^\xi_t$ is constructed by updating $\sigma^\xi_{t^-}$ as follows:
\begin{itemize}
\item if either $U^x_n= 1$ and $\sigma_{t^-}(x+1)\leq\sigma
_{t^-}(x)$, or $U^x_N=0$ and $\sigma_{t^-}(x+1)\geq\break\sigma_{t^-}(x+1)$,
we exchange the values of $\sigma_{t^-}(x)$ and $\sigma_{t^-}(x+1)$;
\item in the other cases, we do nothing.
\end{itemize}
In other words, when the clock process associated to $x$ rings, we sort
the cards in position $x$ and $x+1$ if $U^n_x=1$, and we
reverse sort them if $U^x_i=0$. It is straightforward to check that
this construction gives a Markov chain with generator $\mathcal L$
described in \eqref{defgen}.

The effect of the update on $\widetilde\sigma$ is the following: for each
$y\in\{1,\dots, N-1\}$,
if $(\widetilde\sigma_{t^-}(z,y))_{z\in\{1,\dots,N-1\}}$ presents a
local minimum at $z=x$ and $U^x_n= 1$, then it is turned into a local maximum
[$\widetilde\sigma_{t}(x,y)=\widetilde\sigma_{t^-}(x,y)+1$].
On the contrary if it has a local minimum at $z=x$ and $U^x_n=0$, then
$\widetilde\sigma_{t}(x,y)=\widetilde\sigma_{t^-}(x,y)-1$. We call this
operation an update of $\sigma$ at
coordinate $x$.

The fact that the order is conserved by this construction is not a new
result (see, for instance, \cite{cfWilson}), but we choose to include
a short proof here for the sake of completeness.

%
\begin{proposition}\label{orderpreserv}
Let $\xi\geq\xi'$ be two elements of $S_N$. With the graphical
construction above, we have
%
%
\begin{equation}
\sigma^\xi_t\geq\sigma^{\xi'}_t.
\end{equation}
\end{proposition}

\begin{pf}
The only thing to check is that the order is conserved each time a the
clock process rings; that is, for every $(n,x)$ and $t=\mathcal{T} ^x_n$,
\[
\sigma^\xi_{t^-}\geq\sigma^{\xi'}_{t^-}
\quad\Rightarrow\quad\sigma^\xi_t\geq
\sigma^{\xi'}_t.
\]
The right-hand side in the above relation is satisfied if we have
\[
\forall y\in\{1,\dots,N-1\},\qquad\widetilde\sigma^\xi_t(x,y)
\geq\sigma^{\xi'}_t(x,y)
\]
because the other coordinates are not changed at time $t$.

Let us fix $y$.
Note that when $\widetilde\sigma^\xi_{t^-}(x,y)> \widetilde\sigma
^{\xi
'}_{t^-}(x,y)$, there is nothing to prove because it is not possible for
$\widetilde\sigma^\xi$ to jump down while $\widetilde\sigma^{\xi
'}$ jumps
up. For this reason, we might assume that
\[
\widetilde\sigma^\xi_{t^-}(x,y)=\widetilde
\sigma^{\xi'}_{t^-}(x,y).
\]

If $U^x_n= 1$, we just have to check that if $\widetilde\sigma^{\xi
'}_{t}(x,y)$ jumps up, so does $\widetilde\sigma^{\xi}_{t}(x,y)$. This
is easy because if
$\widetilde\sigma^{\xi'}_{t^-}(\cdot,y)$ presents a local minimum at
$x$, then so does $\widetilde\sigma^{\xi}_{t^-}(\cdot,y)$, which is
situated above.

If $U^x_n= 0$, for the same reasons, if $\widetilde\sigma^{\xi
}_{t}(x,y)$ jumps down so does $\widetilde\sigma^{\xi}_{t}(x,y)$,
and we
are done.
\end{pf}

\subsection{Stochastic ordering and its preservation}

Let us recall in this section the definition of stochastic dominance
for probability measures.

Let $\alpha$ and $\beta$ be two probability measures on a finite
ordered set $\Omega$.
We say that $\alpha$ stochastically dominates $\beta$ and write
$\alpha
\succeq\beta$ if
one can find a coupling $\pi$, that is, a probability on $\Omega
\times
\Omega$ such that the first marginal has law $\alpha$ and the second
$\beta$, which satisfies
\[
\omega_1\geq\omega_2, \qquad\pi\mbox{ almost surely.}
\]

We say that a function $f$ on $\Omega$ is increasing if
\[
\forall\omega,\omega'\in\Omega, \qquad\omega\geq
\omega' \quad\Rightarrow\quad f(\omega)\geq f \bigl(
\omega' \bigr).
\]
For an ordered set $\Omega$, we say that a subset $A$ is increasing if
the function $\mathbf{1}_A$ is increasing or equivalently if
%
%
\begin{equation}
\forall\omega\in A,\qquad\omega'\geq\omega\quad\Rightarrow\quad
\omega\in A.
\end{equation}

Recall the notation $\alpha(f)$ for the expectation of $f(\omega
)$ with respect to $\alpha$.
The Kantorovic duality lemma (see, e.g., \cite{cfvillani}, Theorem~5.10, item (i)) provides the following equivalent characterization of
stochastic domination:

%
\begin{lemma}\label{transport}
Consider $\alpha$ and $\beta$ two probability measures on a finite
ordered set $\Omega$. The following statements are equivalent:
\begin{itemize}
\item$\alpha$ dominates $\beta$;
\item for all increasing functions $f$ defined on $\Omega$,
\[
\alpha(f)\geq\beta(f).
\]
\end{itemize}
\end{lemma}

A consequence of Proposition~\ref{orderpreserv} is that if $\nu$ and
$\nu'$ are two probability measures on $S_N$,
then
%
%
\begin{equation}
\nu\succeq\nu' \quad\Rightarrow\quad\forall t\geq0,\qquad
P^{\nu
}_t \succeq P^{\nu'}_t.
\end{equation}

Let us now mention a simple tool to produce stochastic couplings.

%
\begin{lemma}\label{limitlema}
Let $\Omega$ be a finite set and $(\omega^1_t)_{t\geq0}$ and
$(\omega
^2_t)_{t \geq0}$ be two stochastic\vspace*{2pt} processes on $\Omega$.
Assume that the distribution of $\omega^1_t$ and $\omega^2_t$
respectively converge toward two probability measures $\alpha$ and
$\beta$
when $t$ tends to infinity.

If one can find a coupling of the processes such that almost surely
\[
\forall t\geq0,\qquad\omega^1_t\geq
\omega^2_t,
\]
then
\[
\alpha\succeq\beta.
\]
\end{lemma}

\begin{pf}
Let $\pi_t$ be the law of $(\omega^1_t, \omega^2_t)$ under the
coupling given by the assumption of the lemma.
For all $t\geq0$, $\pi_t$ is supported by
\[
\mathcal D= \bigl\{ \bigl(\omega^1,\omega^2 \bigr)\in
\Omega_2 \vert\omega^1\geq\omega^2 \bigr\}.
\]

As $\pi_t$ lives on a compact space (for the topology induced by the
total variation distance),
it has a least one limit point which we call $\pi$ and is supported on
$\mathcal{D} $. The measure $\pi$ provides a coupling proving $\alpha
\succeq
\beta$.
\end{pf}

\subsection{Correlation inequalities and the FKG inequality}

The preservation of monotonicity by the dynamics will be used in
various ways over the course of our proof.
One of the important tools we will use are the correlation
inequalities, which roughly means that conditioning $\mu$ on an
increasing event makes all the other increasing events more likely.
First let us recall a classical result for probability laws on $\mathbb
{R} $.

%
\begin{lemma}\label{correlineq}
Let $f$ and $g$ be two increasing real functions of a real variable and
$X$ be a real random variable of law $P$.
We have
%
%
\begin{equation}
\label{correl} E \bigl[f(X)g(X) \bigr]\geq E \bigl[f(X) \bigr]E
\bigl[g(X)
\bigr].
\end{equation}
\end{lemma}

\begin{pf}
Consider $X'$ an independent copy of $X$, and expand the inequality
$ E[(f(X)-f(X'))(g(X)-g(X'))]\geq0$.
\end{pf}

Inequality \eqref{correl} is not true in general for all the notions
of partial order,
but a generalization of it exists for ``distributive lattices,'' the so
called Fortuin--Kasteleyn--Ginibre or FKG inequality, introduced and
proved in \cite{cfFKG}.

Unfortunately, $S_N$ is not a distributive lattice. More precisely,
if one defines for $\sigma$ and $\sigma'$ in $S_N$, $\operatorname
{\mathbf{min}}(\widetilde
\sigma, \widetilde\sigma')$ and $\operatorname{\mathbf
{max}}(\widetilde\sigma, \widetilde\sigma
')$ by
%
%
\begin{eqnarray}
\label{defveeperm} %
\operatorname{\mathbf{min}} \bigl( \widetilde\sigma,
\widetilde\sigma' \bigr) (x,y)&:=&\min\bigl( \widetilde\sigma
(x,y), \widetilde\sigma'(x,y) \bigr),
\nonumber
\\[-8pt]
\\[-8pt]
\operatorname{\mathbf{max}} \bigl(\widetilde\sigma, \widetilde
\sigma
' \bigr) (x,y)&:=&\max\bigl(\widetilde\sigma(x,y), \widetilde
\sigma'(x,y) \bigr),
\nonumber
\end{eqnarray}
then $\operatorname{\mathbf{min}}(\widetilde\sigma, \widetilde
\sigma')$ and $\operatorname{\mathbf{max}}(\widetilde
\sigma, \widetilde\sigma')$ are not necessarily images of elements
in $S_N$.
However, the proof of \cite{cfHolley} can be adapted to our case.

%
\begin{proposition}[(The FKG inequality for permutations)] \label{FKGbis}
For any pair of increasing functions $f$ and $g$ defined on $S_N$,
%
%
\begin{equation}
\label{permut1} \mu\bigl(f(\sigma)g(\sigma) \bigr)\geq\mu\bigl
(f(\sigma) \bigr)
\mu\bigl(g(\sigma) \bigr).
\end{equation}
\end{proposition}

The proof is postponed to Section~\ref{afkg}.

\subsection{The censoring inequality}\label{censorsec}

The censoring inequality in a result established by Peres and Winkler
\cite{cfPW}, Theorem~1.1, for ``monotone systems'' is a notion which is
a slight generalization of Glauber dynamics for spin systems with
totally a ordered spin space.

What the inequality says is that canceling some of the spins updates
has the effect of delaying the mixing.
Unfortunately, the AT shuffle is NOT a monotone system in the
Peres/Winkler sense. However, we can adapt the proof of the result to
our setup.
Before stating the result, we introduce some terminology and notation.
A {\sl censoring scheme} is a c\`adl\`ag function
\[
\mathcal{C} \dvtx\mathbb{R} ^+\to\mathcal P \bigl(\{1,\dots,N-1\}
\bigr),
\]
where $\mathcal P(\Omega)$ is the set of subsets of $\Omega$.

The censored dynamics with scheme $\mathcal{C}$ is the dynamics
obtained from
the graphical construction of Section~\ref{graphix},
except that if $\mathcal{T} ^x$ rings at time $t$, the update is
performed if
and only if $x\in\mathcal{C}(t)$.

It is quite natural to think that each time a clock rings, it brings
$\sigma_t$ ``closer to equilibrium'' and hence that
censoring will only make convergence to the equilibrium slower. The
censoring inequality establishes that this is true if one starts
from a measure whose density is an increasing function.

Given censoring scheme $\mathcal C$ and $\nu$ a probability
distribution on $S_N$, let $P_t^{\nu,\mathcal C}$ denote the distribution
of $\sigma_t$, which has performed the censored dynamics up to time
$t$ starting with initial distribution $\nu$.
We say that a probability law $\nu$ on $S_N$ is increasing if $\sigma
\mapsto\nu(\sigma)$ is an increasing function of $\sigma$.

%
\begin{proposition}[{(From \cite{cfPW}, Theorem~1.1)}]\label{censor}
If $\nu$ is increasing, then for all $t\geq0$,
%
%
\begin{equation}
\bigl\llVert P_t^{\nu,\mathcal C}-\mu\bigr\rrVert\geq\bigl
\llVert
P_t^{\nu}-\mu\bigr\rrVert.
\end{equation}
\end{proposition}

The proof is postponed to Section~\ref{acensor}

The censoring inequality has been used in a variety of contexts to
bound the mixing times of Markov chains. The strategy is usually to
cook up a
censoring scheme which allows one to have better control over where the
dynamics goes without slowing it down to much. We refer to the
introduction of \cite{cfPW}
for numerous applications of this tool.

\subsection{Projection and monotonicity}

In our proof we sometimes have to work with projections of $\widetilde
\sigma$ on one or a few coordinates.
In this section we show that if $\nu$ is an increasing probability
measure on $S_N$, then its
projections have increasing densities with respect to the projections
of the equilibrium measure.

For $i\in\{0,\dots, K\}$, we set
%
%
\begin{equation}
\label{defxi} x_i:=\lceil i N/K\rceil.
\end{equation}
We define $\widehat\sigma$, the semi-skeleton of $\sigma\in S_N$ defined
on $\{0,\dots,N\}\times\{0,\dots,K\}$, by
%
%
\begin{equation}
\label{semiskel} \widehat\sigma(x,j):=\widetilde\sigma(x,x_j).
\end{equation}
We call $\widehat S_N$ the set of admissible semi-skeletons (the image of
$S_N$ by this transformation).
We define the skeleton $\bar\sigma\in\mathbb{R} ^{\{0,\dots,K\}
^2}$ of a
permutation $\sigma\in S_N$ to be
%
%
\begin{equation}
\label{barsigma} \bigl(\bar\sigma(i,j) \bigr)_{0\leq i,j\leq K}:=
\bigl(\widetilde
\sigma(x_i,x_j) \bigr)_{0\leq
i,j\leq K}.
\end{equation}
We call
\[
\bar S_N :=\{ \bar\sigma\vert\sigma\in S_N \}
\]
the set of admissible skeletons. We equip $\bar S_N$ with the natural order
\[
\bar\sigma\geq\bar\sigma' \quad\Leftrightarrow\quad\bigl(
\forall i,j \in\{0,\dots, K\}, \bar\sigma(i,j)\geq\bar\sigma'(i,j)
\bigr),
\]
and do the same for $\widehat S_N$.
Given $\nu$, a probability measure on $S_N$, we write $\bar\nu$ for
the image measure on $\bar S_N$
of $\nu$ by the skeleton projection and $\widehat\nu$ for the image
measure of the semi-skeleton. We write $\bar\nu_{i,j}$ for
the image measure of $\nu$ by the projection
$\sigma\mapsto\bar\sigma(i,j)$. In particular $\bar\mu$ and $\bar
\mu_{i,j}$ denote the projections
of the equilibrium measure.

%
\begin{rem}
For $N=52$ and $K=2$, the semi-skeleton encodes the positions of the
red cards in the decks, while the skeleton (which is one dimensional)
indicates the number of red cards in the first half of the pack.
Note that while $(\widehat\sigma_t)_{t\geq0}$ is a Markov chain,
$(\bar
\sigma_t)_{t\geq0}$ is not.
\end{rem}

%
\begin{proposition}[(Preservation of monotonicity by projection)]\label
{projecmon}
\begin{itemize}[(iii)]
\item[(i)]
Consider $\bar\sigma^1, \bar\sigma^2\in\bar S_N$.
If $\bar\sigma^1\geq\bar\sigma^2$, then
%
%
\begin{equation}
\label{nastro} \mu\bigl(\cdot\vert\bar\sigma=\bar\sigma^1
\bigr)
\succeq\mu\bigl(\cdot\vert\bar\sigma=\bar\sigma^2 \bigr).
\end{equation}

\item[(ii)] Given $(i,j)\in\{0,\dots, K\}^2$ and $z_1\leq z_2$, two
admissible values for $\bar\sigma(i,j)$, we have
%
%
\begin{equation}
\label{dominus} \mu\bigl(\cdot\vert\bar\sigma(i,j)=z_1 \bigr)
\succeq\mu\bigl(\cdot\vert\bar\sigma(i,j)=z_2 \bigr).
\end{equation}

\item[(iii)] If $\nu$ an increasing probability measure on $S_N$,
then the density
$\bar\nu/\bar\mu$ is an increasing function on $\bar S_N$.

\item[(iv)] If $\nu$ an increasing probability measure on $S_N$, then
$\bar\nu_{i,j}/\bar\mu_{i,j}$ is an
increasing function on the set of admissible value for $\bar\sigma(i,j)$.
\end{itemize}
\end{proposition}

The proof is postponed to Section~\ref{aprojecmon}.

\section{Some additional tools}\label{nomonotool}

In this section we present a connection between the evolution of
$\widetilde\sigma$ and the heat equation, which is an essential
ingredient of the proof,
some nonoptimal estimates on the mixing time, which will use as an
input in the proof, and a technical result to decompose the total
variation distance.

\subsection{Connection with the heat equation}

If one follows the motion of one card only, we see a nearest neighbor
symmetric random walk on the set
$\{1,\dots,N\}$. This indicates a connection between the AT shuffle
and diffusions. We also find this connection when looking at the
evolution of
the mean $\widetilde\sigma_t(x,y)$.

As observed during the graphical construction,
the height $\widetilde\sigma_t(x,y)$ can only jump down when
$\widetilde
\sigma_t(\cdot,y)$ presents a local maximum at $x$, and up
when it presents a local minimum. In each case, this happens with rate one.
When computing the expected drift of $\widetilde\sigma_t(x,y)$, this gives
%
%
\begin{eqnarray}
\partial_t \mathbb{E} \bigl[\widetilde
\sigma_t(x,y) (t) \bigr]&=& \mathbb{E} [\mathbf{1}_{\{\widetilde
\sigma_t(x,y)> \max(\sigma
_t(x-1,y),\widetilde
\sigma_t(x+1,y))\}}
\nonumber
\\
&&\hphantom{\mathbb{E} [}{}- \mathbf{1}_{\{\widetilde\sigma
_t(x,y)< \min
(\sigma
_t(x-1,y),\widetilde\sigma_t(x+1,y))\}}]
\\
&=&\mathbb{E} \bigl[\widetilde\sigma_t(x-1,y)+\widetilde\sigma
_t(x+1,y)-2\widetilde\sigma_t(x,y) \bigr],
\nonumber
\end{eqnarray}
where the last equality follows from the definition of $\widetilde
\sigma$.
Hence the function $f$ defined by
%
%
\begin{equation}
\cases{ \{0,\dots,N\}^2\times\mathbb{R} _+ \to
\mathbb{R},
\cr
(x,y,t)\mapsto\mathbb{E} \bigl[\widetilde\sigma_t(x,y)
\bigr] } %
\end{equation}
is the solution of the one-dimensional discrete heat equation
%
%
\begin{equation}
\label{disheat} %
\cases{ \partial_t f=
\Delta_x f \qquad  \mbox{on }\{1,\dots,N-1\}\times\mathbb{R} _+,
\cr
f(0,t)=f(N,t)=0,
\cr
f(x,y,0)= \widetilde\sigma_0(x,y), } %
\end{equation}
where $\Delta_x$ denotes the discrete Laplacian acting on the $x$ coordinate
\[
\Delta_x f(x,y,t)= f (x+1,y,t)+f(x-1,y,t)-2f(x,y,t).
\]

%
\begin{lemma}\label{boudk}
For all $\sigma_0\in S_N$ and $t\geq0$ we have
%
%
\begin{equation}
\label{group} \max_{x\in\{0,\dots,N\}}\mathbb{E} \bigl[
\widetilde\sigma
_t(x,y) \bigr] \leq4\min(y,N-y) e^{-\lambda_N t },
\end{equation}
where
\[
\lambda_N:=2 \biggl(1-\cos\biggl(\frac{\pi}{N} \biggr)
\biggr)=\frac
{\pi^2}{N^2} \bigl(1+o(1) \bigr).
\]
In particular,
%
%
\begin{equation}
\label{group2} \max_{(x,y)\in\{0,\dots,N\}^2}\mathbb{E} \bigl[
\widetilde\sigma
_t(x,y) \bigr] \leq2N e^{-\lambda_N t }.
\end{equation}

For $\sigma_0=\mathbf{1}$ we have
%
%
\begin{equation}
\label{group3} \mathbb{E} \bigl[ \widetilde\sigma_t(x,y) \bigr
]\geq
\frac{\min
(y,N-y)}{\pi
}\sin\biggl( \frac{\pi x}{N} \biggr)e^{-\lambda_N t} .
\end{equation}
\end{lemma}

The proof is postponed to Section~\ref{aboudk}.

\subsection{Wilson's upper bound on the mixing time}

Several times, we will use Wilson's upper bound as an input in our proof.
The result as it is cited is contained the proof of \cite{cfWilson},
Theorem~10.
For more details, see the proof of Proposition~\ref
{wilson2}.

%
\begin{proposition}\label{wilson}
For all $N$ sufficiently large, for all $\varepsilon>0$
%
%
\begin{equation}
d^N(t)\leq10N \exp(-t \lambda_N),
\end{equation}
where
\[
\lambda_N:=2 \bigl(1-\cos(\pi/N) \bigr).
\]
\end{proposition}

\subsection{Erasing the labels and decomposing the mixing
procedure}\label{labelerase}

Let us suppose for one moment that we change the labels assigned to the
cards in the following manner:
each card whose label previously belonged to $\{x_{i-1}+1,\dots,x_{i}\}
$, $i=1, \dots,K$ receives the label ${\bf i}$ (for $K=4$ and $N=52$,
we can think of this as differentiating only clubs, spades, hearts and
diamonds instead of looking at each individual card).
The pack of cards with the new labels is then described by the
semi-skeleton $\widehat\sigma$ described in \eqref{semiskel}.

It is quite intuitive that for $\sigma_t$ to reach equilibrium we need:
\begin{itemize}[(ii)]
\item[(i)] the semi-skeleton $\widehat\sigma_t$ to be close to its
equilibrium distribution;
\item[(ii)] conditionally to each semi-skeleton, we need that the
order of the card with label ${\bf i}$ to be close to uniformly distributed.
\end{itemize}
The aim of this short section is to make this intuitive claim rigorous;
see Lemma~\ref{cross}.

We introduce a transformation of the measures which has the effect of
making the card whose labels belongs to
$\{x_{i-1}+1,\dots,x_{i}\}$ indistinguishable.

Define $\widetilde S_N$ to be the largest subgroup of $S_N$ that leaves
all the sets $\{x_{i-1}+1,\dots,x_{i}\}$ invariant.
It is isomorphic to $\bigotimes_{i=1}^K S_{\Delta x_i}$ (recall that
$\Delta
x_i:=x_i-x_{i-1}$).

Given $\nu$ a probability measure on $S_N$, we define $\widetilde\nu
$ as
%
%
\begin{equation}
\label{tildenu} \widetilde\nu(\sigma)=\frac{1}{\prod_{i=1}^K
(\Delta x_i) !}\sum
_{\widetilde
\sigma\in\widetilde S_N} \nu(\widetilde\sigma\circ\sigma).
\end{equation}

Note that the semi-skeleton of $\sigma$
is left invariant by composition on the right by an element of
$\widetilde
S_N$ (in other words $\widehat S_N$ is in bijection with the set of
right-cosets of the subgroup $\widetilde S_N$).
Hence (recall that $\widehat\nu$ denotes the image law of $\nu$ for the
semi-skeleton projection) we have
%
%
\begin{equation}
\label{crck} \widetilde\nu(\sigma):=\frac{1}{\llvert \widetilde
S_N\rrvert }\widehat\nu(\widehat
\sigma).
\end{equation}
This leads to the following result:

%
\begin{lemma}\label{cross}
For all probability laws $\nu$ on $S_N$ we have
%
%
\begin{equation}
\llVert\widetilde\nu- \mu\rrVert_{\mathrm{TV}}=\llVert\widehat
\nu-\widehat
\mu\rrVert_{\mathrm{TV}},
\end{equation}
and as a consequence,
%
%
\begin{equation}
\llVert\nu-\mu\rrVert_{\mathrm{TV}} \leq\llVert\widehat\nu
-\widehat\mu
\rrVert_{\mathrm{TV}}+\llVert\nu-\widetilde\nu\rrVert_{\mathrm{TV}}.
\end{equation}
\end{lemma}

\begin{pf}
We have
%
%
\begin{equation}
2\llVert\widetilde\nu-\mu\rrVert_{\mathrm{TV}} =\sum
_{\xi\in\widehat S_N}\sum_{\{\sigma\in S_N \vert \widehat
\sigma=\xi\}
} \bigl\llvert
\widetilde\nu(\sigma)-\mu(\sigma) \bigr\rrvert.
\end{equation}
Now from \eqref{crck}, $\widetilde\nu$ is constant on $\{\sigma
\vert
\widehat
\sigma=\xi\}$
and thus
%
%
\begin{eqnarray}
2\llVert\widetilde\nu-\mu\rrVert_{\mathrm{TV}} &=&\sum
_{\xi\in\widehat S_N} \biggl\llvert\sum_{\{\sigma\in S_N \vert
\widehat
\sigma=\xi\}}
\widetilde\nu(\sigma)-\mu(\sigma) \biggr\rrvert
\nonumber
\\
&=&\sum_{\xi\in\widehat S_N} \biggl\llvert\sum
_{\{\sigma\in S_N \vert
\widehat
\sigma=\xi\}} \nu(\sigma)-\mu(\sigma) \biggr\rrvert
\\
&=&\sum_{\xi\in
\widehat
S_N} \bigl\llvert\widehat\nu(\xi)-
\widehat\mu(\xi) \bigr\rrvert= 2\llVert\widehat\nu-\widehat\mu
\rrVert
_{\mathrm{TV}}.
\nonumber
\end{eqnarray}
\upqed
\end{pf}

\section{Proof of Theorem \texorpdfstring{\protect\ref{mainres}}{2.2}: Upper bound for the mixing
time of the AT shufle}\label{proofmres}

\subsection{Strategy}

We are now ready to prove the asymptotics for the mixing time for the
AT shuffle.
As the lower bound is already known (\cite{cfWilson},\vadjust{\goodbreak} Theorem~6; see
also Section~\ref{lbexp} of the present paper), we only need to prove
in this section that
for every $\varepsilon>(0,1)$, $\delta>0$ for all $N$ sufficiently large,
%
%
\begin{equation}
\label{uopbound} d_N \biggl( (1+\delta)\frac{N^2}{2\pi^2} \log N
\biggr)\leq\varepsilon.
\end{equation}

Let us now explain how we plan to prove \eqref{uopbound}.
We run a censored dynamics with the following censoring scheme:
\begin{itemize}[(iii)]
\item[(i)] During a time $(\delta/3)\frac{N^2}{2\pi^2} \log N$ we
cancel the updates occurring at $x_i$,\vspace*{2pt} $i\in\{1,\dots,K-1\}$
with $K$ chosen to be $\lceil1/\delta\rceil$. According to
Proposition~\ref{wilson} this gives enough time to mix the order of
the set of cards whose label belongs to $\{x_{i-1}+1,\dots,x_i\}$.
\item[(ii)] Then, during a time $\frac{N^2}{2\pi^2}(1+\delta/3)
\log N$, we run the dynamics with no\vspace*{2pt} censoring. Using Lemma~\ref
{boudk} and monotonicity,
we prove that after such a time, the distribution of the skeleton $\bar
\sigma_t$ comes close to equilibrium (this is the most delicate part).
\item[(iii)] Finally during a time $(\delta/3)\frac{N^2}{2\pi^2}
\log N$, we censor the updates of the $x_i$s again.\vspace*{2pt} Using
Proposition~\ref{wilson}
and the fact that the skeleton is at equilibrium,
we prove that the dynamics puts the semi-skeleton $\widehat\sigma$ at
equilibrium.
\end{itemize}
After all these steps, the distribution of the semi-skeleton is close
to $\widehat\mu$ and the distribution of the order of the cards whose
label belongs
$\{x_{i-1}+1,\dots, x_i\}$ is close to uniform (for each $i$). Thus,
using Lemma~\ref{cross},
we can conclude that $\sigma_t$ has come close to equilibrium.
The censoring inequality (Proposition~\ref{censor}) guarantees that
$\sigma_t$ is even closer to equilibrium for the noncensored dynamics,
and this implies~\eqref{uopbound}.

\subsection{Decomposition of the proof}

Now let us turn the strategy we have exposed into mathematical statements.
Set
%
%
\begin{eqnarray}
t_1&:=&\frac{N^2}{2\pi^2}(\delta/3) \log N,
\nonumber
\\
t_2&:=&\frac{N^2}{2\pi^2}(1+2\delta/3) \log N,
\\
t_3&:=&\frac{N^2}{2\pi^2}(1+\delta) \log N
\nonumber
\end{eqnarray}
and\vspace*{-2pt}
\[
K:=\lceil1/\delta\rceil.
\]
Recall the definition of $x_i$ \eqref{defxi}, and
consider a dynamic $\sigma_t$ starting from the identity and adhering
to the following censoring scheme:
\begin{itemize}
\item in the time interval $[0,t_1]$, the updates at $x_i$, $i=1,\dots
,K-1$ are canceled;
\item in the time interval $(t_1,t_2]$, there is no censoring;
\item in the time interval $[t_2,t_3]$,
the updates at $x_i$, $i=1,\dots,K-1$ are censored.
\end{itemize}
What the dynamic does after time $t_3$ is irrelevant since we are only
interested in is the distance to equilibrium at time $t_3$.

Let us call $\nu_t=P^{\mathcal{C}}_t$ the distribution of $\sigma_t$ for
this censored dynamics.
As the identity is the maximal element, the initial distribution (i.e.,
a Dirac mass on the identity) is an increasing probability, and thus
from Proposition~\ref{proppres}, $\nu_{t}$ is increasing for all $t$.
This fact is one of the key points in the proof.

We decompose the proof of \eqref{uopbound} in three statements.
First we show that after time $t_1$ the distribution of $\nu_t$ is not
too different from $\widetilde\nu_t$ defined in
Section~\ref{labelerase}.

%
\begin{proposition}\label{groom1}
For any $\delta$ and $\varepsilon>0$, for all $N$ sufficiently large,
we have,
for all $t\geq t_1$,
%
%
\begin{equation}
\label{greluche}\llVert\widetilde\nu_t-\nu_t\rrVert
\leq\varepsilon/3.
\end{equation}
\end{proposition}

Second, we show that at time $t_2$
the law of the skeleton $\bar\sigma_t$ [recall \eqref{barsigma}] is
close to equilibrium.

%
\begin{proposition}\label{groom2}
For any $\delta$ and $\varepsilon>0$, for all $N$ sufficiently large,
%
%
\begin{equation}
\llVert\bar\nu_{t_2}-\bar\mu\rrVert\leq\varepsilon/3.
\end{equation}
\end{proposition}

The above statement is not directly used to prove the theorem, but it
is the starting point for the
proof that at time $t_3$, the semi-skeleton distribution [recall \eqref
{semiskel}] is close to equilibrium.

%
\begin{proposition}\label{groom3}
For any $\delta$ and $\varepsilon>0$, for all $N$ sufficiently large,
%
%
\begin{equation}
\llVert\widehat\nu_{t_3}-\widehat\mu\rrVert\leq2\varepsilon/3.
\end{equation}
\end{proposition}

\begin{pf*}{Proof of Theorem~\ref{mainres} from Propositions~\ref{groom1} and \ref{groom3}}
From Proposition~\ref{censor} and Lemma~\ref{cross}, we have
%
%
\begin{equation}
d_N(t_3):=\llVert P_{t_3}-\mu\rrVert\leq
\llVert\nu_{t_3}-\mu\rrVert\leq\llVert\widehat\nu_{t_3}-
\widehat\mu\rrVert+\llVert\widetilde\nu_{t_3}-\nu_{t_3}
\rrVert.
\end{equation}
When $N$ is large enough, the right-hand side is smaller than $\varepsilon$
according to
Propositions \ref{groom1} and \ref{groom3}.
\end{pf*}

\subsection{Proof of Proposition \texorpdfstring{\protect\ref{groom1}}{5.1}}

Let us first prove \eqref{greluche} at time $t_1$.
Up to time~$t_1$, because of the censoring, the dynamics is just the
product of $K$ independent dynamics on
$S_{\Delta x_{i}}, i\in\{ 1,\dots,K\}$.

Thus for all $t\leq t_1$, we have
$\sigma_t\in\widetilde S_N$ and
\[
\widetilde\nu_t=\widetilde\delta_{\mathbf{1}}
\]
for all $t\leq t_1$ where $\widetilde\delta_{\mathbf{1}}$ is the uniform
probability on $\widetilde S_N$
($ \delta_{\mathbf{1}}$ is the Dirac mass on the identity).

For each $i=1,\dots,K$, let $\nu^i_t$ denote the law of $\sigma_t$
restricted to $\{x_{i-1}+1,\dots, x_i\}$, and set
$\mu^i$ to be the corresponding equilibrium measure (uniform on the
permutation of $\{x_{i-1}+1,\dots, x_i\}$).
Using Proposition~\ref{wilson} for each dynamics on $S_{\Delta x_{i}}$
and the fact that the total variation distance between product measures
is smaller than the sum of the total variation distances
of the marginals, we have
%
%
\begin{eqnarray}
\label{desdree}\llVert\nu_t -\widetilde\delta_{\mathbf{1}}
\rrVert&\leq&\sum_{i=1}^K \bigl\llVert
\nu^i_t -\mu^i \bigr\rrVert\leq\sum
_{i=1}^K 10 \Delta x_ie^{-t \lambda_{\Delta x_i}}
\nonumber
\\[-8pt]
\\[-8pt]
&\leq& K\times10 \biggl(\frac{N}{K}+1 \biggr)\exp\biggl(-2t
\biggl(1-
\cos\biggl(\frac{\pi}{(N/K+1)} \biggr) \biggr) \biggr).
\nonumber
\end{eqnarray}
In the last inequality we used $\Delta x_i\leq N/K+1$.

For $t=t_1$, the right-hand side is smaller than
%
%
\begin{equation}
11 N \exp\bigl(- (10\delta)^{-1} \log N \bigr)\leq\varepsilon/3,
\end{equation}
provided $\delta$ has been chosen small enough and that $N$ is large enough.
Now what is left to show is that $\llVert \nu_t -\widetilde\nu_t
\rrVert $ is
decreasing.
We remark that from the definition \eqref{tildenu}, $\widetilde\nu
_t$ is
simply the law of $\sigma_t$ for the dynamics started with initial
distribution~$\widetilde\delta_{\mathbf{1}}$,
and the result follows from a standard coupling argument.

\subsection{Proof of Proposition \texorpdfstring{\protect\ref{groom2}}{5.2}}

This is, perhaps, the most delicate part of the proof.
In this section we temporarily forget that we have fixed $K=\lceil
\delta^{-1} \rceil$, as the result is valid for any finite $K$. Of
course, here, $N$ sufficiently large means
$N$ larger than something which depends on $K$.

Let us first explain the idea in the case $K=2$
for didactic purposes (say that $N$ is even).
We want to show that starting with distribution $\nu_{t_1}$
after a time\vspace*{2pt} $\frac{N^2}{2\pi^2}(1+\delta/3) \log N$, the height
$\sigma(N/2,N/2)=\bar\sigma(1,1)$ (we write simply $\bar\sigma$ as
it brings no confusion) is close to its equilibrium distribution.
The reader can check that at equilibrium $\bar\sigma\approx(\sqrt
{N}/4)\mathcal N$, where $\mathcal N$ is a standard Gaussian.

Using Lemma~\ref{boudk} we know that at time $t_2$, we have
%
%
\begin{equation}
\label{bim1} \nu_{t_2}(\bar\sigma)\leq2N e^{-\lambda_N
(t_2-t_1)}\leq
N^{1/2-\delta/10}.
\end{equation}
Hence the expected value of $\bar\sigma$ at time $t_2$ is much
smaller than its equilibrium fluctuation. This is, however, not
sufficient to conclude that $\nu_{t_2}$ is close to equilibrium.
The extra ingredient we use is that the density $\bar\nu_{t_2}/ \bar
\mu$ of the distribution of $\bar\sigma$ is increasing: from
Proposition~\ref{proppres},
$\nu_{t_2}$ has increasing density and from Proposition~\ref
{projecmon}; this is also the case for the projection.
Then the following lemma allows us to conclude:

%
\begin{lemma}\label{unitframe}
There exists a constant $C$ such that for any $N$ and for any measure
$\nu$ such that
$\bar\nu/\bar\mu$ is increasing, one has
%
%
\begin{equation}
\llVert\bar\nu-\bar\mu\rrVert_{\mathrm{TV}}\leq\frac{ C\bar
\nu(\bar\sigma) }{N^{1/2}}.
\end{equation}
\end{lemma}

\begin{pf}
Set
\[
\mathcal A:= \bigl\{ x\in\{-N/4,N/4+1,\dots,-N/4\} \vert\bar\nu
(x)\geq\bar
\mu(x) \bigr\},
\]
which is an increasing set by the assumption of $\nu$.

Furthermore, from the definition of the total variation distance, we have
%
%
\begin{equation}
\label{croom} \bar\nu(\mathcal A)- \bar\mu(\mathcal{A})=\llVert
\bar\nu-\bar\mu
\rrVert_{\mathrm{TV}}.
\end{equation}
Now let us prove a lower bound for $\bar\nu(\bar\sigma)$ which is a
function of $\bar\nu(\mathcal A)- \bar\mu(\mathcal{A})$.
First we split the expectation into two contributions by conditioning.
%
%
\begin{equation}
\label{froome} \bar\nu(\bar\sigma)= \bar\nu(\mathcal{A})\bar
\nu( \bar\sigma
\vert\mathcal{A}) + \bar\nu\bigl(\mathcal{A}^c \bigr)\bar\nu
\bigl(
\bar\sigma\vert\mathcal{A}^c \bigr).
\end{equation}
Then using the correlation inequality (Lemma~\ref{correlineq}) for the
two functions\vspace*{2pt} $\bar\sigma\mapsto\bar\sigma$
and $\bar\sigma\mapsto\frac{\bar\nu}{\bar\mu}(\bar\sigma)$
(which is increasing by Proposition~\ref{projecmon}), we have
%
%
\begin{eqnarray}
\bar\nu(\mathcal{A})\bar\nu( \bar\sigma\vert\mathcal
{A})&=&\bar\mu(
\mathcal{A}) \bar\mu\biggl( \frac{\bar\nu}{\bar\mu}(\bar
\sigma)\bar\sigma\Big\vert
\mathcal{A} \biggr)
\nonumber
\\[-8pt]
\\[-8pt]
&\geq&\bar\mu(\mathcal{A}) \bar\mu\biggl( \frac{\bar\nu}{\bar
\mu}(\bar\sigma) \Big\vert
\mathcal{A} \biggr)\bar\mu( \bar\sigma\vert\mathcal{A} ) =\bar
\nu(\mathcal{A})
\bar\mu( \bar\sigma\vert\mathcal{A} ).
\nonumber
\end{eqnarray}
Similarly,
%
%
\begin{equation}
\bar\nu\bigl(\mathcal{A}^c \bigr)\bar\nu\bigl( \bar\sigma
\vert
\mathcal{A}^c \bigr)\geq\bar\nu\bigl(\mathcal{A} ^c
\bigr)\bar\mu\bigl( \bar\sigma\vert\mathcal{A}^c \bigr).
\end{equation}
Plugging these inequalities in the right-hand side of \eqref{froome} and subtracting
\[
0=\bar\mu(\bar\sigma)=\bar\mu(\mathcal{A})\bar\mu( \bar\sigma
\vert
\mathcal{A})+\bar\mu\bigl(\mathcal{A}^c \bigr)\bar\mu\bigl(
\bar
\sigma\vert\mathcal{A}^c \bigr),
\]
we obtain
%
%
\begin{eqnarray}
\label{froome2} \qquad\bar\nu(\bar\sigma) &\geq& \bigl(\bar\nu(
\mathcal{A})-\bar\mu(\mathcal{A}) \bigr)\bar\mu( \bar\sigma
\vert\bar\sigma\geq
x_{\mathcal{A}})+ \bigl(\bar\nu\bigl(\mathcal{A}^c \bigr)-\bar
\mu
\bigl(\mathcal{A}^c \bigr) \bigr)\bar\mu( \bar\sigma\vert\bar
\sigma< x_{\mathcal{A}})
\nonumber
\\[-8pt]
\\[-8pt]
&\geq&\llVert\bar\nu-\bar\mu\rrVert_{\mathrm{TV}} \bigl(\bar
\mu( \bar\sigma
\vert\bar\sigma\geq x_{\mathcal{A}})-\bar\mu( \bar\sigma\vert
\bar\sigma<
x_{\mathcal{A}}) \bigr),
\nonumber
\end{eqnarray}
where the last line is deduced from \eqref{croom}.
Finally we use the fact that from the Gaussian scaling
\[
\bar\mu( \bar\sigma\vert\bar\sigma>0) =-\bar\mu( \bar\sigma
\vert\bar
\sigma<0)\geq c \sqrt{N},
\]
and hence
%
%
\begin{equation}
\bar\nu(\mathcal A) \geq c\sqrt{N}\llVert\bar\nu-\bar\mu\rrVert
_{\mathrm{TV}}.
\end{equation}
\upqed
\end{pf}

When $K\geq3$, the idea is roughly the same, and the hope is that
dealing with finite dimensional marginals
does not bring too many complications.

Set
\[
v(\bar\sigma):= \sum_{i,j=1}^{K-1} \bar
\sigma(i,j)
\]
to be the volume below the graph of the skeleton.
Similar to the proof of Lemma~\ref{unitframe} we want to show that
if $\nu(v(\bar\sigma))$ is small with respect to
its equilibrium fluctuations (which are of order $\sqrt{N}$), and
$\nu$ is increasing,
then $\bar\nu$ and $\bar\mu$ are close to each other.

%
\begin{lemma}\label{areaisal}
Let $\nu$ be a probability measure on $S_N$ whose density with respect
$\mu$ is increasing.
For every $\varepsilon$, there exists $\eta(K,\varepsilon)$ such
that for $N$
sufficiently large, we have
%
%
\begin{equation}
\label{crimoi}\llVert\bar\mu- \bar\nu\rrVert\leq\varepsilon/3,
\end{equation}
whenever
%
%
\begin{equation}
\nu\bigl(v(\bar\sigma) \bigr)\leq\sqrt{N}\eta.
\end{equation}
\end{lemma}

\begin{pf*}{Proof of Proposition~\ref{groom2} from Lemma~\ref{areaisal}}
From
Lemma~\ref{boudk} we know that at time $t_2$, we have
%
%
\begin{equation}
\label{bim} \nu_{t_2} \bigl[ v(\bar\sigma) \bigr]
\leq2N(K-1)^2 e^{-\lambda_N
(t_2-t_1)}\leq\sqrt{N} \eta,
\end{equation}
where the last inequality is valid for any fixed $\eta$ when $N$ is
large enough. As, by Proposition~\ref{proppres}, $\nu_{t_2}$ is
increasing, an thus
Lemma~\ref{areaisal} is sufficient to conclude.
\end{pf*}

Before starting the proof of Lemma~\ref{areaisal} we need to introduce
some notation and two technical results.
Given $A>0$ a positive constant, we define
%
%
\begin{eqnarray}
\quad\qquad\mathcal A_{i,j}&:=& \bigl\{ \sigma\vert
\bar\sigma(i,j) \geq\sqrt{N}A \bigr\},
\nonumber
\\
\mathcal A&:=& \bigcap_{i,j=1}^{K-1} \mathcal
A_{i,j} = \bigl\{ \sigma\vert\forall(i,j)\in\{1,\dots,K-1
\}^2, \bar\sigma(i,j) \geq\sqrt{N}A \bigr\},
\\
\mathcal B&:=& \Biggl(\bigcup_{i,j=1}^{K-1}
\mathcal A_{i,j} \Biggr)^c= \bigl\{ \sigma\vert
\forall(i,j)\in\{1,\dots,K-1\}^2, \bar\sigma(i,j) < \sqrt{N}A
\bigr\}.
\nonumber
\end{eqnarray}

%
\begin{lemma}\label{graincon}
When $N$ tends to infinity,
%
%
\begin{equation}
\label{zij} \frac{\bar\sigma(i,j)}{\sqrt{N}}\Rightarrow Z(i,j),
\end{equation}
where the $Z(i,j)$ is a Gaussian of variance
\[
s^2(i,j):=\frac{i} K \biggl(1-\frac{i} K \biggr)
\frac{j} K \biggl(1-\frac{j}K \biggr)
\]
and of mean 0.

In particular, given $\delta\in(0,1/2)$ sufficiently small, there
exist $A(\delta,K)$ and $\delta'(\delta,K)$ which satisfy (for any $K>0$),
\[
\lim_{\delta\to0} \delta\bigl(\delta',K \bigr)=0,
\]
which are such that
%
%
\begin{eqnarray}
\label{lesoufs} %
\mu(\mathcal A)&\geq&\delta^{(K-1)^2}:=
\delta_1,
\nonumber
\\[-8pt]
\\[-8pt]
\mu(\mathcal B)&\geq&1-(K-1)^2\delta' :=1-
\delta_2.
\nonumber
\end{eqnarray}
\end{lemma}

%
\begin{rem}
It seems that in fact the process
\[
\biggl(\frac{\sigma(\lceil xN,yN\rceil)}{\sqrt{N}} \biggr
)_{x,y\in[0,1]^2}
\]
should converge to a Brownian sheet conditioned to be zero on the
boundary of $[0,1]^2$.
However, even convergence of the finite dimensional marginals seems
tricky to prove, and we do not need this result.
\end{rem}

\begin{pf*}{Proof of Lemma~\ref{graincon}}
A simple way to prove \eqref{zij} is to note that (see \cite{cffeller}, page~146)
\[
\mu\biggl(\bar\sigma(i,j)=k-\frac{x_ix_j}{N} \biggr)=\frac
{\left({{x_i} \atop {k}}\right)
\left({{N-x_i} \atop {x_j-k}}\right)}
{\left({{N} \atop {x_j}}\right)}
\]
and use Stirling's formula to obtain a local central limit theorem.

Now given $\delta< 1/2$, we define $A$ to be such that
\[
\mathbb{P} \bigl[ K^{-1} \bigl(1-K^{-1} \bigr) Z \geq A
\bigr]=\delta/2,
\]
where $Z$ is a standard Gaussian,
and $\delta'$ is such that
\[
\mathbb{P} [ Z/4 \geq A ]=2\delta'.
\]
With this definition it is obvious that when $\delta$ tends to zero,
$\delta'$ does as well.

Then from \eqref{zij} [here it is important to note that the standard
deviation of $Z(i,j)$ is always larger than $K^{-1}(1-K^{-1})$ and
smaller than $1/4$] and our choice of $\delta'$ and $A$, we have that
for all $N$ large enough, for all $(i,j)$,
%
%
\begin{equation}
\delta\leq\mu( \mathcal A_{i,j} ) \leq\delta'.
\end{equation}
Then \eqref{lesoufs} can be deduced from the FKG inequality
(Proposition~\ref{FKGbis}) for the first line and
a standard union bound for the second line.
\end{pf*}

The next lemma is quite intuitive, but the proof is quite technical and
is postponed to Section~\ref{acromostic}.

%
\begin{lemma}\label{cromostic}
We have
%
%
\begin{equation}
\label{ared} \mu(\cdot\vert\mathcal{A}) \succeq\mu\bigl(\cdot
\vert
\mathcal{B} ^c \bigr).
\end{equation}
In particular, if $\nu$ is an increasing probability on $S_N$, we have
%
%
\begin{equation}
\label{letrucmoche} \frac{\nu(\mathcal{A})}{\mu(\mathcal{A})}\geq
\frac{\nu(\mathcal
{B} ^c)}{\mu(\mathcal{B} ^c)}.
\end{equation}
\end{lemma}

\begin{pf*}{Proof of Lemma~\ref{areaisal}}
Let us choose $\delta$ such that (with the notation of Lemma~\ref
{graincon}) $\delta_2\leq\varepsilon/6$.
We will prove two implications and deduce the result from them.
First we show that a lower bound on $\nu(\mathcal{A})$ gives a lower
bound on
$\nu(v(\bar\sigma))$
%
%
\begin{equation}
\label{ninjatwist} \forall\alpha>0,\qquad\nu(\mathcal{A})\geq
(1+\alpha)\mu(
\mathcal{A}) \quad\Rightarrow\quad\nu\bigl(v(\bar\sigma) \bigr
)\geq
\delta_1\alpha A \sqrt{k}.
\end{equation}
Then we show that if $(\nu-\mu)(\mathcal{A})$ is small, then the law
\textit
{of the skeletons} $\bar\mu$ and $\bar\nu$
must be close in total variation distance
%
%
\begin{equation}
\label{xencuts} \nu(\mathcal{A})\leq(1+\alpha)\mu(\mathcal{A})
\quad\Rightarrow
\quad\llVert\bar\nu-\bar\mu\rrVert\leq2\alpha+\delta_2.
\end{equation}
Now \eqref{xencuts} and \eqref{ninjatwist} for $\alpha=\varepsilon
/12$ (or
rather its contrapositive) combined
implies \eqref{crimoi} with $\eta:=\delta_1\alpha A$.

To prove \eqref{ninjatwist}, we first show, similar to \eqref
{froome2}, using the correlation inequality (Lemma~\ref{correlineq})
and the fact that the density $\bar\nu_{i,j}/\bar\mu_{i,j}$ is an
increasing function (Proposition~\ref{projecmon}),
that
%
%
\begin{eqnarray}
\label{sigij} \nu\bigl(\bar\sigma(i,j) \bigr) &\geq&(\nu-\mu)
(\mathcal
A_{i,j})\mu\bigl( \bar\sigma(i,j) \vert\mathcal A_{i,j}
\bigr)
\nonumber
\\
[-8pt]
\\[-8pt]
&&{}+(\nu-\mu) \bigl(\mathcal A_{i,j} ^c \bigr)\mu\bigl(
\bar\sigma(i,j) \vert\mathcal A^c_{i,j} \bigr).
\nonumber
\end{eqnarray}
Then we remark that the second term in the right-hand side of \eqref{sigij} is positive,
and deduce using the definition of $\mathcal A_{i,j}$,
%
%
\begin{equation}
\label{sigismon} \nu\bigl(\bar\sigma(i,j) \bigr)\geq(\nu-\mu)
(\mathcal
A_{i,j})\sqrt{N}A.
\end{equation}
We consider now the increasing function
\[
\theta(\sigma):= \Biggl(\sum_{i,j=1}^{K-1}
\mathbf{1}_{\mathcal
A_{i,j}} \Biggr)-\mathbf{1}_{\mathcal A}.
\]
Using the FKG inequality (Proposition~\ref{FKGbis}) applied to the
functions $\theta$ and \mbox{$(\nu/\mu-1)$}
we obtain
%
%
\begin{equation}
\sum_{i,j=1}^{K-1} (\nu-\mu) (\mathcal
A_{i,j})\geq(\nu-\mu) (\mathcal{A}).
\end{equation}
Hence summing inequality \eqref{sigismon} over $(i,j)\in\{1,\dots
,K-1\}^2$, one obtains that
%
%
\begin{equation}
\nu\bigl(v(\bar\sigma) \bigr)\geq\sqrt{N}A(\nu-\mu) (\mathcal{A}),
\end{equation}
which, together with \eqref{lesoufs}, implies \eqref{ninjatwist}.

To prove \eqref{xencuts} we need to show the following result.

Although it is quite an intuitive statement, the proof is a bit
technical, and we will perform it in Appendix~\ref{app}.

We go back to the proof of \eqref{xencuts}.
Assume that $\nu$ is increasing and satisfies
%
%
\begin{equation}
\label{assume} \nu(\mathcal{A})\leq(1+\alpha)\mu(\mathcal{A}).
\end{equation}
Then from \eqref{letrucmoche} we have
%
%
\begin{equation}
\label{graff} \nu\bigl(\mathcal{B} ^c \bigr)\leq(1+\alpha)\mu
\bigl( \mathcal{B} ^c \bigr).
\end{equation}
Notice also that from the definition, if
$\bar\sigma\in\mathcal{B} $, $\bar\sigma'\in\mathcal{A}$
(improperly one can
consider $\mathcal{A}$ and $\mathcal{B} ^c$ as subsets of $\bar S_N$),
then $\bar\sigma\leq\bar\sigma'$, and thus from Proposition~\ref
{projecmon},
%
%
\begin{equation}
\forall\bar\sigma\in\mathcal{B} , \forall\bar\sigma'\in
\mathcal{A},\qquad\frac{\bar\nu}{\bar\mu}(\bar\sigma)\leq
\frac{\bar\nu}{\bar
\mu} \bigl(\bar
\sigma' \bigr),
\end{equation}
which, once averaged on $\sigma\in\mathcal{A}$, gives [using \eqref{assume}]
%
%
\begin{equation}
\label{graf} \forall\bar\sigma\in\mathcal{B} , \qquad\frac{\bar
\nu(\bar
\sigma
)}{\bar\mu(\bar\sigma)}\leq
\frac{\nu}{\mu}(\mathcal{A})\leq1+\alpha.
\end{equation}
Hence using \eqref{graff}, \eqref{graf} and \eqref{lesoufs} we have
%
%
\begin{eqnarray}
\llVert\bar\mu-\bar\nu\rrVert&\leq&\int_{\mathcal{B} ^c}
\biggl(
\frac{\bar\nu}{\bar\mu}(\bar\sigma)-1 \biggr)_+\bar\mu
(\mathrm{d}\bar\sigma)+ \int
_{\mathcal{B} } \biggl(\frac{\bar\nu}{\bar\mu}(\bar\sigma)-1
\biggr)_+\bar
\mu(\mathrm{d}\bar\sigma)
\nonumber
\\[-8pt]
\\[-8pt]
&\leq&\bar\nu\bigl(\mathcal{B} ^c \bigr)+\alpha\bar\mu
(\mathcal{B}
) \leq(1+\alpha)\delta_2+ \alpha\leq2\alpha+\delta_2.
\nonumber
\end{eqnarray}
\upqed
\end{pf*}

\subsection{Proof of Proposition \texorpdfstring{\protect\ref{groom3}}{5.3}}

Between time $t_2$ and $t_3$, a consequence of the censoring is that
the values taken by the sets
\[
\sigma_t \bigl(\{x_{i-1}+1,\dots, x_{i}\}
\bigr), \qquad i\in\{1,\dots,K\}
\]
are constant in time. On this time interval, the dynamics can be
considered as a product of $K$ independent AT shuffle, and the corresponding
equilibrium measure conditioned on the starting point $\sigma_{t_2}$
is simply
\[
\mu\bigl( \cdot\vert\sigma\bigl(\{x_{i-1}+1,\dots,
x_{i}\} \bigr)=\sigma_{t_2} \bigl(\{x_{i-1}+1,
\dots, x_{i}\} \bigr),\ \forall i\in\{1,\dots,K\} \bigr)=:
\mu_{\sigma_{t_2}}.
\]

Using Proposition~\ref{wilson} and with the same reasoning as in the
proof of Proposition~\ref{groom1},
we have, for any realization of $\sigma_{t_2}$,
%
%
\begin{eqnarray}
\quad&& \bigl\llVert\mathbb{P} ( \sigma_{t_3}\in\cdot\vert
\sigma
_{t_2} ) - \mu_{\sigma_{t_2}} \bigr\rrVert_{\mathrm{TV}}
\nonumber
\\
&&\qquad\leq K\times10 \biggl(\frac{N}{K}+1 \biggr)\exp\biggl(-2t
\biggl(1-\cos\biggl((t_3-t_2)\frac{\pi}{(N/K+1)} \biggr)
\biggr) \biggr)
\\
&&\qquad\leq\varepsilon/3,
\nonumber
\end{eqnarray}
provided that $N$ has been chosen small enough.

Considering the push-forward of the measures on semi-skeleton, and
integrating on the event $\{ \bar\sigma_{t_2}=\xi\}$,
we obtain that for every $\xi\in\bar S_N$,
%
%
\begin{equation}
\label{petitdist} \bigl\llVert\widehat\nu_{t_3}( \cdot\vert\bar
\sigma=
\xi)-\widehat\mu(\cdot\vert\bar\sigma=\xi) \bigr\rrVert
_{\mathrm{TV}} \leq
\varepsilon/3.
\end{equation}
Finally, to conclude we just need to remark that the distribution of
$\bar\sigma_{t_3}$
is the same as the one of $\bar\sigma_{t_2}$ (indeed, with the
censoring we have $\bar\sigma_{t_3}=\bar\sigma_{t_2}$) which is
close to equilibrium, according to Proposition~\ref{groom2}, so that
we can conclude. More formally we have
%
%
\begin{eqnarray}
\label{ladecomm} 2\llVert\widehat\nu_{t_3}- \widehat\mu
\rrVert_{\mathrm{TV}}&=&\sum_{\xi\in\bar
S_N}\sum
_{\{
\widehat\sigma\in\widehat S_N \vert \bar\sigma=\xi\}} \big\vert
\widehat\nu_{t_3}( \widehat\sigma)-
\widehat\mu(\widehat\sigma)\big\vert
\nonumber
\\
&\leq&\sum_{\xi\in\bar S_N}\sum_{\{ \widehat\sigma\in\widehat
S_N \vert \bar\sigma=\xi\}}
\bar\nu_{t_3}(\xi) \big\vert\widehat\nu_{t_3} (\widehat\sigma
\vert\bar\sigma=\xi)- \widehat\mu(\widehat\sigma\vert\bar
\sigma=\xi) \big\vert
\nonumber
\\
&&\hspace*{67pt}{} + \widehat\mu(\widehat\sigma\vert\bar\sigma=\xi) \big\vert
\bar
\nu_{t_3}(\xi)-\bar\mu(\xi) \big\vert
\\
&=&2 \biggl(\llVert\bar\nu_{t_3}-\bar\mu\rrVert_{\mathrm{TV}}+
\sum_{\xi\in\bar S_N} \bar\nu_{t_3}(\xi) \bigl\llVert
\widehat\nu_{t_3}(\cdot\vert\bar\sigma=\xi)-\widehat\mu(\cdot
\vert
\sigma=\xi) \bigr\rrVert_{\mathrm{TV}} \biggr)
\nonumber
\\
&\leq&4\varepsilon/3,
\nonumber
\end{eqnarray}
where the last inequality uses Proposition~\ref{groom2} and \eqref{petitdist}.

\section{Technical tools for the exclusion process}

To compute the mixing time of the exclusion process, we need tools
similar those developed in Sections~\ref{monotool} and \ref{nomonotool}.
In many cases, the proof is either a consequence of or exactly similar
to the proof performed for $S_N$, and thus is left to the reader.

\subsection{Ordering \texorpdfstring{$\Omega_{N,k}$}{$Omega_{N,k}$} and monotonicity properties}

To each $\gamma\in\Omega_{N,k}$ we can associate a lattice path
$\eta$
in the following manner:
%
%
\begin{equation}
\label{defeta} \eta(x):=\sum_{z=1}^x
\gamma(z)-\frac{xk}{N}.
\end{equation}
It is an injective mapping.

In what follows we describe the dynamics only in terms of $\eta$ (and
write $\Omega_{N,k}$ for the image set of $\gamma\mapsto\eta$
as it brings no confusion).

We consider the natural order on $\Omega_{N,k}$ given by
%
%
\begin{equation}
\eta\geq\eta' \quad\Leftrightarrow\quad\forall x \in\{1,\dots,
N-1\},\qquad\eta(x)\geq\eta'(x).
\end{equation}
We call $\wedge$ the maximal element of $\Omega_{N,k}$ and $\vee$ its
minimal element.
These symbols are used because they look like the graphs of the
extremal paths.
We have
%
%
\begin{eqnarray}
\wedge(x)&=&N^{-1}\min\bigl( (N-k)x, k(N-x) \bigr),
\nonumber
\\[-8pt]
\\[-8pt]
\vee(x)&=&N^{-1}\max\bigl( -kx, (N-k) (x-N) \bigr).
\nonumber
\end{eqnarray}

Note that the mapping $\gamma\mapsto\eta$ corresponds the $k$th line
of the mapping
$\sigma\mapsto\widetilde\sigma$ [see \eqref{tildesigma}]
introduced in
Section~\ref{monotool}, or more precisely if $\gamma=\gamma_{\sigma
}$ is
the image of $\sigma$ by the mapping \eqref{gammasigma}, then $\eta
(\cdot)=\widetilde\sigma(\cdot,k)$.

For $\xi\in\Omega_{N,k}$, we write $(\eta^\xi_t)_{t\geq0}$ for the
dynamics with initial condition $\xi$ and $P_t^\xi$ for the marginal
law at time $t$.
If $\nu$ is a probability on $\Omega_{N,k}$, we write $P_t^\nu$ for the
law of $\eta_t$ starting with an initial condition that has
distribution $\nu$.

The projection on $\Omega_{N,k}$ of the graphical construction of
Section~\ref{graphix} provides a coupling of the different $(\eta^\xi
_t)_{t\geq0}$
that preserves the order, that is, which is such that
%
%
\begin{equation}
\label{orderpreservrrr} \xi\geq\xi' \quad\Rightarrow\quad\forall
t \geq0,
\qquad\eta^\xi_t \geq\eta^{\xi'}_t.
\end{equation}
In Section~\ref{glopal} we will present another construction that also
preserves the order.

\subsection{FKG and censoring and monotonicity conservation}

The statespace $\Omega_{N,k}$ is a distributive lattice when equipped
with the two operations
$\operatorname{\mathbf{min}}$ and $\operatorname{\mathbf{max}}$
defined (for $\eta, \xi\in\Omega_{N,k}$) as follows:
%
%
\begin{eqnarray}
\label{defvee} %
\forall x\in\Omega_{N,k},\qquad
\operatorname{\mathbf{min}}(\eta, \xi) (x)&=&\min\bigl(\eta
(x),\xi(x) \bigr),
\nonumber
\\[-8pt]
\\[-8pt]
\forall x\in\Omega_{N,k},\qquad\operatorname{\mathbf{max}}(\eta,
\xi)
(x)&=& \max\bigl(\eta(x),\xi(x) \bigr).
\nonumber
\end{eqnarray}
This means that $\Omega_{N,k}$ is stable by these operations and that
each one is distributive with respect to the other.
For this reason the FKG inequality as proved in \cite{cfFKG} is
valid. In the proof we also need a stronger result which is a
consequence Holley's inequality.

%
\begin{proposition}[{(\cite{cfFKG}, Proposition~1,
\cite{cfHolley}, Theorem~6)}]
If $f$ and $g$ are two increasing functions on $\Omega_{N,k}$, then
%
%
\begin{equation}
\label{FKG} \mu(fg)\geq\mu(f)\mu(g).
\end{equation}
Furthermore if $A$ and $B$ are increasing subsets of $\Omega_{N,k}$
such that
$A\subset B$ and $\operatorname{\mathbf{min}}(A, B)\subset B$, where
\[
\operatorname{\mathbf{min}}(A, B):= \bigl\{ \operatorname{\mathbf
{min}} \bigl(
\eta, \eta' \bigr) \vert\eta\in A, \eta'\in B \bigr
\},
\]
then for any increasing function $f$,
%
%
\begin{equation}
\label{Holley} \mu(f \vert A) \geq\mu(f \vert B).
\end{equation}
\end{proposition}

\begin{pf}
A sufficient condition for the FKG inequality \cite{cfFKG},
Proposition~1, to hold for $\mu$ is that
%
%
\begin{equation}
\mu\bigl(\operatorname{\mathbf{min}}(\eta,\xi) \bigr)\mu\bigl
(\operatorname{
\mathbf{max}}(\eta,\xi) \bigr)\geq\mu(\eta)\mu(\xi),
\end{equation}
which is obviously satisfied for the uniform measure on $\Omega_{N,k}$.
The second inequality is Holley's inequality \cite{cfHolley}, Corollary~11,
applied to $\mu(f \vert A)$ and
$\mu(f \vert B)$.
What has to be checked is that
%
%
\begin{equation}
\label{gromof} \mu\bigl(\operatorname{\mathbf{max}}(\eta,\xi)
\vert A \bigr)\mu
\bigl( \operatorname{\mathbf{min}}(\eta,\xi) \vert B \bigr)\geq
\mu(\eta\vert A)
\mu(\xi\vert B),
\end{equation}
which is obviously valid if either $\eta\notin A$ or $\xi\notin B$.
If $\eta\in A$ and $\xi\in B$, then, as $A$ is increasing
$\operatorname{\mathbf{max}}(\eta
,\xi)\in A$ and
from the assumption $\operatorname{\mathbf{min}}(A, B)\subset B$, we
have $\operatorname{\mathbf{min}}(\eta,\xi
) \in B$, and hence \eqref{gromof} holds in any case.
\end{pf}

Using the terminology of Section~\ref{graphix}, we say that an update
of $\eta_t$ is performed\vspace*{1pt} at the coordinate $x$ when $\mathcal{T} _x$ rings.
As in Section~\ref{censorsec}, we define $P_t^{\nu,\mathcal{C}}$ to
be the
law of $\eta_t$ which has performed a censored dynamics with scheme
$\mathcal{C}$ with initial
distribution $\nu$.

The reader can check that Proposition~\ref{censor} is also valid for
the chain $\eta_t$, and there are two different ways to do this,
either by saying that it is just \cite{cfPW}, Theorem~1.1, and
checking that our Markov chain with its system of updates is a monotone
system for the definition given in \cite{cfPW}, or by performing the
necessary changes
to the proof of Proposition~\ref{censor}.

Finally we remark that Proposition~\ref{proppres} also applies to the
exclusion process. To adapt the proof one needs
to consider, instead of $\sigma_x^{\bullet}$, the
sets
\[
\eta_x^{\bullet}:= \bigl\{\xi\in\Omega_{N,k} \vert
\forall y\ne x, \xi(y)=\eta(y) \bigr\},
\]
which, depending on the values of $\xi$ and $x$ can have either one or
two elements.
We record these results here.

%
\begin{proposition}\label{censorsep}
If $\nu$ is an increasing probability on $\Omega_{N,k}$, then for all
positive $t$ and all censoring schemes $\mathcal{C}$,
$P_t^\nu$ and $P^{\nu,\mathcal{C}}_t$ are increasing.

Furthermore we have
\[
\bigl\llVert P_t^\nu-\mu\bigr\rrVert_{\mathrm{TV}}
\leq\bigl\llVert P_t^{\nu,\mathcal{C}}-\mu\bigr\rrVert
_{\mathrm{TV}}.
\]
\end{proposition}

\subsection{Stability for projection}

The equivalent of Proposition~\ref{projecmon} is valid for $\Omega_{N,k}$
and is in fact much easier to prove.

We define $\bar\eta$ the skeleton of $\eta$ as [recall \eqref{defxi}]
%
%
\begin{equation}
\forall i \in\{0,\dots,K\},\qquad\bar\eta(i)=\eta(x_i)
\end{equation}
and equip the set of skeletons $\bar\Omega_{N,k}$ with the natural order.
For $\nu$ probability law on $\Omega_{N,k}$, define $\bar\nu$ to be
the pushed forward law for the projection
$\eta\mapsto\bar\eta$. We define in the same manner $\bar\nu_i$
for the projection on one coordinate.

%
\begin{proposition}\label{projecmonsep}
If $\nu$ is an increasing probability on $\Omega_{N,k}$,
then the density of $\bar\nu/\bar\mu$ is an increasing function of
$\bar\Omega_{N,k}$.

The density $\bar\mu_i$ is also increasing.
\end{proposition}

The proof is identical to that of \eqref{ledernier}.

\subsection{Limit of the mean height and rough upper bounds on the
mixing time}

As $\eta_t$ has the same law as $\widetilde\sigma_t(\cdot, k)$,
Lemma~\ref{boudk} gives us the behavior of the mean
value $\mathbb{E} [\eta^\xi_t(x) ]$.
More precisely, we have the following:

%
\begin{lemma}\label{boudk2}
For all $k\leq N/2$ we have:
\begin{itemize}
\item for any $\xi\in S_N$ and $t\geq0$, we have
%
%
\begin{equation}
\label{groupeta} \max_{x\in\{0,\dots,N\}}\mathbb{E} \bigl[
\eta^\xi_t(x) \bigr] \leq4 k e^{-\lambda_N t },
\end{equation}
where
\[
\lambda_N:=2 \biggl(1-\cos\biggl(\frac{\pi}{N} \biggr)
\biggr)=\frac
{\pi^2}{N^2} \bigl(1+o(1) \bigr);
\]
\item when $\xi=\wedge$,
%
%
\begin{equation}
\label{groupeta2} \mathbb{E} \bigl[ \eta^\wedge_t(x) \bigr]
\geq\frac{k}{\pi}\exp(-\lambda_N t) \sin\biggl(
\frac{\pi x}{N} \biggr).
\end{equation}
\end{itemize}
\end{lemma}

Similar to Proposition~\ref{wilson} we have the following upper
bound for the distance to equilibrium.

%
\begin{proposition}\label{wilson2}
For all $N$ sufficiently large and $k\in\{0,\dots,N\}$, for all
$\varepsilon>0$,
%
%
\begin{equation}
d^{N,k}(t)\leq10k \exp(-t \lambda_N),
\end{equation}
where
\[
\lambda_N:=2 \bigl(1-\cos(\pi/N) \bigr).
\]
\end{proposition}

The idea of the proof essentially comes from \cite{cfWilson}, Section~8.1,
with some modification performed to adapt to continuous time and the
fact that we deal with the exclusion process.
The reader can check that taking $k=N$ in the proof gives a proof of
Proposition~\ref{wilson}.

\begin{pf*}{Proof of Proposition \ref{wilson2}}
Using \eqref{gramsci2}, it is sufficient to bound the distance
$\llVert P^\xi_t-P^{\xi'}_t\rrVert _{\mathrm{TV}}$ uniformly in
$\xi$, $\xi'$.
To this end, we construct a coupling of $\eta_t^\xi$ and $\eta
_t^{\xi'}$ (which is not the one given by the graphical construction
and is not even Markovian) and prove that for this coupling,
%
%
\begin{equation}
\label{xidor} \mathbb{P} \bigl[\eta_t^\xi\ne
\eta_t^{\xi'} \bigr] \leq10k \exp(-t \lambda_N).
\end{equation}
It is in fact more convenient to consider the AT shuffle and
construct a coupling for this larger process.
Instead of proving \eqref{xidor}, we prove that for all $\xi$, $\xi
'\in S_N$,
%
%
\begin{equation}
\label{zidor} \mathbb{P} \bigl[\forall i\in\{1,\dots, k\}, \bigl(
\sigma_t^\xi\bigr)^{-1}(i) = \bigl(
\sigma_t^{\xi'} \bigr)^{-1}(i) \bigr] \leq10k
\exp(-t \lambda_N),
\end{equation}
and then deduce \eqref{xidor} from \eqref{zidor} using that the
mapping \eqref{gammasigma} projects the AT shuffle on the exclusion process.

The coupling has the following rules:
\begin{itemize}
\item if $\sigma^\xi_t(x)\ne\sigma^{\xi'}_t(x)$ and $\sigma^{\xi
}_t(x+1)\ne\sigma^{\xi'}_t(x+1)$, then the transition $\sigma\to
\sigma\circ\tau_x$
occurs independently with rate one for each of the two processes;
\item if either $\sigma^\xi_t(x)= \sigma^{\xi'}_t(x)$ or $\sigma
^{\xi}_t(x+1)= \sigma^{\xi'}_t(x+1)$ (or both), then the transition
$\sigma\to\sigma\circ\tau_x$
occurs simultaneously for the two processes (with rate one).
\end{itemize}

Let $X^i_t:=(\sigma^{\xi}_t)^{-1}(i)$ and $Y^i_t(\sigma^{\xi
'}_t)^{-1}(i)$ denote the trajectory of the particle labeled $i$ for
the two coupled permutations.
The couple $(X^i_t,Y^i_t)$ is a Markov chain with the following
transition rules:
\begin{itemize}
\item if $x\ne y$, then the transitions $(x,y)\to(x\pm1,y)$,
$(x,y)\to(x,y\pm1)$ occur with rate one, provided
the two coordinates stay between $1$ and $n$;
\item if $x=y$, then the transitions $(x,y)\to(x+1,y+1)$ and $(x,y)\to
(x-1,y-1)$ occur with rate one, provided the two coordinates stay
between $1$ and $n$.
\end{itemize}
All the other transitions have rate $0$. In particular, once $X^i_t$
and $Y^i_t$ have merged, they stay together.

By union bound, we have
%
%
\begin{eqnarray}
&&\mathbb{P} \bigl[\exists i\in\{1,\dots, k\}, \bigl(\sigma_t^\xi
\bigr)^{-1}(i) \ne\bigl(\sigma_t^{\xi'}
\bigr)^{-1}(i) \bigr]
\nonumber
\\[-8pt]
\\[-8pt]
&&\qquad\leq k \max_{(x,y)\in\{1,\times, N\}^2} \mathbf{P} _{x,y}[X_t
\ne Y_t],
\nonumber
\end{eqnarray}
where $(X_t,Y_t)$ is a Markov chain starting from $(x,y)$ and whose
transitions rules are the same as those of $(X^i_t,Y^i_t)$.

We conclude by using the following lemma.

%
\begin{lemma}\label{wilson9}
We have for all $(x,y)$,
%
%
\begin{equation}
\label{cfaz} \mathbf{P} _{x,y}[X_t\ne Y_t]
\leq10 \exp(-t \lambda_N).
\end{equation}
\end{lemma}

\begin{pf}
This result is proved in \cite{cfWilson}, Lemma~9 (to which we refer
for the computations), in the discrete case by diagonalization of the
transition matrix of the random-walk
$(X,Y)$ killed when it hits the diagonal. We write $G^*_t$ for the
semi-group of this process.

Let us explain briefly how it adapts to continuous time.
By symmetry it is sufficient to consider $1\leq x< y\leq N$ [hence we
have a killed Markov chain with $N(N-1)/2$ possible states].
For convenience we shift coordinates by $1/2$ so that $x,y\in\{
1/2,\dots, N-1/2\}$.

We remark that the functions $u_{i,j}$, $0\leq i<j<N$, defined by
%
%
\begin{equation}
u_{i,j}(x,y):=\cos\biggl(\frac{i \pi x}{n} \biggr) \cos\biggl(
\frac
{j \pi y}{n} \biggr) - \cos\biggl(\frac{i \pi y}{n} \biggr) \cos
\biggl(
\frac{j \pi x}{n} \biggr),
\end{equation}
form an orthogonal basis of eigenfunctions for the generator of the
killed random walk (see \cite{cfWilson}), with respective eigenvalues
$-\lambda_{i,j,N}$ where
\begin{eqnarray*}
\lambda_{i,j,N}&:=& 2 \bigl[ \bigl(1-\cos(i\pi/N) \bigr)+ \bigl
(1-\cos(j
\pi/N) \bigr) \bigr]\\
&\geq&(i+j) 2\bigl(1-\cos(\pi/N)\bigr).
\end{eqnarray*}
We furthermore have
\[
\llVert u_{i,j} \rrVert_2^2=
N^2(1+\mathbf{1}_{i=0})/4\geq N^2/4.
\]

Hence by decomposition of $G^*_t$ on the basis of eigenfunction, we have
%
%
\begin{eqnarray}
\mathbf{P} _{x_0,y_0}[X_t\ne Y_t]&=&
\sum_{1\leq x<y\leq N-1/2} G^*_t \bigl((x_0,y_0),(x,y)
\bigr)
\nonumber
\\
&=& \sum_{0\leq i<j< N} \sum_{1/2 \leq x<y\leq N-1/2}
\frac
{u_{i,j}(x_0,y_0)u_{i,j}(x_0,y_0)}{\llVert u_{i,j} \rrVert
_2^2}e^{-\lambda
_{i,j,N}t}
\nonumber
\\[-8pt]
\\[-8pt]
&\leq&8 \sum_{0\leq i<j< N} e^{-(i+j)\lambda_{N}t} \leq8 \sum
_{i=0}^\infty\sum
_{j=1}^\infty e^{-(i+j)\lambda_{N}t}
\nonumber
\\
&=& \frac{ 8 e^{-\lambda_{N}t}}{(1-e^{-\lambda_N t})^2},
\nonumber
\end{eqnarray}
where in the first inequality we used $\llVert u_{i,j}\rrVert _\infty
\leq2$.
Then \eqref{cfaz} is trivial if $e^{-\lambda_{N}t}\geq1/10$ and is a
consequence of the above inequality when
$e^{-\lambda_{N}t}\leq1/10$.\qed\qed\noqed
\end{pf}\noqed\end{pf*}

\section{Lower bound for the mixing times for the exclusion
process}\label{lbexp}

In this section we prove that if $\min(k(N),N-k(N))\to\infty$, then
for all $\varepsilon\in(0,1)$ and $\delta>0$, for $N$ large enough,
%
%
\begin{eqnarray}
\label{cromas} d^{N,k} \biggl( \frac{1}{2\pi^2}N^2
\log\min(k,N-k) (1-\delta) \biggr)&\geq& \varepsilon,
\nonumber
\\[-8pt]
\\[-8pt]
d^{N,k}_{S} \biggl(\frac{1}{\pi^2}N^2 \log
\min(k,N-k) (1-\delta) \biggr)&\geq&\varepsilon.
\nonumber
\end{eqnarray}
We consider for simplicity that $k\leq N/2$, the result for $k>N/2$
follows by symmetry.
A proof of the first inequality is in fact already present in \cite
{cfWilson}, but we present an alternative short proof at the end of
the section for the sake of completeness.

To prove the second inequality, we need the following assumption:
\[
\lim_{N\to\infty} \frac{\log\log k}{\log N}=\infty.
\]
This is mainly for technical reasons, and we believe that the result
holds with greater generality.

\subsection{For the separation distance}

As we are looking for a lower bound on $d^{N,k}_{S}(t)$, it is
sufficient to have a lower bound for $d_S(P^{\wedge}_t,\mu)$,
even though we cannot prove that the separation distance is maximized
when starting from an extremal condition.
From Proposition~\ref{censorsep}, $P^{\wedge}_t$ is an increasing
probability (because the Dirac measure on $\wedge$ is an increasing
probability), and we have
%
%
\begin{equation}
d_S \bigl(P^{\wedge}_t,\mu\bigr)=1-
\frac{P^{\wedge}_t(\vee)}{\mu(\vee)}.
\end{equation}
Hence what we have to prove is that for $t=t_1:=\frac{1-\delta}{\pi
^2}N^2 \log k$,
%
%
\begin{equation}
\label{groomy} \frac{P^{\wedge}_{t_1}(\vee)}{\mu(\vee)}\geq
1-\varepsilon.
\end{equation}
By reversibility of the dynamics, one has for all $\eta$, $\eta'$ and
all $t\geq0$,
\[
P_t^{\eta'}(\eta)=P^{\eta}_t \bigl(
\eta' \bigr).
\]
Combining this with the semi-group property, we have
%
%
\begin{equation}
\label{extreme} P^{\wedge}_t(\vee)=\sum
_{\eta\in\Omega_{N,k}} P^{\wedge
}_{t/2}(\eta)
P^{\vee}_{t/2}(\eta).
\end{equation}

Now, we partition $\Omega_{N,k}$ into two sets,
%
%
\begin{eqnarray}
\Omega_1&:=& \bigl\{ \eta\in\Omega_{N,k} \vert
\eta\bigl(\lceil N/2 \rceil\bigr)\geq0 \bigr\} ,
\nonumber
\\[-8pt]
\\[-8pt]
\Omega_2&:=& \bigl\{ \eta\in\Omega_{N,k} \vert\eta\bigl(
\lceil N/2 \rceil\bigr) < 0 \bigr\},
\nonumber
\end{eqnarray}
and bound from above the contribution of each in \eqref{extreme}.

Note that both $\Omega_1$ and $\Omega_2$ are distributive lattices (both
sets are stable under the composition laws $\operatorname{\mathbf
{min}}$ and $\operatorname{\mathbf{max}}$), and
thus the FKG inequality \eqref{FKG} is also valid when $\mu$ is
replaced by
$\mu(\cdot\vert \Omega_i)$.
Hence we have
%
%
\begin{eqnarray}
\sum_{\eta\in\Omega_{1}} P^{\wedge}_{t/2}(
\eta) P^{\vee
}_{t/2}(\eta)&=& \pmatrix{N\cr k}\mu(
\Omega_1) \sum_{\eta\in\Omega_1} \mu(\eta\vert
\Omega_1) P^{\wedge
}_{t/2}(\eta) P^{\vee}_{t/2}(
\eta)
\nonumber
\\
&\leq&\pmatrix{N\cr k}\mu(\Omega_1) \biggl(\sum
_{\eta\in\Omega_1} \mu(\eta\vert\Omega_1)
P^{\wedge}_{t/2}(\eta) \biggr)
\nonumber
\\[-8pt]
\\[-8pt]
&&{}\times\biggl(\sum_{\eta\in
\Omega
_1} \mu(\eta\vert
\Omega_1) P^{\wedge}_{t/2}(\eta) P^{\vee
}_{t/2}(
\eta) \biggr)
\nonumber
\\
&=& \pmatrix{N\cr k}^{-1}\mu(\Omega_1)^{-1}P^{\wedge}_{t/2}(
\Omega_1)P^{\vee
}_{t/2}(\Omega_1).
\nonumber
\end{eqnarray}
Similarly,
%
%
\begin{equation}
\sum_{\eta\in\Omega_{2}} P^{\wedge}_{t/2}(\eta)
P^{\vee
}_{t/2}(\eta)\leq\pmatrix{N\cr k}^{-1}\mu(
\Omega_2)^{-1}P^{\wedge}_{t/2}(\Omega
_2)P^{\vee
}_{t/2}(\Omega_2).
\end{equation}
Thus from \eqref{extreme} we have
%
%
\begin{equation}
\frac{P^{\wedge}_t(\vee)}{\mu(\vee)}\leq\mu(\Omega
_1)^{-1}P^{\wedge
}_{t/2}(
\Omega_1)P^{\vee}_{t/2}(\Omega_1)+ \mu(
\Omega_2)^{-1}P^{\wedge}_{t/2}(
\Omega_2)P^{\vee}_{t/2}(\Omega_2).
\end{equation}

As $\eta_{\lceil N/2\rceil}$ satisfies the central limit theorem, we have
\[
\lim_{N\to\infty} \mu(\Omega_i)=1/2,\qquad i=1,2,
\]
and hence, for all $N$ sufficiently large,
\[
\frac{P^{\wedge}_t(\vee)}{\mu(\vee)}\leq3 \bigl(P^{\vee
}_{t/2}(
\Omega_1)+P^{\wedge}_{t/2}(\Omega_2)
\bigr).
\]
Hence to prove \eqref{groomy}, we just need to show that $P^{\vee
}_{t/2}(\Omega_1)$ and $P^{\wedge}_{t/2}(\Omega_2)$ are small.

%
\begin{lemma}\label{lemecaumilieu}
Set
\[
t_0:=\frac{1}{2\pi^2}N^2\log k(1-\delta).
\]
Then if
\[
\lim_{N\to\infty} \frac{\log k}{\log\log N}=\infty,
\]
we have
%
%
\begin{eqnarray}
\lim_{N\to\infty}P^{\vee}_{t_0}(
\Omega_1)&=&0,
\nonumber
\\[-8pt]
\\[-8pt]
\lim_{N\to\infty}P^{\wedge}_{t_0}(
\Omega_2)&=&0.
\nonumber
\end{eqnarray}
\end{lemma}

We only prove the second limit, the first being exactly the same.

\subsection{Proof of Lemma \texorpdfstring{\protect\ref{lemecaumilieu}}{7.1}}

We want to prove that when one starts the dynamics from the maximal
path $\wedge$, w.h.p. $\eta_{t_0}(\lceil N/2\rceil)\geq0$.
To do so we compute the expectation and variance of $\eta_{t_0}(\lceil
N/2\rceil)$.

%
\begin{lemma}\label{lemecaumilieubis}
We can find a constant $C$ such that for all $N$ large enough,
%
%
\begin{eqnarray}
P^{\wedge}_{t_0} \bigl( \eta\bigl( \lceil N/2
\rceil\bigr) \bigr)&\geq& C^{-1} k^{(1+\delta)/2},
\nonumber
\\[-8pt]
\\[-8pt]
\operatorname{Var}_{P^{\wedge}_{t_0}} \bigl( \eta\bigl(\lceil N/2
\rceil\bigr) \bigr)&\leq& C k
\log N.
\nonumber
\end{eqnarray}
\end{lemma}

Then Lemma~\ref{lemecaumilieu} is easily deduced by using Chebytchev's
inequality.

\begin{pf*}{Proof of Lemma~\ref{lemecaumilieubis}}
The inequality for the expectation is obtained by using \eqref{group3}
[recall that $\eta_t$ has the same law that
$\widetilde\sigma_t(\cdot, k)$].

To control the variance, we use an idea similar to that in \cite{cfLST}, Section~7, with the use of martingale and Fourier coefficients.
The Fourier decomposition of $\eta$ on the basis of eigenfunctions
$(u_i)_{i=1}^{N-1}$ given by \eqref{eigenu}, implies that for all
$y\in\{0,\dots, N\}$,
%
%
\begin{equation}
\label{fourierdeco} \eta(y)= \frac{2} N \sum_{i=1}^{N-1}
\sum_{x=1}^{N-1} \eta(x)\sin\biggl(
\frac{i\pi x}{N} \biggr) \sin\biggl(\frac{i\pi y}{N} \biggr).
\end{equation}
The reader can check that
\[
\eta\mapsto\sum_{x=1}^{N-1} \eta(x)\sin
\biggl(\frac{i\pi
x}{N} \biggr)
\]
are eigenfunctions of the generator of the Markov chain \eqref
{crading} with eigenvalue $-\lambda_{N,i}$; recall \eqref{lambdani}.
For this reason, for each $i$, the process
\[
e^{\lambda_{N,i}t}\sum_{x=0}^N
\eta_t(x)\sin\biggl(\frac{i\pi
x}{N} \biggr)= e^{\lambda_{N,i}t}
a_{i}(\eta_t),
\]
where
\[
a_{i}(\eta):=\sum_{x=0}^N
\eta(x)\sin\biggl(\frac{i\pi x}{N} \biggr)
\]
is a martingale (in $t$).

We consider the following martingale which is a linear combination of
the above:
%
%
\begin{equation}
M_t:= \frac{2}{N}\sum_{i=1}^{N-1}
e^{\lambda_{N,i}(t-t_0)}\sin\bigl(\pi i \lceil N/2 \rceil/N \bigr)a_{i}(
\eta_t).
\end{equation}
As a consequence of \eqref{fourierdeco}, it satisfies
\[
M_{t_0}=\eta_{t_0} \bigl( \lceil N/2 \rceil\bigr).
\]

To control the variance of $M_{t_0}$, we prove a uniform upper bound on
the martingale bracket and use the fact that,
as the initial variance is zero, we have
%
%
\begin{equation}
\operatorname{Var}\bigl[M^2_{t_0} \bigr]=\mathbb{E} \bigl[ \langle M
\rangle^2_{t_0} \bigr].
\end{equation}

It is easy to obtain an upper bound on the bracket of the martingale.
As each transition changes the value of $M$ by at most
\[
\frac{2} N \sum_{i=1}^N
e^{\lambda_{N,i}(t-t_0)}
\]
and the transitions occur with a rate at most $2k$ (there are $k$
particles which can perform at most two transitions, each with rate
$1$), we have
%
%
\begin{eqnarray}
\langle M \rangle^2_{t_0}&\leq&\int_0^{t_0}
\frac
{8k}{N^2} \Biggl(\sum_{i=1}^{N-1}
e^{\lambda_{N,i}(t-t_0)} \Biggr)^2 \,\mathrm{d}t
\nonumber
\\[-8pt]
\\[-8pt]
&\leq&\int_{-\infty}^0 \frac{8k}{N^2} \Biggl(
\sum_{i=1}^{N-1} e^{\lambda
_{N,i}t}
\Biggr)^2 \,\mathrm{d}t = \frac{8k}{N^2} \sum
_{i,j=1}^{N-1} \frac{1}{\lambda_{N,i}+\lambda_{j,N}}.
\nonumber
\end{eqnarray}
One can find a constant $C$ such that for all $i$ and $N$,
\[
\lambda_{N,i}\geq\frac{i^2}{C N^2}.
\]
We have
%
%
\begin{equation}
\operatorname{Var}\bigl[M^2_{t_0} \bigr]\leq8C k \sum
_{i,j=1}^{N-1}\frac{1} {i^2+ j^2}\leq C' k
\log N.
\end{equation}
\upqed
\end{pf*}

\subsection{A lower bound on the total variation mixing time}

Let us now give a short proof for the first inequality of \eqref{cromas}.
Set
\[
a_1(\eta):=\sum_{x=1}^N \sin
\biggl(\frac{x}{\pi N} \biggr)\eta(x).
\]
As in the previous section, for any value of $t$,
\[
M_s:= e^{(s-t) \lambda_N}a_1(\eta_t)
\]
is a martingale. Note that
\[
M_t=a_1(\eta_t).
\]
If $\eta_0=\wedge$, there exists a constant $c$ such that for all
$s\geq0$, for all $N$ and $k$,
%
%
\begin{equation}
\label{croumphite} \mathbb{E} [M_s]=e^{-t\lambda_N}a_1(
\wedge)\geq c e^{-t\lambda_N}Nk.
\end{equation}
We control the variance of $M_t$ as follows:
%
%
\begin{equation}
\quad \operatorname{Var}\bigl[a \bigl(\eta^\wedge_t \bigr) \bigr
]=\operatorname{Var}
\bigl[M^2_{t} \bigr]=\mathbb{E} \bigl[ \langle M
\rangle^2_{t_0} \bigr]\leq Ck \int_0^t
e^{2(s-t)\lambda_N}\,\mathrm{d}s \leq Ck N^2.
\end{equation}
Taking $t=\infty$, we obtain that at equilibrium we have
\[
\operatorname{Var}_{\mu} \bigl(a_1(\eta) \bigr)\leq Ck N^2
\quad\mbox{and}\quad\mu\bigl(a_1(\eta) \bigr)=0.
\]

These bounds on the variance and expectation show that at time\vspace*{2pt}
$t=\frac{1}{2\pi^2}N^2\*\log k(1-\delta)$, the expectation of
$a(\eta
_1)$ is much larger than its typical fluctuations
so that its distribution cannot be close to equilibrium.

More precisely,
if $\mathbf{P} $ is a coupling of $P_t^{\wedge}$ (variable $\eta^1$) and
$\mu$ (variable $\eta^2$), we have
(by Chebytchev's inequality)
%
%
\begin{eqnarray}
\mathbf{P} \bigl[\eta^1=\eta^2 \bigr]&\leq&\mathbf{P}
\bigl[a_1 \bigl(\eta^1 \bigr)-a_1 \bigl(\eta
^2 \bigr)\leq0 \bigr] \leq\frac{\operatorname{Var}_{\mathbf{P} }
(a_1(\eta^1)-a_1(\eta^2))}{(\mathbf
{E} [a_1(\eta
^1)-a_1(\eta^2)])^2}
\nonumber
\\
&\leq&2\frac{\operatorname{Var}_{P^{\wedge_t}}(a_1(\eta))+
\operatorname{Var}_{\mu
}(a_1(\eta
))}{(P^{\wedge}_t[a_1(\eta)])^2}
\\
&\leq&\frac{CkN^2}{
e^{-2\lambda_N t }N^2k^2}=Ck^{-1}e^{2\lambda_N t}.
\nonumber
\end{eqnarray}

Applying this inequality for $t=\frac{1}{2\pi^2}N^2\log k(1-\delta)$
we deduce that the first line of \eqref{cromas} holds.

\section{The upper bound on the mixing time for the exclusion
process}\label{excluproc}

As the exclusion process is obtained by projecting the AT shuffle; its
mixing time is smaller.
Hence from Theorem~\ref{mainres} we already have, for any sequence $k(N)$,
%
%
\begin{equation}
\limsup_{N\to\infty} \frac{2\pi^2 T_{\mathrm
{mix}}^{N,k}(\varepsilon)}{N^2\log N}\leq1.
\end{equation}
This is sufficient to prove the upper bound on the mixing time of
Theorem~\ref{mainressep} when $k=N/2$, but this is not the case
when the number of particles is strictly smaller than
$N^{1-o(1)}$.

Contrary to the AT shuffle, the distance to equilibrium for the
exclusion process depends on the initial conditions,
and there
is a priori no reason for it to be maximized when the initial
conditions are chosen to be either $\vee$ or $\wedge$
(the extremal elements). However, most of the arguments involving
motonicity can be used only for these two cases, and
thus one must think of another strategy.

Assume that we have a coupling of the Markov chain trajectories $\eta
^\xi_t$ starting from all initial possible conditions $\xi\in\Omega
_{N,k}$, which preserves the order, or in other words satisfies \eqref
{orderpreservrrr}.
The coupling derived from the graphical construction of Section~\ref
{graphix} is an example of such coupling, but we will use another one
for our proof.
We call
$\mathbb{P} $ the law of the coupling.

Using the triangular inequality, we have for any $\xi$,
%
%
\begin{eqnarray}
\label{gramsci2} \bigl\llVert P_t^\xi- \pi\bigr\rrVert
_{\mathrm{TV}}&=& \bigl\llVert P_t^\xi-
P^{\pi}_t \bigr\rrVert_{\mathrm{TV}}
\nonumber
\\[-8pt]
\\[-8pt]
&\leq&\frac{1} {\llvert \Omega_{N,k}\rrvert } \sum_{\xi' \in
\Omega_{N,k}} \bigl\llVert
P_t^\xi-P^{\xi'}_t \bigr\rrVert
_{\mathrm{TV}}\leq\max_{\xi'} \bigl\llVert
P_t^\xi- P^{\xi'}_t \bigr\rrVert
_{\mathrm{TV}}.
\nonumber
\end{eqnarray}

As $\mathbb{P} $ provides a coupling between $P_t^\xi$ and $P^{\xi'}_t$,
using the characterization of the total variation distance given in
Lemma~\ref{caravar}, we have
%
%
\begin{equation}
\label{gramsci} \bigl\llVert P_t^\xi-
P^{\xi'}_t \bigr\rrVert\leq\mathbb{P} \bigl[
\eta^\xi_t\ne\eta_t^{\xi
'} \bigr]
\leq\mathbb{P} \bigl[\eta^\vee_t\ne\eta_t^\wedge
\bigr],
\end{equation}
where the last inequality is a consequence of \eqref{orderpreservrrr}:
both $\eta^\xi_t$ and $\eta_t^{\xi'}$
are squeezed between $\eta^\vee_t$ and $\eta_t^\wedge$, and thus
they must be equal once the dynamics starting from the extremal
initial conditions have coalesced.

This reasoning was used in \cite{cfWilson} to obtain an upper bound
on the
mixing-time using the coupling derived from the
graphical construction of Section~\ref{graphix}.
To have an improvement on Wilson's bound, one must necessarily use
another coupling. Indeed
the estimate he obtained
for the merging time of $\eta^\vee_t$ with $\eta_t^\wedge$, for the
coupling obtained with the graphical construction,
is tight; see \cite{cfWilson}, Table~1, coupling column.

\subsection{An alternative graphical construction for the exclusion
process}\label{glopal}

Let us present an alternative coupling that can be constructed for the
exclusion process.
The underlying idea is to find a construction that maximizes the
fluctuation of the area between
$\eta^\vee_t$ and $\eta_t^\wedge$ in order to make them coalesce
faster. To maximize the fluctuation, we want to make the corner-flips of
both trajectories as independent as possible.

The construction corresponds exactly to the graphical construction
for the zero-temperature Ising model in a $k\times(N-k)$ rectangle
with mixed boundary condition; see, for example, \cite{cfLST}, Section~2.3 and Figure~3, for a description of the model.

Set
\begin{eqnarray*}
\Theta&:=& \big\{ (x,z) \vert x\in\{1,\dots,N-1\}
\mbox{ and}
\\
&&\hphantom{\big\{}z\in\bigl\{ \max(0,x-N+k)-xk/N,\min(x,k)-xk/N \bigr\} \big\},
\end{eqnarray*}
and set $\mathcal{T} ^\uparrow$ and $\mathcal{T} ^\downarrow$ to be
two independent
rate-one clock processes
indexed by $\Theta$ [$\mathcal{T} ^\uparrow_{(x,z)}$ and
$\mathcal{T} ^\downarrow_{(x,z)}$ are two independent Poisson
processes of
intensity one of each $(x,z) \in\Theta$].

If $\mathcal{T} ^\uparrow_{(x,z)}$ rings at time $t$ then:
\begin{itemize}
\item if $\eta^\xi_{t^-}(x)=z$ and $\eta^\xi_{t^-}$ has a local
minimum at $x$,
then $\eta^\xi_{t}(x)=z+1$, and the other coordinate remains unchanged;
\item if these conditions are not satisfied, we do nothing.
\end{itemize}
If $\mathcal{T} ^\uparrow_{(x,z)}$ rings at time $t$, then:
\begin{itemize}
\item if $\eta^\xi_{t^-}(x)=z$ and $\eta^\xi_{t^-}$ has a local
maximum at $x$,
then $\eta^\xi_{t}(x)=z-1$, and the other coordinate remains unchanged;
\item if these conditions are not satisfied, we do nothing.
\end{itemize}

The reader can check that the dynamics we obtain is the exclusion
process and that it provides a coupling
satisfying \eqref{orderpreservrrr}. We call $\mathbb{P} $ the law of this
construction.

We want to prove the following:

%
\begin{proposition}\label{wegonosqueez}
Given $\delta>0$, set
\[
t_1:=\frac{N^2}{2\pi^2}\log k(1+\delta).
\]
Then for any $\varepsilon>0$, we have
\[
\mathbb{P} \bigl[\eta^\vee_t\ne\eta_t^\wedge
\bigr]\leq\varepsilon.
\]
\end{proposition}

The upper bound on the mixing time can then be deduced from \eqref
{gramsci} and \eqref{gramsci2}.

Our strategy to prove the result is the following:
it follows from Lemma~\ref{boudk} that\vspace*{1pt} after time $t_0:=\frac
{N^2}{2\pi^2}\log k(1+\delta/2)$,
we have
\[
A(t):=\sum_{x=1}^{N-1} \bigl(
\eta^\wedge_t- \eta_t^\vee\bigr)
(x) \ll k^{1/2} N,
\]
or in other words, that the area between the two curves is much smaller
than the typical fluctuation
of $\sum_{x=1}^N \eta(x)$ under the equilibrium measure $\mu$.

Then we want to use the extra time $t_1-t_0=\frac{N^2}{2\pi^2}\log
k(\delta/2)$
to make the two paths coalesce by comparing the evolution of the area
$A(t)$ (which is a supermartingale)
to a symmetric random walk
with a time change.

To perform this last step, we need to know that
both $P^{\vee}_{t_0}$ and $P^{\wedge}_{t_0}$ are close to equilibrium.
This fact is proved following the ideas developed in Section~\ref{proofmres}.
Then we use the fact that typically, in the interval $[t_0,t_1]$ both
$\eta^\wedge_t$ and $\eta^\vee_t$ present
a lot of flippable corners, and this allows us to produce enough
fluctuation for the two to coalesce with
large probability.

\subsection{Reaching equilibrium from the extremal conditions}

As a preliminary work we need to prove that $\eta^\vee_t$ and
$\eta^\wedge_t$ have reached their equilibrium distribution a bit
before $t_1$.

%
\begin{proposition}\label{topandbottom}
Set
\[
t_0:=\frac{N^2}{2\pi^2}\log k(1+\delta/2).
\]
We have for all $\varepsilon>0$, for all $N$ large enough,
%
%
\begin{eqnarray}
\lim_{N\to\infty} \bigl\llVert P_{t_0}^\wedge-
\mu\bigr\rrVert_{\mathrm{TV}}&=&0,
\nonumber
\\[-8pt]
\\[-8pt]
\lim_{N\to\infty} \bigl\llVert P_{t_0}^\vee-\mu
\bigr\rrVert_{\mathrm{TV}}&=&0.
\nonumber
\end{eqnarray}
\end{proposition}

The proof of this statement has a structure similar to that of the
proof of \eqref{uopbound} (the similar result for the AT shuffle) but
is slightly simpler.
One needs only two steps instead of three to make $\eta_t$ close to
equilibrium. Note that by symmetry, we only need to consider the
initial condition $\wedge$.

Let us quickly sketch the proof.
We set $K:=\lceil1/\delta\rceil$.

We consider a dynamic $\eta_t$ starting from the initial condition
$\wedge$ with the following censoring scheme:
\begin{itemize}
\item up to time $t_2:=\frac{N^2}{2\pi^2}\log k(1+\delta/4)$, we run
the dynamics without censoring;
\item in the time interval $[t_2,t_0]$, the updates at coordinate $x_i$
[recall \eqref{defxi}] are censored.
\end{itemize}

Let $\nu_t$ be the law of $\eta_t$ under this dynamics. According to
Proposition~\ref{censor}, we have
\[
\bigl\llVert P_t^\wedge-\mu\bigr\rrVert_{\mathrm{TV}}
\leq\llVert\nu_t-\mu\rrVert_{\mathrm{TV}},
\]
and hence it is sufficient to prove that $\nu_{t_0}$ is close to equilibrium,
or that for every $\varepsilon>0$ , if $N$ is large enough,
%
%
\begin{equation}
\llVert\nu_{t_0}-\mu\rrVert_{\mathrm{TV}}\leq\varepsilon.
\end{equation}
We prove that at time $t_2$ the skeleton $\bar\eta$ has come close to
its equilibrium distribution and use the time interval $[t_2,t_0]$
to put all the segments between skeleton points to equilibrium.

%
\begin{proposition}\label{skeletontime}
 We have for all $\varepsilon>0$, for all $N$ large enough,
%
%
\begin{equation}
\llVert\bar\nu_{t_2}-\bar\mu\rrVert_{\mathrm{TV}}\leq
\varepsilon/2.
\end{equation}
\end{proposition}

We prove Proposition~\ref{skeletontime} in the next section. Let us
now explain how we prove Proposition~\ref{topandbottom}.

\begin{pf*}{Proof of Proposition~\ref{topandbottom} using
Proposition~\ref{skeletontime}}
Between time $t_2$ and $t_0$, a consequence of the censoring is that
the number of particles
in the interval $(x_{i-1},x_{i}]$ remains constant for every $i\in\{
1,\dots,K\}$.
Hence on the time interval $[t_2,t_0]$, conditionally to $\eta_{t_2}$,
$(\eta_{t})_{t\geq t_2}$ is a product dynamics of $K$ independent
exclusion processes.
We denote the corresponding equilibrium measure by $\mu_{\eta
_{t_2}}$. We have
%
%
\begin{equation}
\mu_{\eta_{t_2}}:=\mu\bigl(\cdot\vert\forall i\in\{1,\dots,K-1\},
\eta(x_i)=\eta_{t_2}(x_i) \bigr).
\end{equation}

We define $k_i(\eta_{t_2})$ to be the number of particles in the
interval $(x_{i-1},x_{i}]$,
\[
k_i:=\eta_{t_2}(x_i)-\eta_{t_2}(x_{i-1})+
\frac{k}{N}(x_i-x_{i-1}).
\]

Using Proposition~\ref{wilson2} and the fact that the total variation
distance between product measures is smaller than the sum of the
total variation distances
of the marginals, we obtain, similar to \eqref{desdree}, that
%
%
\begin{equation}
\bigl\llVert\mathbb{P} [ \eta_{t_0}\in\cdot\vert
\eta_{t_2}]-\mu_{\eta_{t_2}} \bigr\rrVert_{\mathrm{TV}} \leq\sum
_{i=1}^K k_ie^{-\lambda_{\Delta x_i} (t_0-t_2)}.
\end{equation}
Then we use that $k_i\leq k$ for all $i$, and that if $N$ is large enough,
\[
\lambda_{\Delta x_i}=2 \biggl(1-\cos\biggl(\frac{\pi}{\Delta
x_i} \biggr)
\biggr) \geq\frac{\pi^2}{2 (\Delta x_i)^2}\geq\frac{\pi
^2}{3\delta
^2 N^2},
\]
to conclude that
%
%
\begin{equation}
\bigl\llVert\mathbb{P} [ \eta_{t_0}\in\cdot\vert
\eta_{t_2}]-\mu_{\eta_{t_2}} \bigr\rrVert_{\mathrm{TV}}\leq k K
e^{-\log k/(24\delta)}\leq\varepsilon/2.
\end{equation}
Even though the right-hand side above is a random variable, the inequality holds
not only with probability one, but also everywhere.
Using Jensen's inequality after taking the average on the event $\{\bar
\eta_{t_2}=\xi\}$,
we obtain that for every $\xi\in\bar\Omega_{N,k}$,
%
%
\begin{equation}
\label{crooom} \bigl\llVert\nu_{t_0}(\cdot\vert\bar\eta=\xi) -
\mu( \cdot\vert\bar\eta=\xi) \bigr\rrVert_{\mathrm{TV}} \leq
\varepsilon/2.
\end{equation}
Then similar to \eqref{ladecomm} we have
%
%
\begin{eqnarray}
\label{kokok}\llVert\nu_{t_0}-\mu\rrVert_{\mathrm{TV}}&\leq&
\llVert\bar\nu_{t_2}-\bar\mu\rrVert_{\mathrm{TV}}
\nonumber
\\[-8pt]
\\[-8pt]
&&{} +\sum_{\xi\in\Omega_{N,K}} \bar\nu_{t_
0}(\xi) \bigl
\llVert\nu_{t_0}(\cdot\vert\bar\eta=\xi) - \mu(\cdot\vert\bar
\eta=
\xi) \bigr\rrVert_{\mathrm{TV}} \leq\varepsilon,
\nonumber
\end{eqnarray}
where in the last inequality we used \eqref{crooom} and
Proposition~\ref{skeletontime}.
\end{pf*}

\subsection{Proof of Proposition \texorpdfstring{\protect\ref{skeletontime}}{8.3}}

The proof strongly relies on the fact that $\nu_{t_2}=P^\wedge_{t_2}$
is increasing and presents many similarities with the proof of
Proposition~\ref{groom2}.
Set
\[
v(\bar\eta):=\sum_{i=1}^{K-1} \bar\eta(i)
\]
to be the volume below the skeleton of $\eta$.
The idea is to show that once the expected volume $v(\bar\eta_t)$
becomes much smaller than its equilibrium fluctuations (which are of
order $K\sqrt{k}$), then we must be close to equilibrium.

%
\begin{lemma}\label{areaisall}
Let $ \nu$ be a probability measure whose density with respect $\mu$
is increasing.
For every $\varepsilon$, there exists $\delta(K,\varepsilon)$ such
that for $N$
sufficiently large, we have
%
%
\begin{equation}
\label{chrom} \nu\bigl(v(\bar\eta) \bigr)\leq(K-1)\sqrt
{k}\delta\quad
\Rightarrow\quad\llVert\bar\nu- \bar\mu\rrVert\leq\varepsilon/2.
\end{equation}
\end{lemma}

\begin{pf*}{Proof of Proposition~\ref{skeletontime} from Lemma~\ref
{areaisall}}
According to \eqref{groupeta} for \mbox{$t=t_2$}, we have
%
%
\begin{equation}
\qquad\bar\nu_{t_2} \bigl(v(\bar\eta) \bigr)\leq4 k e^{-\lambda
_N t_2}
=4 ke^{-(1+\delta/2)(1+\cos(\pi/N))N^2 \pi^{-2}\log k}\leq8
k^{1/2-\delta/4}.\vadjust{\goodbreak}
\end{equation}
Hence from Lemma~\ref{areaisall}, if $N$ is large enough [so that the
left-hand side of \eqref{chrom} is satisfied], then
\[
\llVert\bar\nu_t - \bar\mu\rrVert\leq\varepsilon/2.
\]
\upqed
\end{pf*}

Now to prove Lemma~\ref{areaisall}, all we need to do is to introduce
some notation. Given $A>0$,
we set
%
%
\begin{eqnarray}
\mathcal A_i&:=& \{ \eta\vert\bar\eta_i
\geq\sqrt{k}A \},
\nonumber
\\
\mathcal A&:= &\bigcap_{i=1}^{K-1} \mathcal
A_i = \bigl\{ \eta\vert\forall i\in\{1,\dots,K-1\}, \bar
\eta_i \geq\sqrt{k}A \bigr\} ,
\\
\mathcal B&:=& \Biggl(\bigcup_{i=1}^{K-1}
\mathcal A_i \Biggr)^c= \bigl\{ \eta\vert\forall i\in
\{1,\dots,K-1\}, \bar\eta_i < \sqrt{k}A \bigr\}.
\nonumber
\end{eqnarray}
Note that the $\mathcal A_i$s and $\mathcal A$ are increasing events
while $\mathcal B$ is decreasing.
With a slight abuse of notation, we also consider these sets as subsets
of $\bar\Omega_{N,K}$.

%
\begin{lemma}\label{grainconv}
When $N$ tends to infinity,
%
%
\begin{equation}
\biggl(\sqrt{\frac{N}{k(N-k)}}\eta_{x_i} \biggr)_{i\in
[0,K]}
\Rightarrow(Y_i)_{i\in[0,K]},
\end{equation}
where the $Y$ is a Gaussian process whose covariance function is given by
%
%
\begin{equation}
\mathbb{E} [Y_i Y_j \mathbf{1}_{\{i\leq j\}} ]:=
\frac
{i}{K} \biggl(1-\frac{j}{K} \biggr)\mathbf{1}_{i\leq j}.
\end{equation}
Given $\delta\in(0,1/2)$, we choose $A$ large enough, and $\delta
'(\delta)$ satisfying $\lim_{\delta\to0} \delta'=0$, such that for
all $N$
large enough,
%
%
\begin{eqnarray}
\label{ugly} %
\mu(\mathcal A) &\geq& \delta^{K-1}:=
\delta_1,
\nonumber
\\[-8pt]
\\[-8pt]
\mu(\mathcal B)& \geq&1-(K-1)\delta':=1-\delta_2.
\nonumber
\end{eqnarray}
\end{lemma}

\begin{pf}
This is just a simple consequence of the fact that $ (\sqrt{\frac
{N}{k(N-k)}}\*\eta_{\lceil Nx\rceil} )_{x\in[0,1]}$ converges in
law to a Brownian bridge:
the convergence of the finite dimensional marginals can be proved by
using Stirling's formula (which gives a local central limit theorem),
while the proof of tightness (in the topology of the uniform
convergence) is essentially the same as that for the proof of
convergence of random walk to Brownian motion.

The inequalities of \eqref{ugly} are proved similarly to \eqref{lesoufs}.
\end{pf}

\begin{pf*}{Proof of Lemma~\ref{areaisall}}
We are going to prove that for $N$ sufficiently large, the two
following implications hold:
%
%
\begin{equation}
\label{preum} \nu(\mathcal A)\geq(1+\alpha) \mu( \mathcal A) \quad
\Rightarrow
\quad\nu\bigl(v(\bar\eta) \bigr)\geq\delta_1 \alpha A \sqrt{k}
\end{equation}
and
%
%
\begin{equation}
\label{deuz} \nu(\mathcal A)\leq(1+\alpha)\mu(\mathcal{A})\quad
\Rightarrow
\quad\llVert\bar\nu-\bar\mu\rrVert\leq2\alpha+\delta_2.
\end{equation}

We start with \eqref{preum}. Similar to \eqref{froome2}, for all
$i\in\{1,\dots,K-1\}$, we can prove using the correlation inequality
(Lemma~\ref{correlineq} and the fact that
$\bar\nu_i/\bar\mu_i$ is increasing; cf. Proposition~\ref{projecmonsep})
%
%
\begin{equation}
\label{etaij} \nu\bigl(\bar\eta(i) \bigr)\geq(\nu-\mu)
(\mathcal
A_i)\mu( \bar\eta_i \vert\mathcal A_i) +(
\nu-\mu) \bigl(\mathcal A_i ^c \bigr)\mu\bigl( \bar
\eta_i \vert\mathcal A_i^c \bigr).
\end{equation}

As $\nu$ stochastically dominates $\mu$,
$\nu(\mathcal A_i)\geq\mu(\mathcal A_i)$. Furthermore $\mu( \bar
\eta_i \vert \mathcal A_i)\geq A\sqrt{k}$ and
$\mu( \bar\eta_i \vert \mathcal A_i^c)\leq0$, and hence \eqref
{etaij} implies
%
%
\begin{equation}
\nu(\bar\eta_i)\geq(\nu-\mu) (\mathcal A_i) A \sqrt{k}.
\end{equation}
Summing over $i$ we get
%
%
\begin{equation}
\label{volume} \nu\bigl(v(\bar\eta) \bigr)\geq\sum
_{i=1}^{K-1}( \nu-\mu) (\mathcal A_i) A
\sqrt{k}.
\end{equation}
Then we remark that
\[
\Theta\dvtx\eta\mapsto\sum_{i=1}^K
\mathbf{1}_{\mathcal A_i}(\eta)-\mathbf{1} _\mathcal A(\eta)
\]
is an increasing function, and FKG inequality \eqref{FKG} applied to
$\Theta$ and $\nu/\mu$ gives
%
%
\begin{equation}
\label{primoous} \sum_{i=1}^{K-1}(\nu-\mu) (
\mathcal A_i)\geq(\nu-\mu) (\mathcal A).
\end{equation}
Combining \eqref{volume} with \eqref{primoous} and \eqref{ugly}, we
obtain \eqref{preum}.

For \eqref{deuz} we note that, similar to \eqref{graf},
if
$\bar\nu(\mathcal{A})\leq(1+\alpha)\bar\mu( \mathcal A) $ we can prove,
using the fact that $\frac{\bar\nu}{\bar\mu}$ is an increasing function,
%
%
\begin{equation}
\label{graf2} \forall\bar\eta\in\mathcal{B} ,\qquad\frac{\bar
\nu}{\bar\mu
}(\bar
\eta)\leq\frac{\bar\nu(\mathcal{A})}{\bar\mu(\mathcal{A})}\leq
1+\alpha.
\end{equation}
Now note that if $\eta\in\mathcal A$ and $\eta' \in\mathcal B^c$,
then $\operatorname{\mathbf{min}}(\eta, \eta') \in\mathcal B^c$,
and hence from \eqref{Holley}
we have
%
%
\begin{equation}
\label{graf3} \frac{\nu}{\mu} \bigl(\mathcal B^c \bigr)= \mu
\biggl( \frac{\nu}{\mu} \Big\vert\mathcal{B} ^c \biggr) \leq\mu
\biggl( \frac{\nu}{\mu} \Big\vert\mathcal{A} \biggr)= \frac{\nu
}{\mu}(\mathcal
A)\leq1+\alpha.
\end{equation}

Then combining \eqref{graf2} and \eqref{graf3}, we have
%
%
\begin{eqnarray}
\llVert\bar\nu-\bar\mu\rrVert&=& \int_{\mathcal B^c} \biggl(
\frac{\bar\nu}{\bar\mu}(\bar\eta)-1 \biggr)_+ \bar\mu(\mathrm
{d}\bar\eta) +\int
_{\mathcal B} \biggl(\frac{\bar\nu}{\bar\mu}(\bar\eta)-1 \biggr
)_+\bar\mu
(\mathrm{d}\bar\eta)
\nonumber
\\[-8pt]
\\[-8pt]
&\leq&\nu\bigl(\mathcal B^c \bigr)+\alpha\nu(\mathcal B) \leq
\alpha+(1+\alpha)\delta_2.
\nonumber
\end{eqnarray}
\upqed
\end{pf*}

\subsection{Coupling the top and the bottom in a Markovian manner:
Proof of Lemma~\texorpdfstring{\protect\ref{wegonosqueez}}{8.1}}

The idea of the proof is to say that after time $t_0$,
the area between the two curves shrinks to $0$ in a time of order $N^2$.
This statement cannot be proved only by computing the expectation of
the area,
and one must try to control its fluctuations.

Recall that we denote by
\[
A(t):=\sum^N_{x=0} \bigl(
\eta^{\wedge}_{t-}-\eta^{\vee}_{t}(x) \bigr)
\]
the area between the two curves.

Our strategy is to couple $A(t)$ together with a symmetric random walk.
To do this we need to introduce some notation and an alternative way to
build the dynamics.
We say that $x$ is an active coordinate [and write $x\in C(t)$] if
\[
\exists y\in\{ x-1,x,x+1\},\qquad\eta^\wedge_t(y)>
\eta^{\vee}_t(y)
\]
and that $(x,z)$ is an active point for $\eta^\wedge_t$ (or $\eta
^{\vee}$)
if $x$ is active and $\eta^\wedge_t(x)=z$ (or $\eta^{\vee}$)
corresponds to a local extremum.

%
\begin{figure}

\includegraphics{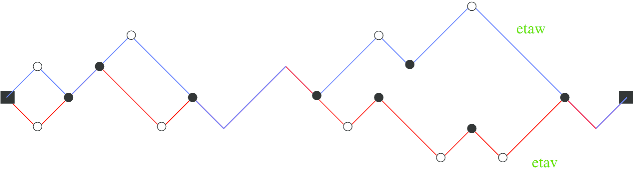}

\caption{Graphical representation of the dynamics
$(\eta^\wedge_t,\eta^\vee_t)$.
The path $\eta^\wedge_t$ is represented in blue and $\eta^\vee_t$
in red.
The active points are represented by circles: full circles for points
in $U(t)$ and void circles for points in $D(t)$. Squares represent the
fixed ends of the lattice paths.
The area between the two paths is made of three bubbles.}
\label{bubbles}
\end{figure}

Among active points, in the following, we specify those that allow an
increase of the area and those that allow the area to decrease:
%
%
\begin{eqnarray}
U(t)&:=& \bigl\{ (x,z) \vert x \in C(t), \eta_t^{\wedge}(x)=z
\mbox{ is a local minimum} \bigr\}
\nonumber
\\
&&{} \cup\bigl\{ (x,z) \vert x \in C(t), \eta_t^{\vee}(x)=z
\mbox{ is a local maximum} \bigr\},
\nonumber
\\
[-8pt]
\\[-8pt]
D(t)&:=& \bigl\{ (x,z) \vert x \in C(t), \eta_t^{\vee}(x)=z
\mbox{ is a local minimum} \bigr\}
\nonumber
\\
&& {}\cup\bigl\{ (x,z) \vert x \in C(t), \eta_t^{\wedge}(x)=z
\mbox{ is a local maximum} \bigr\}.
\nonumber
\end{eqnarray}

We refer to Figure~\ref{bubbles} for a graphical representation of
$U(t)$ and $D(t)$.
We denote by $u(t)$ and $d(t)$ the respective cardinals of $U(t)$ and $D(t)$.
They are the rates at which $A(t)$ increase and decrease respectively.
The reader can check that
\[
(d-u) (t)\in\{0,1,2\},
\]
and hence that $A(t)$ is a supermartingale.

Given a sequence of i.i.d. exponentials $(e_n)_{n\geq0}$ and a
Bernoulli sequence of parameters $1/2$, $(V_n)_{n\geq0}$,
we can reconstruct the dynamics $(\eta_t^{\wedge},\eta_t^{\vee
})_{t\geq t_0}$ (note that we start from time $t_0$ instead of $0$)
as follows:
\begin{enumerate}[$\bullet$]
\item[$\bullet$] The updates of nonactive coordinates [for which
$(\eta^{\wedge
},\eta^{\vee})$ are moving together]
are performed with appropriate rate independently of $e$ and $V$; note that
these updates do not change the value of
$U$ and $D$.
\item[$\bullet$] The updates of active coordinates are performed
using $e$ and $V$
in the following manner.
After the $(n-1)$th update of an active coordinate (which occurred say
at time $t$), we wait a time $e_n/(u(t)+d(t))$ [at time $t_0$ we wait a
time $e_1/(u(t_0)+v(t_0))$],
and then:
\begin{longlist}[(2)]
\item[(1)] if $V_n=-1$, we choose an active point uniformly at random
in $D(t)$
and flip the corresponding corner in either $\eta^{\wedge}$ or $\eta
^{\vee}$;
\item[(2)] if $V_n=1$, then with probability $\frac{d-u}{d+u}(t)$
we choose a corner of $D(t)$ uniformly at random and flip it,
and with probability $\frac{2u}{d+u}(t)$, we switch a corner of $U(t)$.
\end{longlist}
\end{enumerate}

Note that after finitely many updates of active coordinates, $\eta
^\vee(t)$ and $\eta^\wedge(t)$ merge so that only a finite
number of $(V_n)_{n\geq0}$ is used. We let $\mathcal{N} $ be the last
one which
is used.
We define $W_n$ to be equal to $-1$ if the transition corresponding to
$V_n$ decreases the area and
$+1$ if it increases it. From our construction $W_n\leq V_n$, whenever
$W_n$ is defined.

Let $(\widetilde S(t))_{t\geq0}$ be the random walk starting from $A(t_0)$
whose waiting times are given by $e$,
and increments are given by $W_n$, or in other words,
%
%
\begin{equation}
\label{construw} %
\widetilde S_t= %
\cases{
\displaystyle A(t_0)+\sum_{n=1}^N
W_n &\quad  if $\displaystyle\sum_{n=1}^N
e_n \leq t < \sum_{n=1}^{N+1}
e_n, n\leq\mathcal{N} -1$,
\cr
0 &\quad  if $\displaystyle t\geq\sum
_{n=1}^{\mathcal{N} } e_n$. } %
\end{equation}
This process is just a time changed version of $A(t+t_0)$.
We have
%
%
\begin{equation}
\label{timench} A(t+t_0)=S \biggl(\int_0^t
\bigl(d(s)+u(s) \bigr) \,\mathrm{d}s \biggr).
\end{equation}

We define also a set of stopping times for $\widetilde S$ for $i\geq2$,
%
%
\begin{eqnarray}
\tau_i&:=&\min\bigl\{ t\geq0 \vert\widetilde S(t)\leq
k^{1/2-(i+1)\varepsilon}N \bigr\},
\nonumber
\\[-8pt]
\\[-8pt]
\tau_{\infty}&:=&\min\bigl\{ t\geq0 \vert\widetilde S(t)=0 \bigr
\}.
\nonumber
\end{eqnarray}

%
\begin{lemma}\label{leptitauis}
If $\varepsilon\leq\delta/100$, we have, w.h.p.:
\begin{itemize}[(iii)]
\item[(i)] $\tau_2=0$;
\item[(ii)] for all $ i\in\{2 ,\dots, \lceil1/(2\varepsilon)
\rceil\}$,
\[
\tau_{i+1}-\tau_{i}\leq k^{1-(2i+1)\varepsilon}N^2;
\]
\item[(iii)] $\tau_{\infty}-\tau_{\lceil1/(2\varepsilon) \rceil
+1}\leq N^2$.
\end{itemize}
\end{lemma}

\begin{pf}
Item (i) is a consequence of Proposition~\ref{boudk} applied to $t=t_0$.
The two other items follow from the fact that for each $i$,
$(\widetilde
S_{t+\tau_i}-\widetilde S_{\tau_i})_{t\geq0}$ is dominated by a simple
random walk: the coupling is obtained by replacing $W$ with $V$ in
\eqref{construw}.
Then we just have to use the fact that for a simple random walk $X_t$
on $\mathbb{Z} $ starting from the origin and with jump rate $1$,
\[
\lim_{N \to\infty} \mathbb{P} \bigl[ \inf\bigl\{t \vert
X_t\leq N k^{1/2-(i+1)\varepsilon} \bigr\}\geq
N^2k^{1-(2i+1)\varepsilon}
\bigr]=0.
\]
\upqed
\end{pf}

Now we define
%
%
\begin{eqnarray}
\tau'_i&:=&\min\bigl\{ t\geq0 \vert
A(t+t_0)\leq k^{1/2-(i+1)\varepsilon
}N \bigr\},
\nonumber
\\[-8pt]
\\[-8pt]
\tau'_{\infty}&:=&\min\bigl\{ t\geq0 \vert
A(t+t_0)=0 \bigr\}.
\nonumber
\end{eqnarray}
We have from \eqref{timench},
\[
\tau_{i+1}-\tau_{i}=\int_{\tau'_i}^{\tau'_{i+1}}
(d+u) (t) \,\mathrm{d}t.
\]

We want to use this fact and Lemma~\ref{leptitauis} to show that w.h.p.
$\tau'_{\infty}$ is not too large.
In fact we already have from the last item of Lemma~\ref{leptitauis}
and \eqref{timench} that w.h.p.
%
%
\begin{equation}
\label{linfff} \tau'_{\infty}-\tau'_{\lceil1/(2\varepsilon)
\rceil+1}
\leq N^2
\end{equation}
and $\tau_0=0$.
Hence we only have to consider the increments $\tau'_{i+1}-\tau
'_{i}$, $0\leq i \leq\lceil1/(2\varepsilon) \rceil$.

%
\begin{lemma}\label{ladique}
We have
%
%
\begin{equation}
\lim_{N\to\infty} \mathbb{P} \bigl[ \exists i \in\bigl\{2, \dots,
\bigl\lceil1/(2\varepsilon) \bigr\rceil\bigr\}, \tau'_{i+1}-
\tau'_{i}\geq N^2 \bigr]=0.
\end{equation}
\end{lemma}

\begin{pf*}{Proof of Proposition~\ref{wegonosqueez}}
By definition, for any $t\geq0$ we have
%
%
\begin{equation}
\mathbb{P} \bigl[ \eta_{t+t_0}^{\wedge}\neq\eta_{t+t_0}^{\vee
}
\bigr]=\mathbb{P} \bigl[\tau'_\infty> t \bigr].
\end{equation}
From Lemma~\ref{ladique} and \eqref{linfff} we have
%
%
\begin{equation}
\lim_{N\to\infty} \mathbb{P} \bigl[\tau'_\infty
\geq\bigl\lceil1/(2\varepsilon) \bigr\rceil N^2 \bigr]=0.
\end{equation}
From this and \eqref{gramsci}, we can deduce that for any $\varepsilon
\leq
\delta/100$, if $N$ is large enough and such that
\[
t_1\leq t_0+ \bigl\lceil1/(2\varepsilon) \bigr\rceil
N^2,
\]
then we have
\[
d^{N,k}(t_1)\leq d^{N,k} \bigl(t_0+
\bigl\lceil1/(2\varepsilon) \bigr\rceil N^2 \bigr)<\varepsilon.
\]
\upqed
\end{pf*}

To prove Lemma~\ref{ladique}, we need a reasonable lower bound on
$(d+u)(t)$ in the interval
$[\tau'_{i}-\tau'_{i+1})$. To this end, we define a good set of
paths, for which there are sufficiently many active points.

We define $\mathcal{H}$ to be the set of bad paths that we wish to avoid
%
%
\begin{eqnarray}
\qquad \mathcal{H}&=&\mathcal{H}(k,N)\nonumber
\\
&:=& \Bigl\{ \eta\in\Omega_{N,k}
\big\vert\max
_{x\in[0,N]} \bigl\llvert\eta(x) \bigr\rrvert\geq\sqrt k\log k
\Bigr\}\nonumber
\\[-8pt]\\[-8pt]
&&{}\cup\biggl\{ \eta\in\Omega_{N,k} \Big\vert\exists x\in\biggl[0,
N- 2
\frac{N} k (\log k)^2 \biggr],\nonumber\\
&&\hspace*{60pt} \eta_{\vert [x,x+2(N/k) (\log k)^2]}
\mbox{ is
affine} \biggr\}.
\nonumber
\end{eqnarray}

We show first that most of the time, after $t_0$, both $\eta^\wedge
_t$ and $\eta^\vee_t$ stay out of $\mathcal{H}$.

%
\begin{lemma}\label{yapadeh}
We have
\[
\lim_{N\to\infty} \mu(\mathcal{H})=0,
\]
and as a consequence,
%
%
\begin{equation}
\label{cromium} \lim_{N\to\infty} \mathbb{P} \biggl[ \biggl(\int
_{t_0}^{t_0+\lceil
1/(2\varepsilon) \rceil N^2} \mathbf{1} \bigl\{\eta^{\wedge}_t
\in\mathcal{H}\mbox{ or } \eta^{\vee}_t \in\mathcal{H}
\bigr\}\,\mathrm{d}t \biggr) \geq N^2/2 \biggr]=0.
\end{equation}
\end{lemma}

\begin{pf}
The fact that
%
%
\begin{equation}
\lim_{N\to\infty}\mu\Bigl( \max_{x\in[0,N]} \bigl
\llvert\eta(x) \bigr\rrvert\geq\sqrt k\log k \Bigr)=0
\end{equation}
follows from the convergence of $ (\sqrt{\frac{N}{k(N-k)}}\eta
_{\lceil Nx\rceil} )_{x\in[0,1]}$ to the Brownian bridge; see
the proof Lemma~\ref{grainconv}.
For the second point it is sufficient to prove that w.h.p., each segment
\[
\biggl[(i-1)\frac{N} k (\log k)^2; i \frac{N} k (
\log k)^2 \biggr],\qquad i \in\bigl\{0,\dots, \bigl\lfloor k (\log
k)^{-2} \bigr\rfloor\bigr\}
\]
contains at least one particle and one empty site.

The probability for a segment of with $l$ sites ($l\leq N-k$) to contain
no particle is equal to
\[
\frac{(N-k)!(N-l)!}{(N-l-k)!N!}\leq\biggl(1-\frac{k}{N} \biggr)^l.
\]
Here $l\geq N k (\log k)^2/2$, and hence the probability is smaller than
$e^{-(\log k)^2/2}$.
As $k\leq N/2$ the probability of having a segment with no empty sites
is smaller than having a segment with no particle, and we can conclude.
Hence by union bound, after summing the probability of the two events
over all the segments, we obtain
%
%
\begin{eqnarray}
&&\mathbb{P} \biggl[ \exists x\in\biggl[0, N- 2\frac{N} k (\log
k)^2 \biggr], \eta_{\vert [x,x+2\sklfrac{N} k (\log k)^2]} \mbox{ is
affine} \biggr]
\nonumber
\\[-8pt]
\\[-8pt]
&&\qquad\leq k ( \log k)^{-2}e^{-(\log k)^2/2}.
\nonumber
\end{eqnarray}

Now let us deduce \eqref{cromium}.
Of course by symmetry it is sufficient to prove that
%
%
\begin{equation}
\label{rcrb} \lim_{N\to\infty} \mathbb{P} \biggl[ \biggl(\int
_{t_0}^{t_0+\lceil
1/(2\varepsilon) \rceil N^2} \mathbf{1} \bigl\{\eta^{\wedge}_t
\in\mathcal{H} \bigr\}\,\mathrm{d}t \biggr) \geq N^2/4 \biggr]=0.
\end{equation}
First, note that as $\mu$ is stable for the dynamics, we have
%
%
\begin{equation}
\mu\biggl(\mathbb{E} \biggl[\int_{0}^{\lceil1/(2\varepsilon)
\rceil N^2}
\mathbf{1} \bigl\{ \eta^{\xi}_t\in\mathcal{H} \bigr\} \,
\mathrm{d}t \biggr] \biggr)=\mu(\mathcal{H}) \bigl\lceil
1/(2\varepsilon) \bigr
\rceil N^2,
\end{equation}
where $\mu$ is the law of $\xi$.
Hence from the first point and the Markov inequality,
we have
%
%
\begin{equation}
\label{rcra} \lim_{N\to\infty} \mu\biggl(\mathbb{P} \biggl[
\biggl(\int_{0}^{\lceil
1/(2\varepsilon) \rceil N^2} \mathbf{1} \bigl\{
\eta^{\xi}_t\in\mathcal{H} \bigr\}\,\mathrm{d}t \biggr)
\geq N^2/4 \biggr] \biggr)=0.
\end{equation}
The quantity we want to estimate is equal (by the Markov property) to
\[
P^{\wedge}_{t_0} \biggl(\mathbb{P} \biggl[ \biggl(\int
_{0}^{\lceil
1/(2\varepsilon)
\rceil N^2} \mathbf{1} \bigl\{\eta^{\xi}_t
\in\mathcal{H} \bigr\}\,\mathrm{d}t \biggr) \geq N^2/4 \biggr]
\biggr)
\]
and hence
%
%
\begin{eqnarray}
\label{rcrc} && \biggl\llvert\mu\biggl(\mathbb{P} \biggl[ \biggl
(\int
_{0}^{\lceil
1/(2\varepsilon)
\rceil N^2} \mathbf{1} \bigl\{\eta^{\xi}_t
\in\mathcal{H} \bigr\} \,\mathrm{d}t \biggr) \geq N^2/4 \biggr]
\biggr)
\nonumber
\\
&&\quad{} - \mathbb{P} \biggl[ \biggl(\int_{t_0}^{t_0+\lceil
1/(2\varepsilon)
\rceil N^2}
\mathbf{1} \bigl\{\eta^{\wedge}_t\in\mathcal{H} \bigr\} \,
\mathrm{d}t \biggr) \geq N^2/4 \biggr] \biggr\rrvert
\\
&&\qquad\leq\bigl\llVert\mu-P^{\wedge}_{t_0} \bigr\rrVert
_{\mathrm{TV}}.
\nonumber
\end{eqnarray}
By Proposition~\ref{topandbottom} the right-hand side above converges to
zero, and hence \eqref{rcrb} is a consequence of \eqref{rcra} and
\eqref{rcrc}.
\end{pf}

The following result shows that indeed if both $\eta^{\wedge}_t$ and
$\eta^{\vee}_t$ lie outside of $\mathcal{H}$, then there are many
active sites.

%
\begin{lemma}
For all $ i\in\{2 ,\dots, \lceil1/(2\varepsilon)
\rceil\}$,
if $t< \tau'_{i+1}$, $\eta^{\wedge}_t\notin\mathcal{H}$ and $\eta
^{\vee
}_t \notin\mathcal{H}$,
%
%
\begin{equation}
(d+u) (t)\geq\frac{k^{1-(i+2)\varepsilon}}{8(\log k)^{2}}.
\end{equation}
\end{lemma}

\begin{pf}
If $\eta^{\wedge}_t\notin\mathcal{H}$ and $\eta^{\vee}_t \notin
\mathcal{H}$, then
\[
\max_{x\in[0,N]} \bigl(\eta_t^{\wedge}-
\eta^{\vee}_t \bigr)\leq2\sqrt k \log k.
\]
If $t < \tau'_{i+1}$, we also have
\[
A(t)\geq k^{1/2-(i+2)\varepsilon}N.
\]
Combining these two inequalities we have
%
%
\begin{equation}
\label{alldabubl} \# \bigl\{ x\in\{1,\dots,N-1\} \vert\eta
^{\wedge}(x)>
\eta^{\vee
}_t(x) \bigr\}\geq N k^{-(i+2)\varepsilon}(2\log
k)^{-1}.
\end{equation}
Now the set of coordinates where $\eta^{\wedge}_t$ and $\eta^{\vee
}_t$ differ can be decomposed into maximal connected components
(for the usual graph structure on $\mathbb{Z} $), each component
corresponding to a ``bubble'' between $\eta^{\wedge}_t$ and $\eta
^{\vee}_t$; see Figure~\ref{bubbles}.

If $\{x_1, \dots, x_2\}$ corresponds to a bubble, then all the corners
of $\eta^{\wedge}_t$ and $\eta^{\vee}_t$
in the interval $\{x_1, \dots, x_2\}$ are active points. In particular
we have at least two active points per bubble.
We also need to show that long bubbles (i.e., those associated to long
intervals) have a lot of active points.

Note that the interval $\{x_1, \dots, x_2\}$ can be split into
\[
\biggl\lfloor\frac{(x_1-x_2)k}{2 N \log k} \biggr\rfloor
\]
intervals of length $\frac{2 N \log k}{k}$ or longer (not that it
might be zero).
If $\eta^{\wedge}_t\notin\mathcal{H}$, then each of these intervals will
contain at least one active coordinate.
Hence if $\eta_t^\wedge\notin\mathcal{H}$, the number of active
points in a
bubble in the interval $\{x_1, \dots, x_2\}$ is always larger than
\[
\frac{(x_1-x_2)k}{4 N \log k}.
\]

Note that the number has been chosen so that the statement is also
valid when $\lfloor\frac{(x_1-x_2)k}{2 N \log k}\rfloor=0$.

Summing over all bubbles and using \eqref{alldabubl}, we obtain the
following lower bound for the total number of active sites:
\[
(d-u) (t)\geq\frac{k^{1-(i+2)\varepsilon}}{8(\log k)^{2}}.
\]
\upqed
\end{pf}

\begin{pf*}{Proof of Lemma~\ref{ladique}}
It is sufficient that to prove that for each $i \in\{2, \dots, \lceil
1/(2\varepsilon) \rceil\}$, the probability of the event
\[
\mathcal{A}_i:= \bigl\{ \tau'_{i+1}-
\tau'_{i}\geq N^2 \bigr\} \cap\bigl\{
\forall j <i, \tau'_{i+1}-\tau'_{i}<
N^2 \bigr\}
\]
is vanishing.
Note that if the event $\mathcal{A}_i$ occurs, we have
%
%
\begin{eqnarray}
\tau_{i+1}-\tau_i&\geq&\int_{\tau'_i}^{\tau'_i+N^2}
(d+u) (t)\,\mathrm{d}t
\nonumber
\\
&\geq&\frac{k^{1-(i+2)\varepsilon}}{8(\log k)^{2}}\int_{\tau
'_i}^{\tau
'_i+N^2}
\mathbf{1}_{\{\eta^{\wedge}_t \notin\mathcal{H}\mathrm{\ and \ } \eta^{\vee
}_t \notin\mathcal{H}\}} \,\mathrm{d}t
\\
&\geq&\frac{k^{1-(i+2)\varepsilon}}{8(\log k)^{2}} \biggl(N^2-\int
^{\lceil
1/(2\varepsilon) \rceil N^2}_0
\mathbf{1}_{\{\eta^{\wedge}_t\in \mathcal{H} \mathrm{\ or\ } \eta
^{\vee}_t  \in\mathcal{H}\}} \,\mathrm{d}t \biggr).
\nonumber
\end{eqnarray}
According to Lemma~\ref{yapadeh}, w.h.p., the last factor on the right-hand side
is larger than $N^2/2$, and hence w.h.p.,
%
%
\begin{equation}
(\tau_{i+1}-\tau_i)\mathbf{1}_{\mathcal{A}_i} \geq
\frac{N^2
k^{1-(i+2)\varepsilon
}}{16(\log k)^{2}}.
\end{equation}

Hence $\mathcal{A}_i$ has to occur with vanishing probability, or else we
would have a contradiction to Lemma~\ref{leptitauis}.
\end{pf*}

%

%
\begin{appendix}
\section{Proof of technical results}\label{app}

\subsection{Proof of the FKG inequality for permutations}\label{afkg}

We prove that for any pair $(A,B)$ of increasing sets, we have
%
%
\begin{equation}
\label{permut} \mu(A\cap B)\geq\mu(A)\mu(B).
\end{equation}
Then we can deduce the inequality for functions as follows. Given $f$
and $g$ two increasing positive functions (there is no loss of
generality in assuming positivity as adding a constant to $f$ or $g$
leaves the inequality unchanged) and $x,y\in\mathbb{R} $, we define the
increasing sets
\[
A_s= \bigl\{f(\sigma)\geq s \bigr\} \quad\mbox{and}\quad
B_t:= \bigl\{g(\sigma)\geq t \bigr\}.
\]
As $f=\int_{\mathbb{R} _+} A_x\,\mathrm{d}x$,
we can deduce from \eqref{permut} that
%
%
\begin{eqnarray}
\mu\bigl(f(\sigma)g(\sigma) \bigr)&=&\mu\biggl( \int_{\mathbb
{R} _+^2}
\mathbf{1} _{A_s}\mathbf{1}_{B_t}\,\mathrm{d}s \,\mathrm{d}t
\biggr) \geq\int_{\mathbb{R} _+^2} \mu(A_s)
\mu(B_t)\,\mathrm{d}x \,\mathrm{d}y
\nonumber
\\[-8pt]
\\[-8pt]
&=& \mu\bigl(f(\sigma) \bigr)\mu\bigl(g( \sigma) \bigr).
\nonumber
\end{eqnarray}

Let us now prove \eqref{permut}. Let $A$ and $B$ be two increasing
subsets of $S_N$.
Let us start from the identity and run two coupled dynamics $\sigma_t$
and $\sigma^A_t$ defined as follows:
$\sigma_t$ is a normal AT shuffle,
and $\sigma^A_t$ has the same transition rule, except that all the transitions
going out of $A$ are canceled (this is called the \textit{reflected
Markov chain}). We couple the two dynamics using the graphical
construction of Section~\ref{graphix},
with both dynamics using the same clock processes $\mathcal{T} $ and update
variables $U$, the only difference being that $\sigma^A_t$ cancels
the transition that makes it go
out of $A$.

The Markov chain $\sigma^A_t$ is irreducible: the reason for this is
that for each \mbox{$(\sigma, \sigma') \in A^2$} one can always find a
sequence of
up transitions (corresponding to sorting neighbors) from $\sigma$
leading to $\mathbf{1}$ (the identity) and a sequence of down transitions
going from $\mathbf{1}$ to $
\sigma'$. The concatenation of these two sequences provides a path of
transitions from $\sigma$ to $\sigma'$ whose steps are all in $A$
(they are $\geq\sigma$ in the first half and $\geq\sigma'$ in the
second half).
The reader can check that $\mu(\cdot\vert A)$ (i.e., the uniform
measure on $A$) is reversible for $\sigma^A$ (this in fact a general
statement for reflected Markov chain)
and hence that the distribution of $\sigma^A_t$ converges to it.

As the only transitions which are canceled for $\sigma^A$ are those
transitions ``going down'' (corresponding to reverse-sorting of an
adjacent pair), we have
(as a consequence of the proof of Proposition~\ref{orderpreserv})
\[
\forall t\geq0,\qquad\sigma^A_t\geq
\sigma_t.
\]
Using Lemma~\ref{limitlema} we obtain that
%
%
\begin{equation}
\mu(\cdot\vert A)\succeq\mu,
\end{equation}
and we conclude by taking expectation over $B$ for these two measures.

\subsection{Proof of the censoring inequality for permutations}\label{acensor}

To use the censoring inequality, and also to prove it,
we have to work with increasing probability measures.
A key result is that those measures are conserved by the dynamics
(censored and uncensored) in the following sense:

%
\begin{proposition}\label{proppres}
Let $\nu$ be an increasing probability measure on $S_N$.
Then for every $t\geq0$, $P^\nu_t$ is also increasing and for any
censoring scheme,
$P_t^{\nu,\mathcal C}$ is increasing.
\end{proposition}

The strategy to prove such a statement is to show first that each
individual update does not alter monotonicity, and then to average on
the different possibilities
for the chain of updates given by the clock process.

Given $x\in\{1,\dots,N-1\}$, $\sigma\in S_N$, we set
\[
\sigma_x^{\bullet}:= \bigl\{ \xi\in S_N \vert
\forall y\notin\{x,x+1\}, \xi(y)=\sigma(y) \bigr\}.
\]
The set $\sigma_x^{\bullet}$ contains two elements (one of which is
$\sigma$)
$\sigma_x^+\geq\sigma_x^-$, which are obtained respectively by
sorting and reverse sorting
$\sigma(x)$ and $\sigma(x+1)$.
Given $\nu$ a probability measure on $S_N$, one defines $\theta_x(\nu
)$, the measure ``updated at $x$'' as follows:
%
%
\begin{equation}
\theta_x(\nu) (\sigma):=\nu\bigl(\sigma_x^{\bullet}
\bigr)/2.
\end{equation}
The operator $\theta_x$ describes how the law of $\sigma_t$ is
changed when the clock-process rings at $x$.

%
\begin{lemma}\label{conservinc}
If $\nu$ is increasing, so is $\theta_x(\nu)$ and furthermore $\nu
\succeq\theta_x(\nu^x)$.
\end{lemma}

\begin{pf}
If $\sigma\geq\xi$,
the reader can check that $\sigma_x^+\geq\xi_x^+$ and $\sigma
_x^-\geq
\xi_x^-$.
Hence
%
%
\begin{equation}
\nu\bigl(\sigma_x^{\bullet} \bigr)=\nu\bigl(
\sigma_x^+ \bigr)+\nu\bigl(\sigma_x^- \bigr) \geq\nu
\bigl( \xi_x^+ \bigr)+\nu\bigl(\xi_x^- \bigr)=\nu\bigl(
\xi_x^{\bullet} \bigr),
\end{equation}
and thus $\theta_x(\nu)$ is increasing if $\nu$ is increasing.

Let $g$ be an increasing function.
If $\nu$ is increasing, then we have $\nu(\sigma^+_x)\geq\nu
(\sigma
^-_x)$ and hence
%
%
\begin{eqnarray}
\qquad g \bigl(\sigma_x^+ \bigr)\nu\bigl(\sigma^+_x
\bigr)+ g \bigl(\sigma_x^- \bigr)\nu\bigl(\sigma^-_x
\bigr) &\geq& \bigl(g \bigl(\sigma_x^+ \bigr)+g \bigl(
\sigma_x^- \bigr) \bigr)\frac{\nu^x(\sigma_x^+)+\nu
^x(\sigma_x^-)}{2}
\nonumber
\\[-8pt]
\\[-8pt]
&=& g \bigl(\sigma_x^+ \bigr)\theta_x(\nu) \bigl(
\sigma_x^+ \bigr)+ g \bigl(\sigma_x^- \bigr)\theta
_x(\nu) \bigl(\sigma_x^- \bigr).
\nonumber
\end{eqnarray}
Summing over all $\sigma\in S_N$ and dividing by two, one obtains
\[
\nu(g)\geq\theta_x(\nu) (g).
\]
As $g$ is arbitrary, this implies
\[
\nu\succeq\theta_x(\nu).
\]
\upqed
\end{pf}

\begin{pf*}{Proof of Proposition~\ref{proppres}}
Let $\nu$ be an increasing probability and $\sigma^\nu_t$ be the
Markov chain trajectory obtained with
the graphical construction. By definition we have
%
%
\begin{equation}
P^{\nu}_t=\mathbb{P} \bigl[\sigma^\nu_t
\in\cdot\bigr].
\end{equation}

Let $\mathcal{N} $ denote the number of updates which have occurred
before time
$t$ and
$X_{1}, \dots, X_{\mathcal{N} }$ denote the sequence of vertices that have
rung for the clock process (with repetitions).
Then the probability law $\mathbb{P} [\sigma^\nu_t\in\cdot\vert
\mathcal{T} ]$,
knowing the clock process is given by
\[
\theta_{X_{\mathcal{N} }}\circ\cdots\circ\theta_{X_1}(\nu),
\]
is increasing according to Lemma~\ref{conservinc}.
The monotonicity is then conserved when averaging with respect to
$\mathcal{T} $.
The reasoning remains valid for the censored dynamics.
\end{pf*}

We end the preparation of the proof with two additional lemmas on monotonicity.
The first is simply a consequence of the graphical construction of
Section~\ref{graphix}.

%
%
\begin{lemma}\label{microdil}
Updates preserve stochastic domination in the sense that if \mbox{$\nu
_1\succeq\nu_2$}, then
\[
\theta_x(\nu_1)\succeq\theta_x(
\nu_2).
\]
\end{lemma}

%
\begin{lemma}\label{increden}
If $\nu_1$ has an increasing density and $\nu_1\preceq\nu_2$, then
\[
\llVert\nu_1-\mu\rrVert_{\mathrm{TV}}\leq\llVert
\nu_2-\mu\rrVert_{\mathrm{TV}}.
\]
\end{lemma}

\begin{pf}
Set
\[
A:= \bigl\{ \sigma\vert\nu_1(\sigma)\geq\mu(\sigma
)=(n!)^{-1} \bigr\}.
\]
As $\nu_1$ has an increasing density, $A$ is an increasing event and
%
%
\begin{equation}
\llVert\nu_1-\mu\rrVert_{\mathrm{TV}}=\nu_1(A)-
\mu(A) \leq\nu_2(A)-\mu(A)=\llVert\nu_2-\mu\rrVert
_{\mathrm{TV}}.
\end{equation}
\upqed
\end{pf}

Let us first prove Proposition~\ref{censor} for a fixed sequence of updates.

%
\begin{proposition}\label{detcensor}
Let $\nu_0$ be an increasing probability on $S_N$ and $k\in\mathbb
{N} $.

Given $(x_1,\dots, x_k)\in\{1,\dots,N-1\}^k$ (repetitions are allowed)
and $j\in\{1,\dots,\allowbreak k\}$.
Let $\nu_1$ denote the measure obtained by performing successive
updates at site $x_1,\dots, x_k$ and
$\nu_2$ denote the measure being obtained by performing the same
sequence of updates, omitting the one at $x_j$
(i.e., $x_1,\dots,x_{j-1},\break x_{j+1},x_{j+2},\dots, x_k)$.

Then
\[
\llVert\nu_1-\mu\rrVert_{\mathrm{TV}}\geq\llVert
\nu_2-\mu\rrVert_{\mathrm{TV}}.
\]

The result remains valid if several updates are omitted instead of one.
\end{proposition}

\begin{pf}
Without loss of generality
we can consider that $j=1$ as the law obtained after the performing
$j-1$ first update has an increasing density; cf. Lemma~\ref{conservinc}.
Let $\nu'_0$ be the measure obtained after updating $x_1$. From
Lemma~\ref{conservinc}, we have
\[
\nu'_0\preceq\nu_0.
\]

As monotonicity is preserved by the updates at $(x_2,\dots,x_k)$ (cf.
Lemma~\ref{microdil}), we have
\[
\nu_2\preceq\nu_1.
\]

Furthermore from Lemma~\ref{conservinc},
both have increasing densities, and one can conclude using Lemma~\ref
{increden}.

The case of several omissions can be proved using a straightforward
induction.
\end{pf}

\begin{pf*}{Proof of the censoring inequality}
In our dynamics, at time $t$, the set of updates that have been
performed is random and is given by the clock process $\mathcal{T} $ restricted
to $[0,t]$
(recall the graphical construction of Section~\ref{graphix}) so that
Proposition~\ref{detcensor} cannot apply directly. However, for a
fixed realization of $\mathcal{T} $, we can apply Proposition~\ref{detcensor}
conditioned to $\mathcal{T} $.

Set
\[
p_t^{\mathcal{T} }:=\mathbb{P} \bigl[ \sigma^\nu_t
\in\cdot\vert\mathcal{T} \bigr]
\]
to be the law of $\sigma$ obtained after doing the updates
corresponding to $\mathcal{T} $,
and
\[
p_t^{\mathcal{T} ,\mathcal{C}}:= \mathbb{P} ^{\mathcal{C}} \bigl[
\sigma^\nu_t\in\cdot\vert\mathcal{T} \bigr]
\]
the one obtained after performing only the updates allowed the
censoring scheme.
Both probability measures are increasing, and from Proposition~\ref{detcensor},
\[
p^{\mathcal{T} }\succeq p^{\mathcal{T} ,\mathcal{C}}.
\]

These two properties are conserved when averaging with respect to
$\mathcal{T}
$ so that
\[
P_t^{\nu}\succeq P_t^{\nu,\mathcal{C}},
\]
and Lemma~\ref{increden} allows us to conclude.
\end{pf*}

\subsection{Proof of Proposition \texorpdfstring{\protect\ref{projecmon}}{3.8}}\label{aprojecmon}

First of all, we notice that items (iii) and (iv) can be obtained
simply by integrating the increasing function
$\nu/\mu$ against inequalities \eqref{nastro} and \eqref{dominus}.

We will only prove \eqref{nastro}. The reader can check then that the
proof also works if the grid
$(x_i,x_j)_{i,j=1}^{K-1}$ is replaced by an asymmetric one
$(x_i,y_j)_{i,j=1}^{K-1}$ and that in any case the particular values of
the $x_i$ do not play any role.
Thus \eqref{dominus} simply corresponds to the case $K=2$.

We prove the result in two steps.
First, we prove that if $\widehat\sigma^1,\widehat\sigma^2\in
\widehat S_N$ and
\mbox{$\widehat\sigma^1\geq\widehat\sigma^2$}, then
%
%
\begin{equation}
\label{nastros} \mu\bigl(\cdot\vert\widehat\sigma=\widehat
\sigma^1 \bigr)\succeq\mu\bigl(\cdot\vert\widehat\sigma
=\widehat
\sigma^2 \bigr).
\end{equation}
Then we show that if $\bar\sigma^1, \bar\sigma^2\in\bar S_N$ and
$\bar\sigma^1\geq\bar\sigma^2$, we have
%
%
\begin{equation}
\label{ledernier} \widehat\mu\bigl(\cdot\vert\bar\xi=\bar
\sigma^1
\bigr)\succeq\widehat\mu\bigl(\cdot\vert\bar\xi=\bar\sigma^2
\bigr),
\end{equation}
where, in the above equation $\bar\xi$ denotes projection of $\xi\in
\widehat S_N$ on $\bar S_N$.

Before going to the core of the proof, let us show that the combination
of \eqref{nastros} and \eqref{ledernier}
yields \eqref{nastro}. Let $f$ be an increasing function on $S_N$, and
we define $\widehat f$ on $\widehat S_N$ by
%
%
\begin{equation}
\label{hatf} \widehat f(\xi)= \mu\bigl(f(\sigma) \vert\widehat
\sigma=\xi
\bigr).
\end{equation}
Relation \eqref{nastros} implies that $\widehat f$ is an increasing
function on $\widehat S_N$.
Finally, if $\bar\sigma_1\geq\bar\sigma_2$,
%
%
\begin{eqnarray}
\mu\bigl( f(\sigma) \vert\bar\sigma=\bar\sigma_1 \bigr)&=&
\widehat\mu\bigl( \widehat f(\xi) \vert\bar\xi=\bar\sigma_1
\bigr)
\nonumber
\\[-8pt]
\\[-8pt]
&\geq&\widehat\mu\bigl( \widehat f(\xi) \vert\bar\xi
=\bar
\sigma_2 \bigr)= \mu\bigl( f(\sigma) \vert\bar\sigma=\bar
\sigma_1 \bigr),
\nonumber
\end{eqnarray}
where the inequality uses \eqref{ledernier} and the fact that
$\widehat f$
is increasing.
This is enough to conclude by using Lemma~\ref{transport}.

Let us prove \eqref{nastros}. First, we notice that the information
given by $\widehat\sigma$ is exactly the value of the sets
\[
\sigma^{-1} \bigl(\{x_{i-1}+1,\dots,x_{i}\}
\bigr),\qquad i\in\{1,\dots,K\}.
\]
For each $i$, this set is given by
%
%
\begin{eqnarray}
&& \bigl\{ x\in\{0,\dots, N\} \vert
\nonumber
\\[-8pt]
\\[-8pt]
&&\qquad\widehat\sigma(x,i+1)-\widehat\sigma(x-1,i+1)-\widehat
\sigma(x,i)+
\widehat\sigma(x-1,i)>0 \bigr\} .
\nonumber
\end{eqnarray}
The missing information is in what order the cards, whose labels belong
to $\{x_{i-1}+1 ,\dots,x_{i}\}$, appear in the pack.
Hence for each $\xi\in\widehat S^N$, there is a natural bijection
%
%
\begin{eqnarray}
\label{benarfa} %
\bigotimes_{i=1}^K
S_{\Delta x_{i}}&\to&\{ \sigma\in S_N \vert\widehat\sigma=\xi\},
\nonumber
\\[-8pt]
\\[-8pt]
(\sigma_1,\dots,\sigma_K) &\mapsto&\sigma^{(\sigma_1,\dots
,\sigma
_K)}_\xi,
\nonumber
\end{eqnarray}
where ${\Delta x_{i}}:=x_{i}-x_{i-1}$.
The permutation $\sigma^{(\sigma_1,\dots,\sigma_K)}_\xi$, is
defined to be the one\vspace*{2pt} in $\{ \sigma\in S_N \vert \widehat\sigma=\xi
\}$
for which, for all $i\in\{1,\dots,K\}$,
the card with the label $\{x_{i-1},\dots,x_{i}\}$ appears in the deck
in the order specified by $\sigma_i$,
%
%
\begin{eqnarray}
\label{samorder} &&\forall a, b \in\{x_{i-1}+1,\dots, x_i\}
\nonumber
\\[-8pt]
\\[-8pt]
\eqntext{\displaystyle\sigma^{-1}(a)\leq\sigma^{-1}(b)
\quad\Leftrightarrow\quad\sigma_i^{-1}(a-x_{i-1})
\leq\sigma_i^{-1}(b-x_{i-1}).}
\end{eqnarray}
The reader can check that given $(\sigma_1,\dots,\sigma_K)$ and $\xi
$, there is a unique permutation satisfying $\widehat\sigma=\xi$ and
\eqref{samorder}.

Mapping \eqref{benarfa} has the following expression in terms on
surfaces: for all $y\in\{x_{i-1},\dots,x_i\}$
%
%
\begin{eqnarray}
\qquad\widetilde\sigma_{\xi}^{(\sigma^1,\dots,\sigma^K)}(x,y)&=&
\frac{y-x_{i-1}}{\Delta x_i}
\xi(x,i)+ \frac{x_{i}-y}{\Delta x_i}\xi(x,i-1)
\nonumber
\\[-8pt]
\\[-8pt]
&&{}+\widetilde\sigma_i \biggl(\xi(x,i)-\xi(x,i-1)+
\frac{ x\Delta x_i}N ,y-x_{i-1} \biggr).
\nonumber
\end{eqnarray}
If $\xi\geq\xi'$ are two admissible semi-skeletons, it is tedious
but straightforward to check with the above expression that
for any $(\sigma^1,\dots,\sigma^K)$,
\[
\widetilde\sigma_{\xi}^{(\sigma^1,\dots,\sigma^K)}\geq\widetilde
\sigma
_{\xi'}^{(\sigma^1,\dots,\sigma^K)}.
\]
Hence the uniform measure on $\prod_{i=1}^K S_{\Delta x_{i}}$ induces a
monotonous coupling proving \eqref{nastros}.

Let us now prove \eqref{ledernier}.
Given $\bar\sigma_1\geq\bar\sigma_2$, we consider $\widehat S^1$ and
$\widehat S^2$ defined by
\[
\widehat S^i:= \{ \xi\in\widehat S_N \vert\bar\xi=\bar
\sigma_i\}.
\]
Let us prove that each $\widehat S^i$ possesses a maximal element $\xi
^i_{\max}$ and that they satisfy
%
%
\begin{equation}
\label{lesordere} \xi^1_{\max}\geq\xi^2_{\max}.
\end{equation}
%

%
\begin{figure}[b]

\includegraphics{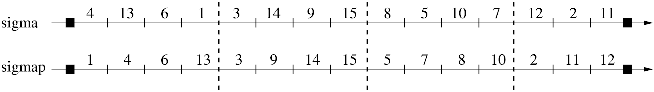}

\caption{The transformation $\sigma\to\sigma'$,
obtained by sorting the cards in each interval ($N=15$, $K=4$).}
\label{sorting}
\end{figure}

To obtain the maximal element of $\widehat S^1$, we start by taking
$\sigma
\in S_N$ such that \mbox{$\widehat\sigma\in\widehat S^1$}.
Then we consider $\sigma'$, the permutation obtained by sorting the
elements in each interval $\{x_{i-1}+1,\dots, x_i\}$, for all $i\in\{
1,\dots, K\}$ (see Figure~\ref{sorting}),
that is, the unique permutation which satisfies
%
%
\begin{equation}
\label{defsprim} \hspace*{4pt}\qquad\forall i\in\{1,\dots, K\},\qquad\sigma'
\bigl( \{x_{i-1}+1,\dots, x_i\} \bigr)=\sigma\bigl(
\{x_{i-1}+1,\dots, x_i\} \bigr),
\end{equation}
and
%
%
\begin{eqnarray}
&&\forall i\in\{1,\dots, K\}, \forall(y,z)\in\{x_{i-1}+1,\dots,
x_i\},
\nonumber
\\[-8pt]
\\[-8pt]
\eqntext{\displaystyle y\leq z \quad\Rightarrow\quad\sigma'(y)
\leq\sigma'(z).}
\end{eqnarray}

Then for all $i\in\{1,\dots, K\}$, $j\in\{0,\dots, K\}$ and $x \in
\{ x_{i-1},\dots, x_i\}$,
we have
%
%
\begin{equation}
\label{dooz} \quad\qquad\widehat\sigma' (x,j):= \min\biggl(
\frac{N-x_j}{N}(x-x_{i-1})+\bar\sigma(i-1,j), \frac{x_j}{N}(x_i-x)+
\bar\sigma(i,j) \biggr).
\end{equation}
This guarantees that $\widehat\sigma'$ is maximal in $\widehat S^1$ (and
hence the existence of a maximal element).
The expression of the maximum implies \eqref{lesordere}.

Let $\xi^1_t$ and $\xi^2_t$ be the Markov chain on $\widehat S^i$
constructed with the graphical construction from $U$ and $\mathcal{T} $
but ignoring the update at $x_i$, $i=1,\ldots,K-1$, starting from $\xi
^1_{\max}$ and $\xi^2_{\max}$, respectively.\vspace*{2pt}
This censoring corresponds to canceling updates that take $\xi^i_t$
out of $\widehat S^i$.

The Markov chains $\xi^1_t$ and $\xi^2_t$ are irreducible:
indeed given $\xi\in\widehat S^1$, we can find $\sigma$ such that
$\widehat
\sigma=\xi$.
Then from $\sigma$ it is possible to construct a path of transition
leading to $\sigma'$ [the maximal element described in \eqref{dooz}]
that does not use
any of the $\tau_{x_i}$, and projecting this path with the
semi-skeleton projection gives us a path of allowed transition
from $\xi$ to $\xi^1_{\max}$.

As the\vspace*{2pt} $\xi^i_t$ are reflected Markov chains, their respective
equilibrium measures are
$\widehat\mu(\cdot\vert \bar\xi= \bar\sigma^i)$, $i=1,2$ (which is
the uniform measure on $\widehat S^i$).
The ordering of the initial condition and the order preservation
induced by the graphical construction (see the proof of
Proposition~\ref{orderpreserv}) implies
\[
\forall t\geq0,\qquad\xi^1_t\geq\xi^2_t.
\]

Having this monotone coupling between the two processes, we use
Lemma~\ref{limitlema} to conclude.

\subsection{Proof of Lemma \texorpdfstring{\protect\ref{boudk}}{4.1}}\label{aboudk}

For any fixed $y$,
the solution of \eqref{disheat} can be computed by Fourier
decomposition on the basis of
eigenfunctions $(u_i)_{i=1}^{N-1}$ of $\Delta_x$ given by
%
%
\begin{equation}
\label{eigenu} u_i\dvtx x\mapsto\sqrt{\frac{2} N}\sin
\biggl( \frac{x i\pi}{N} \biggr).
\end{equation}
The eigenvalue associated to $u_i$ is $-\lambda_{N,i}$ where
%
%
\begin{equation}
\label{lambdani} \lambda_{N,i}:=2 \biggl(1-\cos\biggl(
\frac{i\pi}{N} \biggr) \biggr).
\end{equation}
Hence
%
%
\begin{equation}
\label{fourier} f(x,y,t)= \frac{2}{N}\sum_{i=1}^{N-1}
a_i \bigl(\widetilde\sigma_0(\cdot,y)
\bigr)e^{-\lambda_{N,i} t} \sin\biggl( \frac{x i\pi}{L} \biggr),
\end{equation}
where the Fourier coefficient $a_i$ is given by
%
%
\begin{equation}
a_i \bigl(\widetilde\sigma_0(\cdot,y) \bigr):=\sum
_{x=1}^{N-1}\widetilde\sigma
_0(x,y)\sin\biggl( \frac{x i\pi}{N} \biggr).
\end{equation}

We have, by definition of $\widetilde\sigma$,
\[
\bigl\llvert\widetilde\sigma_0(x,y) \bigr\rrvert\leq\min(y,N-y)
\qquad\forall x\in\{ 0,\dots,N\}
\]
(in the remainder of the proof we assume $y\leq N/2$ for simplicity),
and hence the Fourier coefficients satisfy
\[
\llvert a_i\rrvert\leq yN\qquad\forall i\in\{1,\dots,N-1\}.
\]
Moreover, the reader can check that $\lambda_{i,N}\geq i \lambda
_{N}$, for all
$i\in\{1,\dots,N-1\}$, and hence
we deduce from \eqref{fourier} that
%
%
\begin{equation}
\bigl\llvert f(x,t) \bigr\rrvert\leq2y \sum_{i=1}^{N-1}
e^{-i\lambda_{N} t}= \frac{2y
e^{-\lambda
_{N} t}}{1-e^{-\lambda_{N} t}}.
\end{equation}
When $e^{-\lambda_{N} t}\leq1/2$, this implies \eqref{group}, and when
$e^{-\lambda_{N} t}\geq1/2$
we have that $\llvert f(x,t)\rrvert \leq y$ because $\llvert
\widetilde\sigma(x,y,t)\rrvert \leq y$,
and hence \eqref{group}
is also valid in this case too.

For \eqref{group3},
note that when $y\leq N/2$,
%
%
\begin{eqnarray}
\min\biggl( x \biggl( 1- \frac{y} N \biggr), (N-x) \frac{y} N
\biggr)&\geq& \min\biggl( x \frac{y} N, (N-x) \frac{y} N \biggr)
\nonumber
\\
[-8pt]
\\[-8pt]
&=& \frac{y} \pi\min\biggl( \frac{x\pi}{N}, \pi-
\frac{x\pi}{N} \biggr).
\nonumber
\end{eqnarray}
Hence using the identity $\sin u\leq\min(u, \pi-u)$ valid for $u\in
[0,\pi]$, we obtain
%
%
\begin{equation}
\label{qpro} \forall x\in\{1,\dots,N-1\},\qquad\widetilde
\sigma_0(x,y) \geq\frac
{y}{\pi}\sin\biggl(
\frac{x \pi}{N} \biggr).
\end{equation}
Because of monotonicity of the solution of the heat equation in the
initial condition, one can deduce
\eqref{group3} by considering the solution of \eqref{disheat} at time
$t$ for both sides of \eqref{qpro}.

\subsection{Proof of Lemma \texorpdfstring{\protect\ref{cromostic}}{5.8}}\label{acromostic}

Inequality \eqref{letrucmoche} is obtained by integrating $\nu/\mu$
against the inequality \eqref{ared}.
We prove first \eqref{ared} for the conditioned law of the
semi-skeleton $\widehat\sigma$ [recall \eqref{semiskel}]
%
%
\begin{equation}
\label{deschapeaux} \widehat\mu( \cdot\vert c).
\end{equation}

Starting from the identity, we define $\sigma^1_t$ and $\sigma^2_t$
to be two AT shuffle dynamics for which the
transitions going out of $\mathcal{A}$ (resp., out of $\mathcal{B}
^c$) are canceled.
We couple the two dynamics using the graphical construction. Note that
the two Markov chains we have introduced are irreducible and
hence that their respective equilibrium measures are $\widehat\mu(
\cdot
\vert \mathcal{A})$ and
$\widehat\mu(\cdot\vert \mathcal{B} ^c)$.
We want to show that $\widehat\sigma^1_t\geq\widehat\sigma^2_t$
for all
times and then deduce \eqref{deschapeaux} from Lemma~\ref{limitlema}.

What there is to show is that the order is preserved each time that an
update is performed for either dynamics.
When an update is not censored by either dynamics, it preserves the
order as a consequence of the proof of Proposition~\ref{orderpreserv}.
Note also that as both events $\mathcal{A}$ and $\mathcal{B} ^c$ are
increasing; only
updates going down might be canceled.

It follows that the only thing to check is that if a down update is
censored for $\widehat\sigma^2$ but not for $\widehat\sigma^1$,
it cannot break monotonicity.
Let $z_{\min}(i,j)$ denote the smallest\vspace*{2pt} admissible value of $\bar
\sigma(i,j)$ which is larger or equal to $A\sqrt{k}$.
If the transition at $x_{i}$
is canceled for $\widehat\sigma^2$, say at time at time $t$, it
implies that
\[
\forall j\in\{1,\dots, K-1\},\qquad\widehat\sigma_{t}^2(x_i,j)
\leq z_{\min}(i,j),
\]
and if not, a single jump would not be sufficient to exit $\mathcal{B} ^c$.
By the definition of $\mathcal A$,
\[
\forall j\in\{1,\dots, K-1\},\qquad\widehat\sigma^1_{t}(x_i,j)
\geq z_{\min}(i,j).
\]

As the $\sigma(x,y)$, $x\ne x_{i}$ are not affected by the transition,
we have
$\widehat\sigma^1_{t}\geq\widehat\sigma^2_{t}$ provided $\widehat
\sigma
^1_{t^-}\geq\widehat\sigma^2_{t^-}$. This completes the proof of
\eqref
{deschapeaux}.

To prove the same stochastic domination with $\widehat\mu$ replaced by
$\mu$, we recall (from the proof of Proposition~\ref{projecmon}) that
if $f$ is increasing,
$\widehat f$ is increasing, defined by \eqref{hatf}, and thus for all
increasing $f$s,
%
%
\begin{equation}
\mu( f \vert\mathcal{A}) =\widehat\mu\bigl( \widehat f (\widehat
\sigma) \vert
\mathcal{A} \bigr)\geq\widehat\mu\bigl( \widehat f (\widehat
\sigma) \vert
\mathcal{B} ^c \bigr)=\mu\bigl( f ( \sigma) \vert\mathcal{B}
^c \bigr),
\end{equation}
which, according to Lemma~\ref{transport}, proves stochastic domination.

\section{Back to the original card shuffle}\label{appB}

As we wish to give the full answer to the question given in the
\hyperref[sec1]{Introduction}, we explain in this appendix how to
obtain the result in
discrete time.

We can use the tools we have developed in Section~\ref{monotool} to
compare the mixing time in discrete and continuous times.
We consider $(\sigma_n)_{n\geq0}$ the trajectory discrete Markov
chains described in the \hyperref[sec1]{Introduction}, and which can be described as
follows: we start from the identity
at each step, we chose a $x$ at random in $\{1,\dots,N-1\}$ and
perform an update at $x$.
Let $\mathbf{P} _n$ denote the law of $\sigma_n$.

The continuous time chain can be described in the following manner.
We consider $\mathcal{T} $ a Poisson point process with rate $2(N-1)$
($\mathcal{T}
_0=0$ and $\mathcal{T} _{n}-\mathcal{T} _{n-1}$, $n \geq1$ are
i.i.d. exponential
variables with mean $1/[2(N-1)]$)
which is independent, and set
%
%
\begin{equation}
\forall n\geq0\ \forall t \in[\mathcal{T} _n,\mathcal{T}
_{n+1}), \qquad\sigma'_t=
\sigma_n.
\end{equation}
Then $\sigma'_t$ is the continuous Markov chain with generator \eqref{defgen}.

Hence
%
%
\begin{equation}
\label{decomposture} P_t=\sum_{k=0}^\infty
\frac{(2t(N-1))^n e^{-2(N-1)t}}{k!} \mathbf{P} _n.
\end{equation}

From this we can prove the following result.

%
\begin{proposition}\label{contetdis}
We have for all $t$ and $n$,
%
%
\begin{equation}
\llVert\mathbf{P} _n-\mu\rrVert\leq\frac{\llVert P_t-\mu
\rrVert }{\sum_{k=0}^n ((2t(N-1))^k e^{-2(N-1)t})/k!},
\end{equation}
and
%
%
\begin{equation}
\llVert\mathbf{P} _n-\mu\rrVert\geq\frac{\llVert P_t-\mu
\rrVert -\sum_{k=0}^{n-1}
((2t(N-1))^k e^{-2(N-1)t})/k!}{\sum_{k=n}^\infty((2t(N-1))^k
e^{-2(N-1)t})/k!}.
\end{equation}
\end{proposition}

\begin{pf}
Let us fix $t>0$ and $n\in\mathbb{N} $.
From Proposition~\ref{proppres} (which proof can easily adapt for
discrete time), note also that $\mathbf{P} _n$ is an increasing probability
for all $n$ (as is $P_t$) so that the events
%
%
\begin{eqnarray}
A_1&:=& \bigl\{\sigma\vert\mathbf{P} _n(
\sigma)\geq\mu(\sigma) \bigr\},
\nonumber
\\[-8pt]
\\[-8pt]
A_2&:=& \bigl\{\sigma\vert P_t(\sigma)\geq\mu(\sigma)
\bigr\},
\nonumber
\end{eqnarray}
are increasing events.
Recall that from the definition of the total variation distance,
\begin{eqnarray*}
\mathbf{P} _n(A_1)-\mu(A_1)&=&\llVert
\mathbf{P} _n-\mu\rrVert_{\mathrm{TV}}\quad\mbox{and}
\\
P_t(A_2)-\mu(A_2)&=&\llVert
P_t- \mu\rrVert_{\mathrm{TV}}.
\end{eqnarray*}

Now from Lemma~\ref{conservinc} (plus an average over the coordinate
which is updated), for any increasing event $A$,
$(\mathbf{P} _k(A))_{k\geq0}$ is a nonincreasing sequence tending to
$\mu
(A)$. Hence we have
%
%
\begin{eqnarray}
\llVert P_t-\mu\rrVert_{\mathrm{TV}}& \geq&
\bigl(P_t(A_1)-\mu(A_1) \bigr)
\nonumber
\\
&\stackrel{{\fontsize{8.36}{10}\selectfont{\eqref{decomposture}}}} {=}& \sum_{k=0}^\infty
\frac
{(2t(N-1))^n e^{-2(N-1)t}}{k!} \bigl(\mathbf{P} _n(A_1)-
\mu(A_1) \bigr)
\nonumber
\\
&\geq& \Biggl( \sum_{k=0}^{n}
\frac{(2t(N-1))^k e^{-2(N-1)t}}{k!} \Biggr) \bigl(\mathbf{P} _n(A_1)-
\mu(A_1) \bigr)
\\
& &{}+ \sum_{k=n+1}^{\infty}
\frac{(2t(N-1))^k e^{-2(N-1)t}}{k!} \bigl(\mathbf{P} _{k}(A_1)-
\mu(A_1) \bigr)
\nonumber
\\
&\geq& \Biggl( \sum_{k=0}^{n}
\frac{(2t(N-1))^k e^{-2(N-1)t}}{k!} \Biggr) \bigl\llVert\mathbf{P}
_n(A)-\mu\bigr
\rrVert_{\mathrm{TV}},
\nonumber
\end{eqnarray}
and
%
%
\begin{eqnarray}
\llVert P_t-\mu\rrVert_{\mathrm{TV}}&=&
\bigl(P_t(A_2)-\mu(A_2) \bigr)
\nonumber
\\
&\stackrel{{\fontsize{8.36}{10}\selectfont{\eqref{decomposture}}}} {=}& \sum_{k=0}^\infty
\frac
{(2t(N-1))^n e^{-2(N-1)t}}{k!} \mathbf{P} _n(A_2)-
\mu(A_2)
\nonumber
\\
&\leq&\sum_{k=0}^{n-1} \frac{(2t(N-1))^k e^{-2(N-1)t}}{k!}
\bigl(\mathbf{P} _k(A_2)-\mu(A_2) \bigr)
\\
&&{} + \Biggl( \sum_{k=n}^{\infty
}
\frac{(2t(N-1))^k e^{-2(N-1)t}}{k!} \Biggr) \bigl(\mathbf{P} _n(A_2)-
\mu(A_2) \bigr)
\nonumber
\\
& \leq& \Biggl(\sum_{k=0}^{n-1}
\frac{(2t(N-1))^k
e^{-2(N-1)t}}{k!} \Biggr)
\nonumber
\\
&&{} + \Biggl( \sum_{k=0}^{n}
\frac{(2t(N-1))^k e^{-2(N-1)t}}{k!} \Biggr) \bigl\llVert\mathbf{P}
_n(A)-\mu\bigr
\rrVert_{\mathrm{TV}},
\nonumber
\end{eqnarray}
which completes the proof.
\end{pf}

Now if we set
\[
\bTm^N(\varepsilon):=\inf\bigl\{ n \vert\llVert\mathbf{P}
_n-\mu\rrVert_{\mathrm{TV}}\leq\varepsilon\bigr\},
\]
Theorem~\ref{mainres} is equivalent to the following result.

%
\begin{theorem}\label{mainres2}
For the adjacent transposition shuffle, we have for every $\varepsilon
\in(0,1)$,
%
%
\begin{equation}
\lim_{N\to\infty} \frac{\pi^2\bTm^N(\varepsilon)}{N^3\log N}=1.
\end{equation}
\end{theorem}

\begin{pf}
We use the previous proposition for $t=\frac{n\pm n^{1/3}}{2(N-1)}$,
and we have
%
%
\begin{eqnarray}
\qquad\llVert P_{(n+n^{1/3})/(2(N-1))}-\mu\rrVert_{\mathrm{TV}}
+o(1) &\leq&\llVert
\mathbf{P}_n-\mu\rrVert
\nonumber
\\[-8pt]
\\[-8pt]
&\leq&\llVert P_{(n-n^{1/3})/(2(N-1))}-\mu\rrVert_{\mathrm{TV}}+o(1).
\nonumber
\end{eqnarray}
It is then easy to conclude.
\end{pf}
\end{appendix}

%
\section*{Acknowledgment}
The author is grateful to J. Lehec for
enlightening discussions.


%
%

\printaddresses
\end{document}